\documentclass{birkjour}
\usepackage{amsfonts,amssymb,amsmath,mathtools,hyperref,
upgreek,dutchcal,}
\usepackage[english]{babel}
\usepackage{bigints}
\usepackage{amssymb}
\usepackage{pdfsync}

\usepackage{svg}
\usepackage[normalem]{ulem}
\usepackage{cancel}

\usepackage[all,error]{onlyamsmath}

\usepackage{amssymb,amsfonts,amsmath,
amsrefs,enumerate,color,url,hyperref,
mathbbol,dsfont}
\usepackage{mathrsfs}
\usepackage{tikz}

\usepackage[T1]{fontenc}
\usepackage[cp1250]{inputenc}

\usepackage[strict=true]{csquotes}

\newcommand{\brown}{B}

\numberwithin{equation}{section}

\usepackage{relsize,bigints}

\usepackage{graphicx,tikz}

\usepackage{color}

\newtheorem{thm}{Theorem}[section]
\newtheorem{lemma}[thm]{Lemma}
\newtheorem{proposition}[thm]{Proposition}
\newtheorem{corollary}[thm]{Corollary}

\newtheorem{remark}[thm]{Remark}

\newcommand{\wstop}{B^{\textnormal{stop}}}
\newcommand{\wkill}{B^{\textnormal{kill}}}
\newcommand{\wrefl}{\brown^{\textnormal{ref}}}
\newcommand{\refl}{^{\textnormal{ref}}}
\newcommand{\wmin}{B^{\textnormal{min}}}

\newcommand{\eps}{\varepsilon}

\DeclareMathAlphabet{\mathpzc}{OT1}{pzc}{m}{it}

\newcommand{\ka}{{\mathpzc k}}
\newcommand{\kad}{{\mathpzc K}}
\newcommand{\kal}{\mathpzc L}
\newcommand{\pol}{\tfrac 12}

\newcommand{\jcgi}[1]{\left ( #1 \right )_{i\ge 0}}

\newcommand{\mquad}[1]{\qquad \text{#1} \qquad}
\newcommand{\mqquad}[1]{\quad \text{#1} \quad}

\newcommand{\mbR}{{\mathbb R}}
\newcommand{\elam}{\mathpzc l_\lam}

\newcommand{\mudelta}{\hspace{-0.1cm}\vartriangle}
\newcommand{\mdelta}{\vartriangle}
\newcommand{\mdu}{\vartriangle\hspace{-0.1cm}}
\newcommand{\zero}{\mathfrak 0}
\newcommand{\wt}{\widetilde}
\newcommand{\wrla}{\wt R_\lam}

\newcommand{\cadlag}{c\`adl\`ag}

\newcommand{\e}{\mathrm e}
\newcommand{\ud }{\, \mathrm d}
\newcommand{\mud}{\mathrm d}
\newcommand{\sem}[1]{\left ( \e^{t#1} \right)_{t \ge 0}}
\newcommand{\dom}[1]{\mathcal D(#1)}
\newcommand{\grae}{\lim_{\eps \to 0+}}
\newcommand{\lam}{\lambda}
\newcommand{\rla}{R_\lam} 
 
\newcommand{\jcg}[1]{\left (#1 \right )_{n\ge 1}} 
\newcommand{\gra}{\lim_{n\to \infty}}
\newcommand{\grat}{\lim_{t\to \infty}}
\newcommand{\gras}{\lim_{s\to \infty}}
\newcommand{\grato}{\lim_{t\to 0+}}
\DeclareMathOperator\sgn{sgn}

\newcommand{\rez}[1]{\left (\lam - #1\right )^{-1}}

\newcommand{\mc}{\mathcal}
\newcommand{\coss}{\mbox{$\mathfrak C(S)$}}
\newcommand{\cosso}{\mbox{$\mathfrak C_0(S)$}}
\newcommand{\cossf}{\mbox{$\mathfrak C_{\mathpzc F}(S)$}}

\newcommand{\ie}{i.e., }
\newcommand{\czn}{\int_0^\infty}
\newcommand{\sui}{\sum_{i\in \kad}}
\newcommand{\gen}{\mathfrak G}
\newcommand{\seq}[1]{(#1)_{i\in \kad}}
\newcommand{\EE}{\mathsf E\,}

\newcommand{\rod}[1]{\mbox{$\left (#1 \right )_{t\ge 0}$}} 
\newcommand{\slam}{\sqrt{2\lam}}

\usepackage{graphicx,tikz}
\usetikzlibrary{arrows}
\definecolor{ngreen}{HTML}{006400}

\newcommand{\mum}{\mathpzc m}

\newcommand{\1}{1\!\!\,{\rm I}}
\newcommand{\ap}{\textcolor{blue}}

\newcommand{\Pb}{\mathsf{P}}
\newcommand{\E}{\mathsf{E}}
\newcommand{\R}{\mathbb R}
\newcommand{\cerp}{\mathfrak C[0,\infty]}
\newcommand{\cer}{\mathfrak  C[-\infty,\infty]}

\title[Brownian motions on star graphs]{Analytic and stochastic description of\\ Brownian motions on star graphs}

\author[A. Bobrowski]{Adam Bobrowski}
\address{Lublin University of Technology\\Department of Mathematics\\Nadbystrzycka 38A\\20-618 Lublin, Poland}
\author[A. Pilipenko]{Andrey Pilipenko}
\address{University of Geneva, 
Section de math\'ematiques\\
UNI DUFOUR\\
24, rue du G\`en\`eral Dufour\\
Case postale 64\\
1211 Geneva 4, Switzerland\\ \\
Institute of Mathematics of the Ukrainian National Academy of Sciences\\
 Tereschenkivska st. 3, Kiev-4, 01601, Ukraine}

\begin{document}

\hfill Version of \today

\begin{abstract}{We provide a detailed description of all possible Feller processes on infinite} star graphs {with finite number of edges, processes that while away from the graph's center behave like a one-dimensional Brownian motion. The description can be seen as a continuation of the seminal paper \cite{itop} by It\^o and McKean (devoted to Brownian motions on the half-line), recast from the perspective of the theory of multi-armed bandits.} \end{abstract}



\maketitle

\section{Introduction}

\subsection{Boundary conditions of Feller and Wentzell}\label{sec:bco}
Analytic description of boundary conditions for diffusions is owed to the independent work of W. Feller \cites{fellera3,fellerek} and A.D. Wentzell (sometimes also spelled Ventzel') \cite{wentzellboundary}. In the case of processes with values in the positive half-line the gist of their findings is as follows. 

 Let $[0,\infty]$ be the compactified positive half-line, and let $\cerp$ be the space of continuous functions on 
this half-line; $\cerp$ can be equivalently characterized as the space of continuous functions  $f\colon [0,\infty)\to \R$ such that a~finite limit  $\lim_{x\to \infty} f(x)$ exists.  Suppose that $\gen$ is a Feller generator in $\cerp$ such that its domain is contained in the set of functions that are  differentiable at $x=0$. Then 
there are non-negative constants $\alpha, \beta,$ and $\gamma$, and a (possibly infinite) positive Borel measure $\mum$ on $(0,\infty]$ such that 
$\int_{(0,\infty]} (1\wedge x)\, \mum (\mud x)<\infty$ and
\begin{equation*}\alpha \gen f(0) - \beta f'(0) + \gamma f(0) = \int
[f(x) - f(0)]\, \mum (\mud x), \qquad f \in \dom{\gen }, \end{equation*}	
where $\alpha + \beta >0$ or $\mum (0,\infty] =\infty$.

The case of the one-di\-men\-sional Laplace operator 
defined by  \[ \Delta f = \tfrac 12 f'', \qquad   f \in \dom{\Delta}\coloneqq \mathfrak C^2[0,\infty] \]
is by far the most interesting and important.
Here, $\mathfrak C^2[0,\infty]$ is the subspace of $\cerp$ composed of twice continuously differentiable $f\colon [0,\infty) \to \R$ 
(at $x=0$ we consider right-hand derivatives; alternatively, $f'(0)$ and $f''(0)$ are 
defined as $\lim_{x\to 0+}f'(x)$ and  $\lim_{x\to 0+}f''(x)$, respectively)
such that $f''$ belongs to $\cerp$, so that in particular the limit $\lim_{x\to \infty} f''(x)$ exists and is finite. In this case, Feller and Wentzell's work can be summarized by stating that if a Feller generator $\gen$ is a restriction of  $\Delta$, then  there are non-negative constants $\alpha, \beta,$ and $\gamma$, and a (possibly infinite) positive Borel measure $\mum$ on $(0,\infty)$ such that 
$\czn (1\wedge x)\, \mum (\mud x)<\infty$ and
\begin{equation}\label{int:1} \tfrac \alpha 2 f''(0) - \beta f'(0) + \gamma f(0) = \int
[f(x) - f(0)]\, \mum (\mud x), \qquad f \in \dom{\gen }; \end{equation}	
again, at least one of conditions $\alpha +\beta >0$ and $\mum (0,\infty] =\infty$ holds. In fact, since the point at infinity is added merely for convenience, so that we can work in a compact space, and a Brownian motion cannot reach this point in  a finite time anyway, we can exclude the possibility of $\mum$ having a mass there and can thus assume that $\mum$ is a Borel measure on $(0,\infty)$. Conversely, any restriction of $\Delta$ to the domain of the form visible in \eqref{int:1} (with $ \mum$   having no mass at $\infty$) is a Feller generator (see e.g. \cite{blum}).

A number of remarks are here in order. First of all, it has been amply clear from the very beginning that, whereas the operator $\Delta$ describes the behavior of the related process \emph{inside} the half-line, condition \eqref{int:1} 
determines its behavior at the boundary, that is, at $x=0$. Secondly, as stressed by It\^o and McKean \cite{ito}*{p. 187},  before Feller and Wentzell, it had not been noticed that conditions involving $\gen f(0)$ were possible. Feller himself, alluding to this new term, acknowledges that his inspirations were in part coming from genetics and writes  in \cite{fellera3}*{p. 474}: ``Various probability considerations and a long outstanding problem in genetics made it increasingly clear that in addition to the classical types of boundary conditions there exist some of a new type. Physically speaking, there exist diffusion processes where a particle can be absorbed, and stay for a finite time, after which it penetrates slowly back into the interior.''

\subsection{Stochastic interpretation}\label{sec:si}
As already mentioned, the boundary condition \eqref{int:1} characterizes behavior of the related process at the boundary $x=0$. Several examples, studied throughout the literature, can be singled out here.

\begin{itemize} 

\item [1. ] Brownian motion \emph{stopped} at $0$ (or: \emph{absorbed} at $0$, see e.g. \cite{knight}*{p. 41 and p. 61}). In other words, if $\brown  =\{\brown (t), t\ge 0\}$ is a standard Brownian motion starting at $\brown (0)=x\ge 0$, and    $\sigma\coloneqq \inf\{ t\geq 0\colon \brown (t)=0\}$ is the moment when $\brown $ reaches $0$ for the first time, then $\wstop (t) \coloneqq \brown (t\wedge \sigma), t \ge 0$. This process is characterized by the boundary condition 
\[ f''(0)=0 \]
corresponding to $\beta =\gamma=0$, $\mum =0$ and $\alpha >0$ in \eqref{int:1}.

\item [2. ] The \emph{minimal} Brownian motion. This is the process which is undefined after it reaches $0$ for the first time: $\wmin (t) = \brown (t)$ whenever $t < \sigma$ and undefined for $t \ge \sigma$.  This corresponds to the boundary condition 
\[ f(0)=0 \]
a particular case of \eqref{int:1} with $\alpha=\beta =0$, $\mum =0$ and $\gamma >0$. We stress that, strictly speaking, this is not a process in $[0,\infty]$ but rather in $(0,\infty]$, as $0$ does not belong to the state-space. This is reflected in the fact that the boundary condition presented above does not describe a dense subset of $\cerp$, and thus $\Delta$ with this domain is not a Feller generator in $\cerp$. However, it is a Feller generator in $\mathfrak C_0(0,\infty]$, the subspace of $f\in \cerp$ such that $f(0)=0$. 

\item [3. ] Brownian motion \emph{killed} at $0$ (see e.g. \cite{knight}*{p. 71}). In this process the Brownian traveller is killed upon reaching $0$ for the first time and put at a graveyard, where it stays forever. It will be convenient to think that the graveyard is $\infty$; then $\wkill (t) = \brown (t)$ for $t< \sigma$ and $\wkill (t)=\infty$ for $t\ge \sigma$. This is an honest process, but, somewhat similarly as in Point 2., $0$ is not in its state-space, for the process leaves this point immediately to jump to $\infty$. This forces us to identify $0$ and $\infty$, and so the corresponding semigroup acts in the subspace of $f\in \cerp$ such that $f(0)=f(\infty)$. However, other authors, including Blumenthal \cite{blum}*{p. 56}, use the adjective \emph{killed} to describe the not honest process that is undefined after hitting $0$ for the first time, and thus do not distinguish between the minimal and the killed process. We will follow their example. 


\item [4. ] \emph{Reflected}, or  \emph{reflecting}, Brownian motion. This, by definition, is the process $\{|\brown (t)|, t \ge 0\}$. It is characterized by the boundary condition \eqref{int:1} with $\alpha=\gamma=0, \mum =0$ and $\beta >0$, that is, 
\[ f'(0)=0 .\]

\item [5. ] Brownian motion with \emph{slowly reflecting boundary} \cite{revuz}*{p. 421}, known also as \emph{sticky} Brownian motion \cite{liggett}*{p. 127} (in \cite{Engelbert} it is called \emph{sticky reflecting}) can be seen as a process lying between the reflected and stopped Brownian motions. Its boundary condition has the following form: 
\[ \alpha f''(0) = \beta f'(0);\]
for $\alpha =0$ and non-zero $\beta$ this reduces to the reflected Brownian motion; for $\beta =0$ and non-zero $\alpha $ this becomes the stopped Brownian motion.  Intuitively, this Brownian motion is partly reflected and partly stopped at $0$.  

\item [6. ] The \emph{elastic} Brownian motion (see e.g. \cite{ito}*{p. 45}), in turn, lies between the reflected and the minimal Brownian motions, and  is characterized by the following boundary condition:
\[ \beta f'(0) = \gamma f(0).\]
At the extreme case $\beta =0$ this is the minimal Brownian motion. For $\beta >0$ it is best described in terms of the Brownian motion's local time at $0$ (see also the next section). To wit, when the local time exceeds an independent, exponentially distributed  random variable with parameter $\frac\gamma\beta $, the process is left undefined, but until this happens it is indistinguishable from the reflected Brownian motion.  

\item [7. ] \emph{Elementary return} Brownian motion. In this process, discussed perhaps for the first time in \cite{fellera4}*{p. 3}, a particle reaching $x=0$ stays there for an exponential time with parameter $\alpha$ and then with probability $\delta \in [0,1]$ jumps back to a random point of $(0,\infty)$ distributed according to a Borel probability measure $\mu$ on $(0,\infty)$; with probability $1-\delta$ the process is no longer defined. The boundary condition in this case reads
\[ \tfrac \alpha 2 f''(0)+(1-\delta)f(0) = \delta \int [f -f(0)]\ud \mu. \]
\end{itemize}

As a side remark, we note here an important, if obvious, feature of exemplary processes generated by the Laplace operator $\Delta$ with boundary conditions \eqref{int:1}: all of them are shown above to be in a more or less direct way constructed from the standard Brownian motion (this thought will be made even more precise in the next section). It is engaging that a similar effect can be observed from the operator theory perspective. To explain this, recall first that the standard Brownian motion semigroup, say, $T$, is given by 
\begin{equation}\label{brownek} T(t)f(x) =\mathsf E_x f(\brown (t)), \qquad x \in \R, t \ge 0 \end{equation}
where $\mathsf E_x$ denotes expectation conditional on $\brown $ starting at $x$, and $f$ is, say, a~member of the space $\cer$, where $[-\infty,\infty]$ is a two-point compactification of the real line. This semigroup turns out to have a number of invariant subspaces that are isomorphic to $\cerp$. This allows considering 
images in $\cerp$ of restrictions of $T$ to these subspaces, and some of them turn out to describe the processes characterized by the boundary condition \eqref{int:1}; in particular, all processes corresponding to \eqref{int:1} with $\mum=0$ are of this form. See \cite{kosinusy,newfromold} for details; see also Section \ref{sec:cf}.

\subsection{Stochastic construction}\label{sec:sc}
Points 1. through 6. of Section \ref{sec:si} make it intuitively clear that coefficients $\alpha,\beta$ and $\gamma$ of \eqref{int:1} measure the intensities with which, in the related process, a particle touching $x=0$ is  
 stopped, reflected and/or removed, respectively. Moreover, point 7. suggests that the measure $\mum$ tells of the way in which the particles jump from the boundary to the interior. It is relations between these parameters that dictate in detail the behavior of the process.

However, such a description, though pleasing, is far from being  
satisfactory; there is a need for a more intrinsic, more intricate and more thorough, perhaps pathwise,  picture of the processes involved. Such a picture was provided in the beautiful and insightful paper \cite{itop} by It\^o and McKean in 1963.  To paint this picture here again, however, we first need to discuss two important notions: that of a symmetric L\'evy local time for a  Brownian motion (see Section \ref{sec:llt})
and that of a~subordinator (see Section \ref{sec:suby}). The theorem of It\^o and McKean will be presented in Section \ref{sec:ttoi}.

\subsubsection{L\'evy local time and L\'evy theorem}\label{sec:llt}
We recall that the symmetric L\'evy local time $L_0$ of a Brownian motion $\brown $, a stochastic process with values in $[0,\infty)$ in itself, is defined in one of a number of equivalent ways (see e.g. \cite{ito}*{pp. 42--45}, \cite{karatzas}*{pp. 201--217 and Ch. 6},  \cite{levybook} or  \cite{rosen}*{pp. 31--39}) as the following almost sure limit
\begin{equation}
    \label{eq:Levythm2}
L_0(t)\coloneqq \lim_{\varepsilon\to0+}\tfrac{1}{2\varepsilon}\int_0^t\1_{\{|\brown (s)|\leq \varepsilon\}} \ud s, \qquad  t\ge 0, 
\end{equation}
and it is clear from this definition that $L_0$ is at the same time the local time of $|\brown |$. A surprising yet fundamental theorem of L\'evy states that the distributions of $(\R^+)^2$-valued processes  $(|\brown (t)|, L_0(t))_{t\ge 0}$ and $(\wrefl (t), L(t))_{t\ge 0}$ coincide:  
\begin{equation}
    \label{eq:Levythm1}
(|\brown (t)|, L_0 (t))_{t\geq 0}\overset{d}{=} ( \brown \refl(t) , L(t))_{t\geq 0},
\end{equation}
where (cf. Section \ref{sec:inverses}) 
\begin{equation} L(t)\coloneqq \max_{s\in[0,t]}(-\brown (s) \vee 0) \mqquad{and}\wrefl (t)\coloneqq \brown (t) + L(t), \qquad t \ge 0. \label{deflocal} \end{equation}
In particular, $\wrefl$ is a (more manageable) version  of $|\brown |$, and $L$ is the local time for $\wrefl$. Furthermore,  
each path of $\wrefl$ is the solution to the Skorokhod problem discussed in Section \ref{sec:inverses}.

\subsubsection{Subordinators}\label{sec:suby} Subordinators are L\'evy processes that start from $0$ and have nondecreasing paths in $[0,\infty)$ \cite{bertoin}.  As noted in  \cite{bertoin}*{p. 71} `The terminology comes from the fact that when one time-changes a Markov process $M$ by an independent subordinator $U$, the resulting process $M\circ U$ is again Markov', see \cite{feller}*{Sec. X.7} for details. 

These special properties of subordinators are reflected in the fact that their characteristic exponents (the exponents in the L\'evy--Khintchine formula) have simpler form. First of all, the Gaussian coefficient of a subordinator is $0$, its drift coefficients is nonnegative, and its L\'evy measure is more regular close to $0$. The last statement means that,  as in \eqref{int:1}, $\czn (1\wedge x)\, \mum (\mud x)<\infty$. 
(As we shall see soon this affinity to \eqref{int:1} is not accidental at all.) Importantly, to describe subordinators one can work with Laplace transform instead of the Fourier transform: a subordinator $U=\{U(t), t\ge 0\}$ is characterized by 
\begin{equation}\label{int:2} \EE \e^{-\lam U(t)} = \e^{-t\Phi (\lam)}, \qquad \lam \ge 0, t \ge 0\end{equation}
where $\Phi$, termed the \emph{Laplace exponent} or \emph{cumulant}, is of the form 
\[ \Phi (\lam) = \beta + \int_0^\infty (1-\e^{-\lam x}) \, \mum (\mud x), \qquad \lam \ge 0, \]
with nonnegative drift coefficient $\beta $ and the L\'evy measure $\mum$ described earlier.

\subsubsection{The theorem of It\^o and McKean}\label{sec:ttoi}

Let $|\brown |$ be a reflected Brownian motion and let $L_0$ be its symmetric local time and $0$. Moreover, let $U$ be a subordinator, independent of $|\brown |$, satisfying at least one of the conditions \[ \beta >0 \mquad{or} \mum (0,\infty) =0.\] Such a subordinator has strictly increasing paths and thus it is possible to define these paths' generalized inverses $U^{-1}$, as presented in Section \ref{sec:inverses}. As the following result of \cite{itop} shows, the process $U \circ U^{-1}\circ L_0$ is of fundamental importance in describing  jumps of the Brownian motion related to \eqref{int:1}.

\begin{thm} Let 
    \begin{align}\label{itomckean}
        X(t)&= \begin{cases} ( U\circ U^{-1}\circ L_0 -L_0 + |\brown |) \circ A^{-1} (t), & \text{ if } t<\sigma_{\infty},\\
        \text{undefined}, & \text{ if } t\geq \sigma_{\infty},\end{cases} 
    \end{align}
    where 
    \begin{align}\nonumber
 A (t)& \coloneqq t+\alpha \theta(t), \quad t \ge 0,\nonumber\\
    \theta&\coloneqq U^{-1}\circ L_0,\nonumber \\
    \intertext{and the distribution of $\sigma_\infty$ is specified by}
   \Pb(\sigma_{\infty}>t | X)&=\e^{-\gamma \theta (t)}, \qquad t \ge 0.\label{new:cond}
    \end{align}
Then $X=\{X(t), t\ge 0\}$ is a Markov process, and its  generator is the Laplace operator $\Delta$ restricted to the domain where 
 the boundary condition \eqref{int:1} is satisfied.   \end{thm}

\subsection{The goal of the paper}

\newcommand{\gwiazda}[2]{\begin{tikzpicture}
      \node[circle,fill=blue,inner sep=0pt,minimum size=4pt] at (360:0mm) (center) {};
    \foreach \n in {1,...,#1}{
        \node [circle,fill=blue,inner sep=0pt,minimum size=0pt] at ({\n*360/#1}:2cm) (n\n) {};
        \draw (center)--(n\n);} 
    \foreach \n in {1,...,#1}{
        \node [circle,fill=blue,inner sep=0pt,minimum size=0pt] at ({\n*360/#1}:2.5cm) (n\n) {};
        \draw[dashed] (center)--(n\n);}   
        \node[below] at (-0.05,-0.15) {$\zero$};   
\end{tikzpicture}}
\newcommand{\stargraphn}[2]{\begin{tikzpicture}
  \node[circle,fill=white,inner sep=0pt,minimum size=3pt] at (360:0mm) (center) {};
    \foreach \n in {1,...,#1}{
        \node [circle,fill=black,inner sep=0pt,minimum size=2pt] at ({\n*360/#1}:#2cm) (n\n) {};
        \draw (center)--(n\n);} 
\end{tikzpicture}}

\begin{figure}
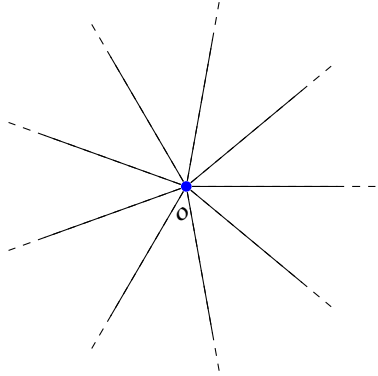

\gwiazda{9}{2} 
\caption{Infinite star graph $K_{1,\ka}$ with $\ka=9$ edges and graph center $\zero$.}\label{slg}
\end{figure} 
The goal of our paper is to provide both \emph{analytic} and \emph{stochastic} description of all possible Brownian motions on star graphs, a description that would be analogous to that we owe to Feller, Wentzell, It\^o and McKean, as detailed in Sections \ref{sec:bco}--\ref{sec:sc}.  In our use here, the adjective \emph{analytic} refers to description of generators and/or resolvents of Feller semigroups
associated with the processes of interest, and the \emph{stochastic} --- to description of their distributions and paths.  Put otherwise: the paper presents a blend of analytic and stochastic reasonings that give a thorough picture of Laplace-operator generated processes on star graphs. The main idea is that these processes can to much extend be described by similar means as those generated by the Laplace operator on the half-line. 

\emph{The analytic approach} can be summarized in the following steps. 
\begin{itemize} 
\item [1. ] In the case of $\alpha + \beta>0$, $\gamma =0$ and $\mum=0$, the boundary condition \eqref{int:1} corresponds to a mixture of reflected and stopped Brownian motions, that is, to the slowly reflecting Brownian motion a.k.a. sticky Brownian motion. A counterpart of this boundary condition for processes on star graph
corresponds to the classical Walsh's spider process and its version with sticky graph's center (see Section \ref{sec:wbm}). In both scenarios, the resolvent can be calculated by elementary means. 

\item [2. ] Introducing $\gamma >0$ makes the process of Point~1. not honest:  it is no longer defined when its local time at $x=0$ exceeds an independent exponential random variable with parameter $\gamma $. Up to that random time, the process does not differ from the one described in Point~1 (see Section~\ref{sec:tcom=0}). As in the previous case, the related resolvent can be calculated by elementary means. 

\item [3. ] If, in the set-up of Point 1., $\mum $ is no longer zero but finite, the related process, weather on the half-line or on the graph, can be seen as a~continuation of that of Point 2., and 
behaves as follows. When the local time at $x=0$ exceeds an independent exponential random variable with parameter $\delta\coloneqq \|\mum\| $ the process restarts at a random point distributed according to the probability measure $\delta^{-1}\mum$. If, additionally, $\gamma >0$, the process restarts with probability $\frac \delta{\gamma +\delta}$ and is left undefined with probability $\frac{\gamma}{\gamma + \delta}$ (see Section \ref{sec:finite}). Its resolvent can be expressed in terms of $\mum$ and the resolvent of the process from Point~2. with $\gamma $ replaced by $\delta$. 

\item [4. ] The resolvent and the semigroup related to \eqref{int:1} with $\alpha +\beta >0$ and infinite $ \mum$, can be obtained as a limit of resolvents and semigroups of the type of Point~3. as follows. We let $\delta (\eps) \coloneqq \mum (\eps, \infty)$ and $\mum_\eps $ to be $\mum $ restricted to $(\eps, \infty)$ and normalized by $\delta (\eps)$. Then, as $\eps \to 0$, the resolvents and semigroups corresponding to such $\mum_\eps$ and $\delta_\eps$ converge to the resolvent corresponding to the original $\mum$. The same is true for the counterparts of these objects on the graphs. See Section \ref{sec:tcoim}.   Notably, in this approximation, $\delta(\eps)$, denoting intensity of jumps, increases to infinity, and the probability measures $\mum_\eps$ converge weakly to the Dirac measure at $0$.  

\item [5. ] The resolvent and the semigroup related to \eqref{int:1} with $\alpha +\beta =0$ and infinite $ \mum$, can be approximated similarly as in Point~4., but here the intensity of jumps grows to infinity yet faster. Again,  the same is true for the counterparts of these objects on the graphs. \end{itemize}

\emph{The stochastic approach} proceeds along an analogous, if unique, line. It also starts with simpler processes corresponding to the case where $\mum$ is zero or at least finite, and is later able to describe those more complex by means of approximation. However, in our estimation, it reaches deeper than the analytic one, as it provides  
a more detailed picture of the processes of interest. For example, the analytic approach, as detailed in Points 4. and 5. above, suggest that in the case of truly infinite $\mum$, jumps from the boundary may be infinitely small and occur infinitely often in finite intervals. The analytic approach not only confirms and elucidates this intuition, but provides explicit formulae for these processes.

An obstacle to overcome in the stochastic analysis is the fact that, unfortunately, limits involving formulae akin to \eqref{itomckean} are not easy to handle. This is because the local time featured in it is a discontinuous functional of the paths. Therefore, rather than \eqref{itomckean} we will use counterparts of  
its alternative version. It reads
   \begin{equation}
       \label{eq:blum}
      X(t) =\begin{cases}
       (\brown  +U\circ U^{-1}\circ L)\circ A_\alpha^{-1}(t), &   \text{ if } U^{-1}\circ{L}\circ A^{-1}_{\alpha}(t) <\zeta,\\
         \text{undefined},  & \text{ if }  U^{-1}\circ{L}\circ A^{-1}_{\alpha}(t)\geq\zeta,
    \end{cases}
   \end{equation}
where \begin{itemize}\item [(a) ] $\zeta$ is an exponentially distributed  random variable with parameter $\gamma$, independent of the already introduced $\brown $ and $U$ and $L$ is the running maximum of $-\brown $ (see \eqref{deflocal}),  
\item [(b) ] $A_\alpha(t)\coloneqq t+\alpha L(t), t \ge 0$ is strictly increasing and continuous, and 
\item [(c) ] $A_\alpha^{-1}$ is its inverse. \end{itemize}

Using \eqref{eq:Levythm1}, one can see that the processes defined in \eqref{itomckean} and \eqref{eq:blum} have the same distribution. The advantage of the latter is that 
it represents $X$ as a function of $\brown , U, L$ and $\zeta$, the processes and the random variable that are easier to handle. In particular, instead of involved $L_0$, we can work with the running maximum $L$, that depends on paths in a continuous way.  Moreover, a somewhat blurred picture of \eqref{new:cond} is now replaced by the role played by $\zeta$. 

We hasten to add that \eqref{eq:blum} was known to It\^o and McKean, see \cite[\S 12]{itop}. 
However, they tended to focus on \eqref{itomckean} instead because this formula allowed them to use excursion theory   to study the processes of interest  and  calculate their resolvents. Our final remark here is that $L_0^X$ featured in \eqref{eq:blum} and defined by  $L_0^X(t)\coloneqq U^{-1}\circ{L}\circ A^{-1}_{\alpha}(t)$ is a local time of $X$ at 0 --- see \cite[p. 69]{blum}, cf. \cite[\S 14]{itop}.


\subsection{Organization of the paper} 

The main body of the paper, that is, Sections \ref{sec:looa}--\ref{sec:stocha}, are preceded by preliminaries in Section \ref{sec:prelim}, where we recall information on Feller semigroups, cosine families,  and Skorokhod reflection problems. 
In Section \ref{sec:looa} the general Feller--Wentzel boundary conditions for the Laplace operator on star graph are discussed (see relation \eqref{loo:0}).  
The fundamental Section~\ref{sec:wbm} is devoted to Walsh's spider process.  Section \ref{sec:tcom=0} in turn describes more general processes with continuous paths, that is, the processes that can be slowed down at the graph center and/or killed when their local time at the graph center reaches the level of an independent exponentially distributed random variable --- they are characterized by the fact that the measure $\mum$ in \eqref{loo:0} is zero.  
Next, in Section \ref{sec:finite}, we commence the analysis of  processes with discontinuous paths: after `a sufficient number of reflections' such processes can restart in the interior of the state-space by a jump --- in the related boundary condition the measure $\mum$ is nonzero but finite.  Sections \ref{sec:tcoim} and \ref{sec:stocha} are devoted to the case where these jumps are infinitesimally small but occur infinitely often in finite intervals; this is the case of infinite $\mum$. 
The most important results of the paper are summarized in Section \ref{sec:synopsis}. This synopsis is preceded by Sections \ref{sec:ltot} and \ref{sec:calculation} where in particular the $\lambda$-potential of the local time is calculated, and an explicit formula for the time-change processes is given.

\vspace{-0.5cm}
\section{Preliminaries}\label{sec:prelim} 

\subsection{Feller semigroups}\label{pofs}

Let $S$ be a compact metric space, and $\coss$ be the space of continuous functions on $S$.  A Feller semigroup  (see e.g. \cites{bass,kniga,kallenbergnew}) 
is a strongly continuous family of positive contraction operators $\rod{T(t)}$  in $\coss$ such that $T(0)=I_{\coss}$ (the identity operator) and 
\[ T(t)T(s) = T(t+s) \qquad s,t \ge 0.\]
Because of the Riesz representation theorem, for each $x\in S $ and $t\ge 0$ there is a Borel measure $\textsf m_{x,t}$ on $S$ such that $\textsf m_{x,t} (S)\le 1$ and 
\begin{equation}T(t)f (x) = \int_S f \ud \textsf m_{x,t}, \qquad f \in \coss; \label{intro:1} \end{equation}
if $\textsf m_{x,t}(S) = 1$ for all $x$ and $t$, the process is said to be conservative or honest.

It is well-known that with each Feller semigroup one can associate a~Markov  process \rod{X(t)} with c\`adl\`ag paths \cite{ethier,kallenbergnew,rosen,rogers} (\ie paths that are right-continuous and possess left limits) in such a way that 
\begin{equation} \label{intro:2} T(t)f(x) = \mathsf E_x  f(X(t))\mathsf 1_{\{t< \tau\}} , \qquad  f\in \coss, t\ge 0, x\in S\end{equation}
where $\mathsf E_x $ denotes expectation conditional on $X(0)=x$ and $\tau$ is the lifetime of the process, that is, $\tau=\tau (\omega)$ is the random time up to which the path $X(t,\omega)$ is defined. Comparing Eqs. \eqref{intro:1} and \eqref{intro:2} we see that $\textsf m_{x,t}(\Gamma),$ {for a Borel set} $ \Gamma \subset S$ should be interpreted as the probability that  \rod{X(t)}  starting at $x$ at time $0$ will be in $\Gamma$ at a time $t\ge 0$. The measures $\textsf m_{x,t}$, together with the Markov property (which translates into Chapman--Kolmogorov equations), determine finite-dimensional distributions of \rod{X(t)}. If the distribution of $X(0)$ is specified, these finite-dimensional distributions, in turn, when combined with the requirement that the paths are \cadlag ,  
determine the distribution of \rod{X(t)} in the Skorokhod space $D([0,\infty), S)$.

This statement on uniqueness of the process, however, comes with a warning. The fact that for a certain process \rod{Y(t)} and a Feller semigroup we have 
\[  T(t)f(x) = \mathsf E_x  f(Y(t))\mathsf 1_{\{t< \tau\}} , \qquad  f\in \coss, t\ge 0, x\in S\]
neither implies that \rod{Y(t)} coincides with \rod{X(t)} nor that \rod{Y(t)} has Markov property. 

To elaborate on this succinct and a bit ambiguous statement, let $S\coloneqq [-\infty,\infty]$ be the two-point compactification of the real line, and $Y$ with values in $S$ be defined as follows. Let $\xi$ be a standard normal variable, and let $Y(t) \coloneqq x + \sqrt t \xi, t \ge 0$ provided that we start from $x\in \R$, and $Y(t) = \pm \infty, t \ge 0$ provided that we start at $\pm \infty$.  This process is manifestly not Markov because to determine its future behavior it does not suffice to know its present state, but rather its entire past. Nevertheless, with a good amount of good will, one can think of 
$\E_x f(Y(t)) \coloneqq \E f(x + \sqrt t \xi)$ for $x\in \R$, and $\E_{\pm \infty}f(Y(t)) = f(\pm \infty), t \ge 0, f\in \coss$, to see that $\E_x f(Y(t))$ coincides with $\E f(x +B(t))$, where $\rod{B(t)}$ is the standard one-dimensional Brownian motion. Our point is that, despite this fact, \rod{Y(t)} and \rod{B(t)} are clearly different.

To proceed, Feller semigroups are conveniently described by their generators. The generator $\gen$ of a Feller semigroup is defined by 
\[\gen f = \lim_{t\to 0} t^{-1} (T(t)f - f) \]
on the domain $\dom{\gen }$ composed of $f$ such that the above strong limit exists; that is, the right-hand side converges to $\gen f$ uniformly on $S$. It is well-known that the generator $\gen$ characterizes the semigroup uniquely; in particular, different semigroups have different generators.  Moreover, an operator $\gen$ is the generator of a Feller semigroup (shortly: a Feller generator) in $\coss$ iff the following three conditions are met 
\begin{itemize}
\item [1. ] $\gen$ is densely defined, 
\item [2. ] $\gen$ satisfies the positive-maximum principle,
\item [3. ] $\gen$ satisfies the range condition: for any $g \in \coss$ and $\lam > 0$ there is an $f \in \dom{\gen }$ such that $\lam f - \gen f = g.$ \end{itemize} 
The semigroup generated by $\gen$ will in what follows be denoted $\sem{\gen}.$ 

In defining Feller semigroups it is customary to require also that the related process is honest or conservative (which comes down to the requirement that $\mathsf 1_S \in \dom{\gen }$ and $\gen \mathsf 1_S= 0$) but in this paper we will allow the process to be not honest. It is easy to see that the process is honest iff $T(t)\mathsf 1_S = \mathsf 1_S$ for all $t\ge 0$ and this holds iff $\tau$ of equation \eqref{intro:1}, that is, the lifetime of the process, is a.s. equal to $\infty$. A common way of making a process honest comes down to adjoining an additional, isolated point, called graveyard, coffin state or cemetery, and commanding the process to stay there for all $t\ge \tau$. We note that this is done at the cost of changing the original Banach space to the space of continuous functions on the enlarged state-space.

\newcommand{\rmi}{R_\nu} 
\newcommand{\lil}{\lim_{\lam \to \infty}}
An alternative description of a Feller semigroup is provided via Feller resolvents. A family $\rla, \lam >0 $ of non-negative operators in $\coss$ is said to be a Feller resolvent if the following conditions are met:  
\begin{itemize}
\item [1. ] the Hilbert equation holds:
\begin{equation}\label{owc:3} (\lam- \nu) \rmi \rla = \rmi - \rla, \qquad \lam,\nu >0, \end{equation}
\item [2. ] for each $f \in \coss$, $\lil \lam \rla f=f,$ 
\item [3. ] $\lam \rla \mathsf 1_S \le \mathsf 1_S$ for all $\lam >0$.
\end{itemize} 

It may be argued (see the already cited references) that for each Feller resolvent there is a Feller generator $\gen$ such that $\rla = \rez{\gen }$, and thus also the related Feller semigroup. Moreover, the related process is conservative iff $\lam \rla \mathsf 1_S = \mathsf 1_S$ for all $\lam >0$.

A word about terminology: a family $\rla , \lam >0$ of bounded operators is said to be a pseudoresolvent, if it satisfies the Hilbert equation. This family is said to be regular, if additionally condition 2. given above holds. 

We note also the following crucial relation between the generator, the resolvent and the semigroup: 
\begin{equation*}
\rez{\gen} f (x) = \rla f (x) = \E_x \czn \e^{-\lam t} g(X(t)) \ud t, \quad x\in S, \lam >0, f\in \coss . \end{equation*}

\subsection{Cosine families} \label{sec:cf} 
A strongly continuous family $\{C(t), t \in \R\}$ of operators in a Banach space $\mathsf F$ is said to be a cosine family iff $C(0)$ is the identity operator and 
\[ 2 C(t) \coss = C(s+t) + C(t-s), \qquad t,s\in \R.\]
The generator of such a family is defined by 
\[\gen f = \lim_{t\to 0} 2t^{-2}{(C(t)f- f)}\]
for all $f \in \mathsf F$ such the limit on the right-hand side exists. 
For example, in $\cer$, the space of continuous functions on the two-point compactification of the real line,  there is the \emph{basic cosine family} given by  
\begin{equation}\label{intro:0} C(t) f(x) = \pol [ f(x+ct) + f(x-ct) ], \qquad x \in \R, t \in \R,\end{equation}
where, for simplicity of the formula, $c\coloneqq \frac {\sqrt 2}2 $. 
Its generator is the operator $f\mapsto \gen f= \frac 12 f''$ with domain composed of twice continuously differentiable functions on $\R$ such that $f'' \in \cer$. 

Each cosine family generator is automatically the generator of a strongly continuous semigroup (but not vice versa). The semigroup  such operator  generates is given by the \emph{Weierstrass formula} (see e.g. \cite{abhn}*{p. 219})
\begin{equation}\label{cosf:1}T(0)f =f    \mqquad {and} T(t) f  = {\textstyle \frac 1{2\sqrt{\pi t}}} \int_{-\infty}^\infty \e^{-\frac {s^2}{4t}} \coss f  \ud s, \qquad t >0, f \in \mathsf F.\end{equation} 
This formula expresses the fact that the cosine family is then, in a sense,  a more fundamental object than the semigroup, and properties of the semigroup can be hidden in those of the cosine family. 

Let us take a closer look at three important, simple examples. First of all, 
we recall that the semigroup of \emph{unrestricted} Brownian motion \eqref{brownek} is generated by the particular case of $\gen$ considered above with $c=\frac {\sqrt 2} 2$, and so the semigroup has the form \eqref{cosf:1}. Next, we note 
the two subspaces of $\cer$ formed by even and odd functions are invariant under $\{C(t), t \in \R\}$. To this end, it suffices to note that 
\begin{equation} (C(t)f)^o = C(t)f^o \mqquad{and} (C(t)f)^e=C(t)f^e, \quad f \in \cer, t \in \R,\label{symetrie}\end{equation}
where $f^e$ and $f^o$ denote the even and the odd parts of $f$, respectively, that is, \[ f^e(x)\coloneqq \tfrac 12 (f(x) +f(-x)) \mqquad{and} f^o(x) \coloneqq \tfrac 12 (f(x) - f(-x)), \qquad x \in \R.\] Also, the subspace of even functions is isometrically isomorphic to $\cerp$ with isomorphism $\mc I_e$ mapping a function in $\cerp$ to its even extension in $\cer$. It follows that 
\[ C_{\mathsf {ref}} (t)\coloneqq \mc I^{-1}_e C(t) \mc I_e, \qquad t \in \R \]
defines a strongly continuous cosine family $\{C_{\mathsf {ref}} (t), t \in \R\}$ in $\cerp$.  Since the even extension of a member $f$ of $\mathfrak C^2[0,\infty]$ belongs to $\mathfrak C^2[-\infty, \infty]$ iff $f'(0)=0$, the domain of the generator, $\gen_{\mathsf {min}}$, of this family is composed precisely of $f \in \mathfrak C^2[0,\infty]$ with this property, and we have $\gen_{\mathsf {min}} f = \frac 12 f''$.  Hence, also the semigroup of reflected Brownian motion (see Sections \ref{sec:si} or \ref{sec:srac}) is  of the form \eqref{cosf:1}. For future reference we note that 
\begin{equation} \label{odbity}   C_{\mathsf {ref}} (t)f(x) = \tfrac 12 \left (f(x+ct) + f(|x-ct|)\right ), \qquad t\in \R, x \in \R.\end{equation} 
A similar analysis of the subspace of odd functions reveals that 
\begin{equation} \label{minimalny}   C_{\mathsf {min}} (t)f(x) = \tfrac 12 \left (f(x+ct) + \sgn (x-ct) \cdot f(|x-ct|)\right ), \quad t\in \R, x \in \R\end{equation} 
defines a strongly continuous cosine family in $\mathfrak C_0(0,\infty]$, and that its generator coincides with that of the semigroup describing minimal Brownian motion.

As we shall see in the main body of the article, also a couple of semigroups describing Brownian motions on star graphs are given by the Weierstrass formula.

\subsection{Skorokhod's problems, generalized inverses}\label{sec:inverses} 

Let $\omega\colon \R^+\to \R$ be a c\`{a}dl\`{a}g function such that  $\omega(0)\geq 0$. A pair $(\omega\refl,\ell)$ of c\`{a}dl\`{a}g functions $\omega\refl,\ell \colon \R^+\to \R$ is called a solution to the Skorokhod reflection problem for $\omega$ if 
\begin{itemize}
\item $\omega\refl$ is nonnegative,
\item $\omega\refl =\omega +\ell$,  whereas 
\item $\ell$ is nondecreasing with $\ell(0)=0$, and increases only when $\omega\refl$ is at $0$, that is,
\begin{equation}\label{eq:S4}
\int_{[0,\,\infty)}\1_{\{\omega\refl(t)>0\}}{\rm d}\ell(t)=0.
\end{equation} 
\end{itemize}
It is well known (see e.g. \cite{karatzas}*{p. 210}, where, however, only continuous $\omega$s are considered; the case of  c\`{a}dl\`{a}g functions is discussed in  \cite{asmussen}*{p. 250}, \cite{refle} and \cite{bookandrey}) that a unique  pair that solves  the Skorokhod's reflection problem is given by the formulae
\begin{equation}\label{srp:sol}
\omega\refl (t) \coloneqq  \omega (t) +\ell (t)  \mqquad{and}  \ell(t)\coloneqq \sup_{s\in[0,t]}(-\omega(s)\vee 0)
\qquad t\geq 0. \end{equation}
The  $\omega\refl $ is termed Skorokhod's reflection of $\omega$ at 0.

 In the generalized Skorokhod reflection problem, besides   $\omega$ described above, we are also given a  strictly increasing   c\`adl\`ag function $\psi\colon \R^+ \to \R$ such that $\psi(0)=0$, and search for a pair of functions $\omega\refl,\ell\colon \R^+\to \R$ with the following properties: 
    \begin{itemize}
        \item $\omega\refl$ is nonnegative,
        \item $\omega\refl =\omega + \psi \circ \ell,$ whereas 
        \item $\ell$ is a non-decreasing continuous function with $\ell(0)=0$ and
        \[
        \int_0^\infty \1_{\{\omega(t)>0\}} \ud \ell (t)=0.
        \]
    \end{itemize}
Such a pair $(\omega\refl,\ell)$ is said to be a solution to the generalized Skorokhod problem with noise $\omega$ and  reflection governed by $\psi$.

As expounded in \cite{Pilipenko+Sarantsev:2024} (see also \cite{bookandrey}) the analysis of this problem involves generalized inverses. To recall, if $\psi$ is an increasing c\`adl\`ag function such that $\lim_{t\to\infty} \psi(t)=+\infty$, then its generalized inverse is defined by 
\[ \psi^{-1}(t)\coloneqq \inf\{s\geq 0\colon \psi(s)>t\}, \qquad  t\geq0. \]
Figure \ref{inverses} shows an example of a function $\psi$, its generalized inverse $\psi^{-1}$, and  composition $\psi \circ \psi^{-1}$.


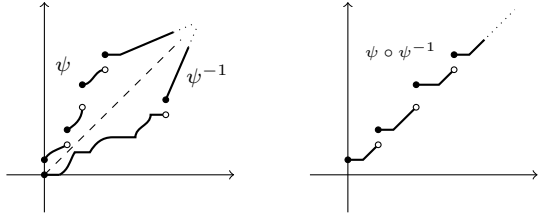
\begin{figure}
\begin{tikzpicture}
\begin{scope}
\draw [->] (-0.5,0)--(2.5,0);
\draw [->] (0,-0.5)--(0,2.3);
\draw[black,fill=black] (0,0.2) circle (.25ex);
\draw [thick] (0,0.2) to [out=60,in=210] (0.3,0.4);
\draw[black,fill=white] (0.3,0.4) circle (.25ex);
\draw[black,fill=black] (0.3,0.6) circle (.25ex);
\draw[thick] (0.3,0.6) to  [out=30,in=270] (0.5,0.9);
\draw[black,fill=white] (0.5,0.9) circle (.25ex);
\draw [thick] (0.5,1.2) to [out=0,in=180] (0.8,1.4);
\draw[black,fill=white] (0.8,1.4) circle (.25ex);
\draw[black,fill=black] (0.5,1.2) circle (.25ex);
\draw[black,fill=black] (0.8,1.6) circle (.25ex);
\draw[thick] (0.8,1.6) to (1,1.6) to (1.7,1.9);
\draw[dotted] (1,1.6) to (2,2.025);
\node [above] at (0.25,1.2) {\Small{$\psi$}};
\draw [dashed] (0,0)--(1.8,1.8);
\node [above] at (2.15,1.05) {\Small{$\psi^{-1}$}};
\draw[black,fill=black] (0,0) circle (.25ex);
\draw [thick] (0,0) to (0.2,0);
\draw [thick] (0.2,0) to [out=30,in=240] (0.4,0.3) to (0.6,0.3) to  [out=60,in=180] (0.9,0.5) to (1.2,0.5) to [out=90,in=270] (1.4,0.8) to (1.6,0.8); 
\draw [thick] (1.6,1) to (1.9,1.7);
\draw[dotted] (1.6,1) to (2.025,2);
\draw[black,fill=white] (1.6,0.8) circle (.25ex);
\draw[black,fill=black] (1.6,1) circle (.25ex);
\end{scope}
\begin{scope}[shift={(4,0)}]
\draw [->] (-0.5,0)--(2.5,0);
\draw [->] (0,-0.5)--(0,2.3);
\node [above] at (0.7,1.4) {\tiny{$\psi\circ\psi^{-1}$}};
\draw[black,fill=black] (0,0.2) circle (.25ex);
\draw [thick] (0,0.2)--(0.2,0.2)--(0.4,0.4);
\draw[black,fill=white] (0.4,0.4) circle (.25ex);
\draw [thick] (0.4,0.6)--(0.6,0.6)--(0.9,0.9);
\draw[black,fill=black] (0.4,0.6) circle (.25ex);
\draw[black,fill=white] (0.9,0.9) circle (.25ex);
\draw [thick] (0.9,1.2)--(1.2,1.2)--(1.4,1.4);
\draw[black,fill=black] (0.9,1.2) circle (.25ex);
\draw[black,fill=white] (1.4,1.4) circle (.25ex);
\draw [thick] (1.4,1.6)--(1.6,1.6)--(1.8,1.8);
\draw [dotted] (1.6,1.6)--(2.2,2.2);
\draw[black,fill=black] (1.4,1.6) circle (.25ex);
\end{scope}
\end{tikzpicture}
\caption{\Small{Generalized inverses. 
If a $u\ge 0$  is a point of continuity of $\psi$ and $\psi (u) =t$ then $\psi^{-1} (t) =u$, and so $\psi \circ \psi^{-1} (t)=t$. If, however, at a  $u\ge 0$ there is a jump of $\psi$, that is, $\psi(u-) < \psi (u)$ (by default $\psi(0-)=0$), then  $\psi^{-1} (t) = u $ for all $t\in [\psi(u-), \psi(u)]$, and so, for such $t$, we have $\psi \circ \psi^{-1}(t)=\psi (u)$.} 
}\label{inverses}
\end{figure}

The theorem alluded to above states that there exists a unique solution to the generalized Skorokhod's problem provided that (a) $\omega$ does not have negative jumps, and (b) $\psi$ is strictly increasing with $\psi (0)=0$ and $\lim_{t\to\infty} \psi(t)=+\infty$. 
The solution is given explicitly: 
\begin{align*} 
\omega\refl &= \omega +\psi \circ \psi^{-1} \circ k\mquad { and }
\ell=\psi^{-1}\circ k 
\intertext{
where}k(t)&\coloneqq \sup_{s\in[0,\,t] }(-\omega(s)\vee 0), \qquad t\geq 0.\end{align*}

\section{Laplace operators on a star graph as generators}\label{sec:looa}
For a given natural number $\ka \ge 2$,  let $S$ be the union of $\ka$ copies of the compactified right half-line $[0,\infty]$, with all left ends identified, and let $\coss$ be the space of continuous functions on $S$. In other words, $S$ is the star graph $K_{1,\ka}$ depicted at Figure \ref{slg}, and continuous functions on $S$ can be identified with continuous functions $f$ on \begin{equation}\label{tildees} \widetilde S\coloneqq \bigcup_{i\in \kad} \left (\{i\} \times [0,\infty]\right ), \qquad \text{where } \kad \coloneqq \{1,\dots, \ka\}, \end{equation}
that satisfy $f(i,0)=f(j,0), i,j\in \kad $; their common value at $0$ will be denoted $f(\zero)$, and the graph's center will be denoted $\zero$.  
Alternatively, members of $\coss$ can be viewed as sequences $(f_i)_{i\in \kad}$ of elements of $\cerp$, where $f_i$ is defined by  $f_i(x)=f(i,x), x \ge 0$, such that $f_i(0)=f_j(0), i,j\in \kad$.

\subsection{Boundary conditions for the Laplace operator in $\coss$}\label{sec:bcft}

Let $\mathfrak C^2(S)$ be the subspace of $\coss$ composed of $f$ such that each $f_i, i\in \kad$ defined above belongs to $\mathfrak C^2[0,\infty]$  (see Introduction) 
and we have $f''_i(0)=f''_j(0)$ for all $i,j\in \kad$. For such $f$, we define \[ \Delta f = \pol f''\] where $f''(i,x) = f_i''(x), i\in \kad, x \ge 0$. 

Following the arguments of Feller \cite{fellera3,fellerek} (see also \cite[{p. 186}]{ito}, \cite[{pp.~39--40}]{mandl} or our Appendix \ref{bckon}, where they are repeated with necessary changes), one can show that if a restriction, say, $\gen$, of $\Delta$ is a Feller generator, then there are non-negative constants $\alpha, \beta_i,i\in \kad $, and $\gamma$, and a (possibly infinite) positive Borel measure $\mum$ on $S_{\zero} \coloneqq S\setminus \{\zero\}$ such that 
$\int_{S_{\zero}} (1\wedge x)\, \mum (\mud x)<\infty$, at least one of the conditions $\alpha + \sui \beta_i >0$ and $\mum (S) =\infty$ holds,  and
\begin{equation}\label{loo:0} {\tfrac \alpha 2} f''(\zero) - \sum_{i\in \kad} \beta_i f'_i(0) + \gamma f(\zero) = \int_{S_{\zero}} [f(x) - f(\zero)]\, \mum (\mud x), \qquad f \in \dom{\gen }. \end{equation}	
We stress that, as in the case of \eqref{int:1}, from the outset we exclude the, theoretically possible but uninteresting, possibility of $\mum$ having a mass at any of the points $(i,\infty), i\in \kad$. Moreover, condition $\int_{S_{\zero}} (1\wedge x)\, \mum (\mud x)<\infty$ means that the function $\{i\}\times [0,\infty)\ni (i,x) \mapsto 1\wedge x$ is integrable with respect to $\mum\vert_{\{i\}\times [0,\infty)}$ for each $i\in \kad$.

It is one of the main goals of our paper to show that the converse is also true: the restriction of $\Delta$ to the domain where \eqref{loo:0} is satisfied, is a Feller generator. Moreover, we would like to construct the related process using natural building blocks. The case where $\mum$ is zero has been thoroughly described in \cites{kostrykin2012}--\cite{kostrykin2}, and it turns out that this is precisely the case in which the paths of the process are continuous; these processes are discussed again in some detail in Sections \ref{sec:wbm}--\ref{sec:tcom=0}. In Sections \ref{sec:finite}--\ref{sec:stocha} the measure $\mum$ is shown to be responsible for jumps from the boundary. Roughly, in Section \ref{sec:finite}, devoted to the case of finite $\mum$,  jumps occur between regular excursions from zero and are comparatively rare. In Sections \ref{sec:tcoim} and \ref{sec:stocha}, where infinite $\mum$ is discussed, they become very frequent and infinitesimally small at the same time; in the extreme case of $\alpha + \sui \beta_i =0$ they turn out to be the only means of exiting from the boundary. 

\newcommand{\cosik}{
\begin{figure}
\begin{tikzpicture}
\begin{scope}
\draw [->] (-0.5,0)--(2.5,0);
\draw [->] (0,-0.5)--(0,2.3);
\draw [dashed] (0,0)--(1.8,1.8);
\draw [thick] (0,0) to (0.2,0);
\draw [thick] (0.2,0) to [out=30,in=240] (0.4,0.3) to (0.6,0.3) to  [out=60,in=180] (0.9,0.5) to (1.2,0.5) to [out=90,in=270] (1.4,0.8) to (1.6,0.8) to (1.9,1.7);
\draw[dotted] (1.6,0.8) to (2,2);
\end{scope}
\begin{scope}[shift={(4,0)}]
\draw [->] (-0.5,0)--(2.5,0);
\draw [->] (0,-0.5)--(0,2.3);
\end{scope}
\end{tikzpicture}
\caption{\Small{Excursions. To be drawn.}}\label{exc}
\end{figure} }

We end this section with the following remark. For convenience, throughout the paper we work with the compactified graph, that is, we adjoin special points $(i,\infty), i \in \kad$ to the locally compact space $\bigcup_{i\in \kad} \left (\{i\} \times [0,\infty)\right )$ --- see \eqref{tildees}. These points are isolated in the sense that in the processes featured in the paper none of them can be reached in finite time from `regular' points of the state-space. Indeed, reaching them is clearly impossible by means of an underlying Brownian motion, and by assuming that $\mum$ has no mass at $\bigcup_{i\in \kad}\{(i,\infty)\}$ we exclude the possibility of reaching them by a jump. These points are furthermore traps: when started at any of them the processes considered in this paper will stay there for ever. This is a natural consequence of the definition of $\Delta$: for $f\in \mathfrak C^2(S)$, each of the limits $\lim_{x\to \infty} f''(i,x), i \in \kad$ exists and equals zero, and so $\Delta f(i,\infty) =0$.

\subsection{Operators $\gen_{\alpha,\beta,\gamma}$ as generators}\label{sec:3.2}

Since the resolvents and semigroups corresponding to the case where $\mum=0$ will be used as building blocks of the more complex resolvents and semigroups involving nonzero $\mum$, we complete this section by proving a generation result pertaining to the former objects.  

Let  non-negative parameters $\alpha, \beta_i, i\in \kad$ and  $\gamma$ be such that  
\begin{equation}\label{mainas}\alpha +\sum_{i\in \kad} \beta_i > 0,\end{equation}  and let $F:\mathfrak C^2(S)\to \mathbb R$ be the functional given by
\begin{equation}\label{loo:1} Ff = {\alpha}\Delta f(\zero) - \sum_{i\in \kad} \beta_i f_i'(0) + \gamma f(\zero). \end{equation}
Then, the operator \begin{equation}\label{gengen} \gen_ {\alpha,\beta,\gamma}\end{equation} defined as the restriction of $\Delta$ to \[ \dom{\gen _{\alpha,\beta,\gamma}} \coloneqq C^2(S)\cap \ker F\] is a Feller generator. Indeed, it is obviously densely defined and an easy argument shows that it satisfies the positive maximum principle --- comp. e.g. \cite{knigazcup}*{p. 17}, or see our Appendix \ref{maks}. Furthermore, the solution $f$ to the resolvent equation $\lam f - \gen_ {\alpha,\beta,\gamma} f =g$ can be given explicitly for any $\lam >0$ and $g\in \coss$: 
\begin{equation}\label{loo:2} f_i(x) =  C_i(g)\e^{\slam x} + D_i(g)\e^{-\slam x}- \sqrt{\tfrac 2\lambda} \int_0^x \sinh\slam (x-y) g_i(y)\ud y\end{equation}
where 
\begin{align}\label{cee} C_i &=C_i(g)  \coloneqq \frac 1{\slam} \czn \e^{-\slam y} g_i(y)\ud y, \qquad i \in \kad  \end{align}
$D_i =D_i(g) \coloneqq E - C_i, i\in \kad $ and 
\begin{equation}
E=E(g)\coloneqq \frac {2\slam \sui \beta_i C_i(g) + \alpha g(\zero)}{\lam \alpha + \slam \sui \beta_i + \gamma }\label{loo:3} \end{equation}
is the common value of $f_i(0),i\in \kad$. 

To prove the last statement we search for solutions of the resolvent equation in the form \eqref{loo:2} where $C_i$s and $D_i$s are constants to be found, and the third term is a particular solution of the related ODE. The requirement that $f_i$ is to belong to $\cerp$ then forces $C_i$ to be defined as above (see \cite{kniga}*{Section 8.2.18} or \cite{knigazcup}*{p. 18} for details, if necessary).  Since $f_i(0)$ cannot depend on $i\in \kad$, and $F\seq{f_i}$ is to be zero, we obtain a system of linear equations on $E$ and $D_i$s, which yields \eqref{loo:3}. Uniqueness of the solution to the resolvent equation is guaranteed by the fact that  $ \gen_ {\alpha,\beta,\gamma}$ satisfies the positive-maximum principle.

For future reference we note that \eqref{loo:2} can be written as 
\begin{equation}\label{loo:4} \rez{\gen _{\alpha,\beta,\gamma}} g = \rla^0 g + E(g) \elam^0 \end{equation}  
where 
\begin{align}\nonumber 
(\rla^0 g)_i(x) &\coloneqq \tfrac 1{\slam} \czn (\e^{-\slam |x-y|} - \e^{-\slam (x+y)}) g_i(y)\ud y \\
&= 2C_i(g) \sinh \slam x - \sqrt {\tfrac 2{\lambda}}\int_0^x \sinh \slam (x-y) g_i(y)\ud y, \label{loo:5} \end{align}
  is the resolvent of the minimal Brownian motion on $S$, and 
\begin{equation}\label{lifetime} (\elam^0 )_i (x) \coloneqq \e^{-\slam x}, \qquad x \ge 0, i \in \kad. \end{equation}

The fact that in the case of $\gamma >0$ the process is not honest is expressed in $\lam \rez{\gen _{\alpha,\beta,\gamma}}\mathsf 1_S \not = \mathsf 1_S$.
The function $\elam \coloneqq \mathsf 1_S- \lam \rez{\gen _{\alpha,\beta,\gamma}}\mathsf 1_S $ is the Laplace transform of the lifetime $\tau$ of the process (see e.g. \cite{konkusSIMA}*{Proposition 3.5}), and a short calculation reveals that 
\begin{equation}\label{loo:6} 
(\elam)_i (x) = \frac \gamma{\lam \alpha + \slam {\sui} \beta_i + \gamma} \e^{-\slam x}, \qquad x \ge 0, i\in \kad.\end{equation}

\subsection{Lumping edges together}\label{sec:lump0}

To record an interesting property of the resolvent of $\gen_ {\alpha,\beta,\gamma}$, given by \eqref{loo:4}, for a given $\mathpzc n\le \ka$, let $\mathfrak C(S_\mathpzc n)$ be the space of continuous functions on the star graph with $\mathpzc n$ edges, to be distinguished from  the space $\mathfrak C(S_\ka)$  of continuous functions on the star graph with $\ka$ edges. Also, let $\Psi$ map $\kad $ onto $\mathpzc N \coloneqq \{1,\dots,\mathpzc n\}$, and $\mathfrak C_\Psi (S_\ka)$ be the related subspace  of $g=\seq{g_i} \in \mathfrak C(S_\ka)$ such that $g_i =g_j$ as long as $\Psi(i)=\Psi(j)$. This subspace is clearly isometrically isomorphic to $\mathfrak C(S_\mathpzc n)$, the isomorphism, say, $\mathcal I: \mathfrak C(S_\mathpzc n)\to \mathfrak C_\Psi (S_\ka)$, being given by the requirement that the $i$th coordinate of $\mc Ig$ equals 
\( g_{\Psi(i)}, g \in  \mathfrak C(S_\mathpzc n), i\in \kad \). Moreover, the operators $\rez{\gen_ {\alpha,\beta,\gamma}}, \lam>0$ leave $\mathfrak C_\Psi (S_\ka)$ invariant (because so do $\rla^0, \lam >0$ and $\elam^0$ belongs to $\mathfrak C_\Psi (S_\ka)$)
and we have 
\begin{equation}\label{let:1} \mc I^{-1} \rez{\gen_ {\alpha,\beta,\gamma}} \mc I = \rez{\gen_{\alpha,\widetilde \beta, \gamma}}, \qquad \lam >0, \end{equation}
where $\gen_{\alpha,\widetilde \beta, \gamma}$ is the generator in $\mathfrak C(S_\mathpzc n)$ with $\widetilde \beta = (\widetilde \beta_j)_{j\in \mathpzc N}$ given by  $\widetilde \beta_j 
 \coloneqq \sum_{\{i\colon \Psi(i)=j\}} \beta_i, j \in \mathpzc N$. 
Formula \eqref{let:1} is an early indication of the fact that by lumping edges of the underlying graph, from the process generated by $\gen_ {\alpha,\beta,\gamma}$ one obtains the process generated by $\gen_ {\alpha,\widetilde \beta,\gamma}$ (see the warning in Section \ref{pofs}). 
We will be  coming back repeatedly to this subject throughout these notes.


\vspace{-0.5cm}
\section{Walsh's process (the case of $\alpha=\gamma=0$ and $\mum=0$)}\label{sec:wbm}

Let $\beta=\seq{\beta_i}$ be a vector \emph{with positive coordinates} such that $\sum_{i\in \kad} \beta_i=1$; the second, unlike the first, of these assumptions is made without loss of generality because multiplying $F$ of \eqref{loo:1} by a positive constant does not change $F$'s kernel. 

Somewhat informally,  the Walsh's motion, a.k.a., Walsh's spider  process or Walsh's spider, is a continuous time-homogeneous strong Markov process on $S$ that behaves like a standard one-dimensional Brownian motion on each ray \[ \mathpzc r_i \coloneqq \{i\}\times [0,\infty], \qquad i \in \kad\] until hitting the graph center; upon hitting the center 
Brownian excursions `select' ray $\mathpzc r_j$ with probability $\beta_j, j \in \kad$. In what follows the process will be denoted $W_\beta$. 

\subsection{Semigroups, resolvents, and cosine families}\label{sec:srac}
As revealed in \cite{barlow}, Walsh's spider is a Feller process on $S$ and its semigroup can be written explicitly in terms of the minimal and reflected semigroups.

To elaborate, we recall first that the (one-dimensional) minimal Brownian motion semigroup \rod{T_{\mathsf {min}}(t)} is a strongly continuous semigroup of operators in $\mathfrak C_0(0,\infty]$, the subspace of $f \in \cerp$ such that $f(0)=0$. The semigroup is given by $T_{\mathsf {min}}(0)f=f$ and 
\[  [T_{\mathsf {min}}(t)f](x) =\int_{0}^\infty \left ( \varphi_t (x-y)- \varphi_t (x+y)\right )f(y) \ud y,  \]
for $t>0,x\ge 0$ and $f \in \mathfrak C_0(0,\infty]$; 
here and throughout the paper, 
\[ \varphi_t (z) =  \tfrac 1{\sqrt{2\pi t}}  \e^{-\frac{z^2}{2t}}, \qquad z\in \R, t >0 .\]
The domain of its generator $\gen_ {\mathsf {min}}$ is composed of $f \in \mathfrak C_0(0,\infty]$ 
that are twice continuously differentiable with $f''\in \mathfrak C_0(0,\infty]$. For such $f$ we have 
$\gen_ {\mathsf {min}} f = \frac 12 f''$.

\newcommand{\starzec}{
To elaborate, the minimal Brownian motion on $S$ is a process which, while on one of the edges, away from the graph center, behaves like a standard one-dimensional Brownian motion. However, at the first moment it touches the center, it is killed and removed from the state-space. Strictly speaking, thus, its state-space is not $S$, but rather $S^0$, defined as $S$ with center removed. The related semigroup can be constructed by means of $\ka$ copies of the semigroup of minimal Brownian motion on the half-line, and has the following form: $T(0)f=f$ and, for $t>0$, 
\[  [T_{\mathsf {min}}(t)f](x,i) =\int_{0}^\infty \left ( \varphi_t (x-y)- \varphi_t (x+y)\right )f(i,y) \ud y, \qquad x\ge 0,i \in \kad;\] 
here and throughout the paper, 
\[ \varphi_t (z) =  \tfrac 1{\sqrt{2\pi t}}  \e^{-\frac{z^2}{2t}}, \qquad z\in \R, t >0 .\]
The fact that the graph's center does not belong to the state-space of the minimal Brownian motion is expressed  in the lack of strong continuity of \rod{T_{\mathsf{min}}(t)}. This semigroup is strongly continuous only on the subspace $\mathfrak \cosso$ of functions that vanish at the graph center, and its generator there is the operator, say, $\gen_ {\mathsf {min}}$ defined as follows. Its domain is composed $(f_i)_{i\in \kad}\in \mathfrak \cosso$ such that each $f_i$ is  twice continuously differentiable  and $f_i''$ belongs to $\mathfrak \cosso$, and then we put $\gen_ {\mathsf {min}} (f_i)_{i\in \kad} = \frac 12 (f_i'')_{i\in \kad}$.}  

We recall also that the reflecting Brownian motion semigroup in $\cerp$  is given by $T_{\mathsf {ref}}(0)f=f$ and 
\[ T_{\text{ref}} (t) f (x) =\EE  f (|x + \brown (t)|) = \int_{0}^\infty \left ( \varphi_t (x-y)+ \varphi_t (x+y)\right )f(y) \ud y, \]
for $x\ge 0,t>0$ and $f\in \cerp$. 
The domain of the generator $\gen_ {\text{ref}}$ of $\rod{ T_{\text{ref}} (t) }$ is composed of $f$ that are twice continuously differentiable with $f'' \in \cerp$, and satisfy $f'(0)=0$; for such $f$ we have  $ \gen_{\mathsf {ref}} f=\frac 12 f''$.

In terms of the semigroups described above, the Walsh's process semigroup $\rod{T_{W_\beta}(t)}$ has the following form  (see \cite{barlow}, eq. (2.2)):  
\begin{equation}\label{wsp:1}  (T_{W_\beta} (t) f)_i = T_{\mathsf {min}} (t) (f_i - \overline f ) + T_{\mathsf {ref}} (t) \overline f, \qquad f\in \coss, t \ge 0, i \in \kad, \end{equation}
where \[ \overline f \coloneqq \sum_{j\in \kad} \beta_j f_j.\] 

There are three immediate and important consequences of \eqref{wsp:1}. First of all,  the generator 
is the restriction of $\Delta$ to $f\in \dom{\Delta}$ such that, as a special case of \eqref{loo:0},  
\begin{equation*}
(\overline f)' (0)= \sum_{i\in \mc K} \beta_i f'_i(0)=0;\end{equation*}
in other words, the generator coincides with $\gen_{0,\beta,0}$ of Section \ref{sec:looa}. 
 Indeed, on one hand, for $f$ described above, $f_i - \overline f $ belongs to the domain of $\gen_ {\mathsf {min}} $ for all $i\in \kad $, and  $\overline f$ belongs to the domain of $\gen_ {\mathsf{ref}}$. Therefore, 
\[ \lim_{t\to 0} t^{-1} ( T_{W_\beta} (t) f - f)_i = \gen_ {\mathsf {min}} (f_i - \overline f) 
+ \gen_ {\mathsf{ref}}\overline f=\tfrac 12 f_i'',  \]
showing that the generator of $\rod{T_{W_\beta}(t)}$ extends $\gen_ {0,\beta,0}$. On the other hand, 
we know from Section \ref{sec:3.2} that  given a $\lam >0$ and a $g\in \coss$ there is precisely one $f \in \dom{\gen_{0,\beta,0}}$ such that $\lam f - \gen_{0,\beta,0} f = g$. A standard argument shows thus that the searched for generator cannot be a proper extension of $\gen_{0,\beta,0}$ (see e.g. \cite{kallenbergnew}*{p. 377} or \cite{rogers}*{p. 242}). A~different derivation of the boundary condition for the Walsh's process can be found in \cite{kostrykin2012}.

Secondly, \eqref{wsp:1} implies that $W_\beta$ has the following transition density function:
\begin{equation}
p_{t}((i,x),(j,y))
=
\begin{cases}
\varphi_t (x-y)
 + (2\beta_i-1)\varphi_t (x+y),& \
i=j, \\
2\beta_j\varphi_t (x+y),& \
i\neq j,
\end{cases}\label{gestosci}
\end{equation}
where  $t>0$ and $x, y\geq 0$. These formulae, in turn, show that if the initial distribution of Walsh's process is absolutely continuous with respect to the natural Lebesgue measure on $S$, then so is its distribution at any $t\ge 0$. One can thus think of the semigroup of Markov operators in $L^1(S)$   that describes this dynamics (for the notion of Markov operators see e.g. \cite{lasota} or \cite{rudnickityran}). As shown in \cite{skosny2} (comp. \cite{abtk}) the domain of the generator of this semigroup is characterized by the `dual' boundary conditions
\[ \sum_{i\in \kad}\phi'_i(0)=0 \mquad{and} \beta_j \phi_i(0) = \beta_i \phi_j(0), \quad i,j\in \kad.\]

As the third consequence of \eqref{wsp:1}, the resolvent $\rez{\gen_{0,\beta,0}}, \lam >0$, which is the Laplace transform of the Walsh's process semigroup, has the following representation: for $g\in \coss, x\ge 0, i\in \kad$,
\begin{align*}
 (\rez{\gen_{0,\beta,0}} g)_i(x)& =  (\rla^0g)_i(x) + 2 \czn \czn \e^{-\lam t} \varphi_t (x+y) \overline g(y) 
\ud y  \\ 
 &=   (\rla^0g)_i(x)  + 2 \sum_{i\in \kad}\beta_i C_i(g) (\elam^0)_i (x), \end{align*}
 where $\rla^0$, introduced in \eqref{loo:5},  is the resolvent of the minimal Brownian motion on $S$, $\elam^0$ is the Laplace transform of its life-time (see \eqref{lifetime}), and $C_i(g)$s are defined in \eqref{cee}.  The above relation, 
  not surprisingly,  is a special case of \eqref{loo:4}.

As shown recently in \cite{ela}, the fundamental formula \eqref{wsp:1} is but a reflection of a deeper result.  To see this, let $\mathsf F= (\cer)^\ka $ be the Cartesian product of $\ka$ copies of the space $\cer$ of Section \ref{sec:cf}, and let $\{\mc C(t), t \in \R\}$ be the \emph{Cartesian product basic cosine family} in $\mathsf F$ defined by means of the basic cosine family of \eqref{intro:0} as follows: 
\[ \mc C(t)(h_i)_{i \in \kad } = \seq{C(t)h_i}, \qquad t \in \R .\]
Using \eqref{symetrie} it is then easy to see that the subspace 
\begin{equation} \mathsf F_\beta \coloneqq \{ (h_i)_{i\in \kad}\in \mathsf F \colon h_i^e = h_j^e, i,j\in \kad \text{ and } (\overline h)^o =0\}\label{przef}\end{equation}
 is invariant under $\{\mc C(t), t \in \R\}$; here, as before, $\overline h\coloneqq \sum_{i\in \kad} \beta_i h_i$, and  $h^e$ and $h^o$ denote the even and the odd parts of $h$, respectively. 

 Moreover, $\mathsf F_\beta$ is isomorphic to $\coss \subset (\cerp )^\ka$. For, given $\seq{h_i} \in \mathsf F_\beta$, we obtain an element $\seq{f_i}\coloneqq \mc R \seq{h_i}$ by defining $f_i(x) = h_i(x), x\ge 0, i\in \kad$; note that, by definition of $\mathsf F_\beta$, $h_i(0) = h_j(0)$ and so $f_i(0)=f_j(0),  i,j\in \kad$. Vice versa, given $\seq{f_i}\in \coss$ we define $\seq{h_i}\coloneqq \mc E_\beta \seq{f_i}$ by 
 \[ h_i(x) = \begin{cases} f_i(x),& x \ge 0, \\
 -f_i(-x) + 2 \overline f (-x),& x < 0,   \end{cases} \]
 to check that $h_i^e (x)=  \overline f (x), x\ge 0, i \in \kad$ and $\sum_{i\in \kad}\beta_i h_i^o=0$,  which implies that $\seq {h_i}$ belongs to $\mathsf F_\beta$.  Since $\mc E_\beta $ and $\mc R$ are inverses of one another, it follows that $\{\mc C_\beta (t), t \in \R\} $ defined by 
 \[ \mc C_\beta (t) = \mc R \mc C(t) \mc E_\beta , \qquad t \in \R \]
is a strongly continuous cosine family in $\coss $. Now, a straightforward calculation shows that 
\begin{equation}\label{glebszy}  (\mc C_\beta (t) f)_i = C_{\mathsf {min}} (t) (f_i - \overline f ) + C_{\mathsf {ref}} (t) \overline f, \qquad f\in \coss, t \in \R,i \in \kad,  \end{equation}
where $ C_{\mathsf {min}} $ and $ C_{\mathsf {ref}}$, defined in Section \ref{sec:cf}, are cosine families related to the minimal and reflected Brownian motions. Because of \eqref{wsp:1} and the
Weierstrass formula, the generator of $\mc C_\beta$ coincides with the generator  of $T_{W_\beta}$, that is, with $\gen_{0,\beta,0}$. 

To summarize: \eqref{wsp:1} mirrors \eqref{glebszy}, and the latter formula is a direct consequence of the fact that $\mathsf F_\beta$ of \eqref{przef} is invariant under the Cartesian product basic cosine family.

\subsection{Stochastic description}\label{sec:opiswalsha}

Having described its semigroup, we proceed with stochastic characterization of Walsh's process itself. In this section, starting with a formal definition of the process, we present three characterization theorems. The first two of these come from the early paper \cite{barlow}, whereas the third has been obtained recently in \cite{bayraktar}. 

\subsubsection{Two theorems of Barlow, Pitman and Yor}\label{sec:bpy}
To begin with, it is clear from the intuitive description given in Section~\ref{sec:srac}, or either of formulae \eqref{wsp:1} and \eqref{gestosci}, that the distance of Walsh's spider from the graph's center is a reflected Brownian motion. In the Walsh's original construction, in turn, the spider process is obtained as a function of reflected Brownian motion started at $0$ as follows. Let a  sequence $\jcgi{\epsilon_i}$ of i.i.d. random variables be given such that $\Pb(\epsilon_i = j)=\beta_j, i\ge 0, j\in \kad$. Moreover, let us enumerate excursions of a given, independent of  $\jcgi{\epsilon_i}$, reflected  Brownian motion $\brown\refl$ in an arbitrary measurable way, and let $n(t)=n(t,\omega)$ be the number of the excursion at time $t$ in this enumeration (if $\brown\refl(t)= 0$, then $n(t)$ is defined arbitrarily, say $n(t)=1$). Then
\begin{equation}\label{defwal} W_\beta (t) \coloneqq (\epsilon_{n(t)},\brown\refl(t))\in S , \qquad t \ge 0\end{equation} 
is a Walsh's spider process with parameter $\beta$. This formula says that we place each excursion of a reflected Brownian motion on the $i$th ray with probability $\beta_i$ independently of all other excursions and of the reflected Brownian motion. 
 
\newcommand{\scan}{S_{\mathsf {can}}}
\newcommand{\cscan}{\mathfrak C(\scan)}

 One can also think of $W_\beta$ as a $\ka$ dimensional process. To this end, using \eqref{defwal},  for $i \in \kad$
 we define $ X_i(t) \coloneqq \brown\refl(t) $ whenever $n(t)=i$, and $X_i(t)=0$ otherwise. Then we can identify $W_\beta $ with the vector $\seq{X_i}$; we note that all possible values of this random vector lie in 
\[ \scan \coloneqq \bigcup_{i\in \kad} \{v \in \R^\ka\colon v =x \mathfrak e_i \text{ for some } x\ge 0, i\in \kad\}\]  
 where $\seq{\mathfrak e_i}$ is the standard basis of $\R^\ka$, and `can' stands for `canonical'.

In this spirit, Barlow, Pitman and Yor  \cite{barlow} have characterized Walsh's process as follows.
\begin{thm}\label{thm:barlowa}  Let $\brown$ be a one-dimensional standard Brownian motion, and $\sigma(\brown)\coloneqq \inf\{t\geq 0\colon \brown(t)=0\}$ be the first time $\brown$ hits zero. Also, let \[ X = \seq{X_i}\] be a strong Markov process with values in $\scan $, and let the first time it hits $\zero$ be denoted $\sigma(X)$. 
Then $X$ is Walsh's process with parameter  $\beta$ iff
\begin{itemize}
\item[(a)] for each $x\ge 0$ and $i\in \kad$ and a Borel subset $\mc B$ of the space of continuous paths
\begin{multline*}
\Pb \big\{(X_i(t)\1_{\{t\leq \sigma(X)\}})_{t\geq 0}\in \mc B \big|X(0)=x\mathfrak e_i \big\}\\
=\Pb \big\{(\brown(t)\1_{\{t\leq \sigma(\brown)\}})_{t\geq 0}\in\mc B \big|\brown(0)=x\big\},
\end{multline*}
\item[(b)] for each $t>0$, $i\in \kad$ and a Borel set $\mathpzc B \subset (0,\infty)$
\[\Pb\{X_i(t)\in \mathpzc B | X(0)=\zero \}= \beta_i \Pb\left \{|\brown(t)|\in \mathpzc B \big | \brown(0)=0\right \}.\]
\end{itemize}
\end{thm}

Condition (a) in the theorem says that up to the time it hits $\zero$, $W_\beta$ started at an $i$th ray is undistinguishable from the standard Brownian motion on this ray, and (b) says that $W_\beta$ started at the origin is at a time $t>0$ on the $i$th ray with probability $\beta_i, i \in \kad$. 

Next, we turn to a martingale characterization of Walsh's spider process. We begin by stating that, since the distance of Walsh's process from the origin is a reflected Brownian motion,  
the  local time of the former process at $0$ can be defined as the almost sure limit 
\begin{equation}\label{andriya}
L(t)\coloneqq \lim_{\varepsilon\to0+}\tfrac{1}{2\varepsilon}\int_0^t\1_{\{0<|W_\beta(s)|\leq \varepsilon\}}\ud s, \quad t \ge 0. 
\end{equation}
We remark in passing that both Walsh's process and reflected Brownian motion have zero sojourn at $0$ with probability $1$. Hence the integral 
above could have been replaced by $\int_0^t\1_{\{ |W_\beta(s)|\leq \varepsilon\}}\ud s$. For reasons that will become clear later,  it is, however, convenient to exclude $0$ from the integration (see, for example, the calculation leading to \eqref{wika:1}).

The following martingale characterization of Walsh's process was first given in \cite{barlow}, see also \cite{PavlyukevichPilipenko2022}.
\begin{thm}\label{thm:Walsh}
Let $X=\seq{X_i}$ and $L$ be two continuous processes, the first with values in 
$\R^\ka$, the second with values in $\R$, and  let  $\mathcal{F}^X$ be the filtration generated by $X$. Then $X$ is a Walsh's process with parameter $\beta$, and $L$ is its local time at   $0$ defined by \eqref{andriya}  iff
\begin{itemize}
\item [(1) ] $X_i(t)\geq 0, 1\leq i\leq \kad $, and $X_i(t)X_j(t)=0, i,j\in \kad, i \not = j$ for $t\geq 0$;
\item [(2) ] $L$ is a.s.\ nondecreasing $\mathcal{F}^X$-adapted process with $L(0)=0$ and
\[ 
\int_{[0,\,\infty)}\1_{\{X(s)\neq \zero \}}\, {\rm d} L(s)=0\quad\text{a.s.};
\]
\item [(3) ] the processes $M_i, i\in \kad $ defined by
\begin{equation}
\label{e:Mnu}
M_i(t)\coloneqq X_i(t)-\beta_i L(t),\quad t\geq 0
\end{equation}
are continuous square integrable martingales with respect to $\mathcal{F}^X$ with predictable quadratic variations
\begin{equation}
\label{e:brM}
\langle M_i\rangle (t)=\int_0^t \1_{\{X_i(s)>0\}} \ud s, \qquad t \ge 0;
\end{equation}
\item [(4) ] $\int_0^\infty \1_{\{X(s)=\zero\}} \ud s=0$ {\rm a.s.}
\end{itemize}

\end{thm}

Let us mention two immediate corollaries  to Theorem \ref{thm:Walsh}.  First, let 
$W_\beta$ be Walsh's process  with parameter $\beta$, and let 
$\Psi$ map $\kad $ onto $\mathpzc N \coloneqq \{1,\dots,\mathpzc n\}$, where $\mathpzc n \le \ka$ is a natural number. 
The process $(X_j)_{j\in \mathpzc N}$ defined by  
     \[
     X_j(t)=\sum_{\{i\colon \Psi(i)=j\}}  (W_\beta)_{i}(t), \qquad j\in \mathpzc N, t \ge 0
     \]
where $(W_\beta)_{i}$ is the $i$th coordinate of $W_\beta$, is then Walsh's process with the new  parameter $\widetilde \beta = (\widetilde \beta_j)_{j\in \mathpzc N}$ obtained as follows $\widetilde \beta_j \coloneqq \sum_{\{i\colon \Psi(i)=j\}} \beta_i, j \in \mathpzc N$. Moreover, the local times of $X$ and $W_\beta$ at $\zero$ coincide. Thus, in the specific instance of $\alpha=\gamma =0$, this establishes the fact suggested already by formula \eqref{let:1}: by lumping edges of the graph in which Walsh's process has its values, we obtain another Walsh's process, on a `smaller' graph. 

In particular, \begin{equation}\label{igrek} X(t)\coloneqq {\sum_{i\in \kad}}(W_\beta)_i(t), \qquad t\geq 0,\end{equation} is a reflected Brownian motion.
Notably, this process can be represented as
     \[
     X(t)=\brown(t)+L(t), \qquad  t\geq 0,
     \]
     where $\brown\coloneqq \sum_{i\in \kad} M_i$ (the martingales $M_i,i\in \kad$ are defined in point (3) above) 
     is a Brownian motion. 
\subsubsection{A theorem of Bayraktar et al.}\label{sec:bay}
Our next and final goal in this section is a representation of Walsh's process that emerged from the analysis of so called \emph{multi-armed bandits}, see  \cite{mandelbaum1987continuous, kaspi1995levy, barlow2000variably, bayraktar}. 
By way of introduction, we recall that any square integrable continuous martingale $M$ is a time-changed Brownian motion:  there exists a Brownian motion $\brown$ such that $M=\brown\circ \langle M\rangle$, where $\langle M \rangle $ is the quadratic variation of $M$ (see e.g. \cite{ikeda}*{Sec. 2.7, Thm. 7.2} for this statement in its more general form). In this context  \eqref{e:brM} says that for each $i\in \kad$, there exists a Brownian motion $\brown_i$ 
such that 
\[
X_i(t)=\brown_i \circ \langle M_i \rangle  (t) + \beta_i L(t), \qquad t \ge 0.
\]
The last expression in turn reveals that the paths of the process $X_i$ are solutions to the Skorokhod reflection problems with noise given by the paths of $ \brown_i \circ \langle M_i \rangle (t), t\geq0$. Since solutions to these problems are determined uniquely and are given by \eqref{srp:sol}, we see that 
\[
X_i(t)=\brown_i \circ  \langle M_i \rangle (t)  + 
\max_{s\in [0,t]} \left (- \brown_i  \circ \langle M_i \rangle (t)\vee 0 \right ), \qquad t\ge 0. 
\]
As an immediate consequence, 
\begin{equation}\label{eq:BarlowToBayraktar}
    X_i =\brown_i\refl\circ \langle M_i \rangle , \qquad i \in \kad, 
\end{equation}
where, as before,  $\brown_i\refl(t)\coloneqq \brown_i(t) +L_i(t), t \ge 0, i \in \kad $ and $ L_i(t)\coloneqq  
\max_{s\in [0,t]}$ $(-\brown_i(s)    \vee 0 ), t \ge 0,$ is the Skorokhod reflection of $\brown_i, i \in\kad $. The Brownian motions $\brown_i, i\in\kad$ are independent due to Knight's theorem, see \cite[Theorem 1.9, Chapter V]{revuz}.

Moreover, \eqref{e:Mnu} reveals that each $X_i$ is a semimartingale, as the sum of the martingale $M_i$ and the nondecreasing process $\beta_i L$. Hence, the Tanaka formula (\cite{karatzas}*{p. 220} or \cite{rogers2}*{p. 96}) applies to give  
\begin{align}
X_i(t) - X_i(0) &= |X_i(t)| - |X_i(0)| \nonumber \\&= \int_0^t \1_{\{X_i(s)>0\}} \ud (M_i(s) + \beta_i L(s)) + L_i(t)\nonumber \\\label{potrzebne}
&=\int_0^t \1_{\{X_i(s)>0\}} \ud M_i(s) + L_{i} (t) \end{align}
where in the first step we have used the fact that $X_i$ is nonnegative and in the last --- the fact that $L$ increases only when $X_i=0$; $L_i$ is the symmetric local time of $X_i$ at $0$. Also, \( \EE  (\int_0^t  1_{\{X_i(s)>0\}} \ud M_i(s) -  \int_0^t 1 \ud M_i(s))^2 = \EE (\int_0^t  1_{\{X_i(s)\le 0\}} \ud M_i(s) )^2 \) and this,  by Ito isometry, 
equals \[ \EE \int_0^t \1_{\{X_i(s)\le 0\}}\ud \langle M_i(t)\rangle = \int_0^t  \1_{\{X_i(s)\le 0\}} \1_{\{X_i(s)>0\}} \ud s =0, \qquad t \ge 0.\]
 with the previous-to-last equality following by \eqref{e:brM}. 
Hence, almost surely,  $\int_0^t 1_{\{X_i(s)>0\}} \ud M_i(s)= \int_0^t 1 \ud M_i(s)= M_i(t)-M_i(0)$ and so \eqref{potrzebne} renders $X_i(t) - X_i(0)= M_i(t) - M_i(0) + L_i(t)$. Since $X_i(0)=M_i(0)$, \eqref{e:Mnu} now proves that $\beta_i L $ and $L_i$ coincide and we can thus write 
\begin{align*}
\label{eq:balance_locTime_Walsh's process}
     \beta_i L(t)&= \lim_{\varepsilon\to0+} \tfrac{1}{2\varepsilon} \int_0^t \1_{\{0<|X_i(s)|\leq \varepsilon\}} \ud s 
     =\lim_{\varepsilon\to0+} \tfrac{1}{2\varepsilon} \int_0^t \1_{\{0<|B_i\refl \circ \langle M_i \rangle (s)|\le \eps \}}  \ud s\\
   &= \lim_{\varepsilon\to0+} \tfrac{1}{2\varepsilon} \int_0^{\langle M_i \rangle (t) } \1_{\{0<|B_i\refl(s))|\leq \varepsilon\}} \ud s\\& = L_i \circ \langle M_i \rangle (t)= L_i\left (\int_0^t \1_{\{X_i(z)>0\}}  {\rm d}z\right ), \qquad t \ge 0. 
\end{align*}
It follows that, almost surely, for all $i,j\in \kad$,  
\begin{equation}
\label{eq:ltimes}
L(t)=\frac{L_i(\int_0^t \1_{\{X_i(z)>0\}} {\rm d}z)}{\beta_i}=\frac{L_j(\int_0^t \1_{\{X_j(z)>0\}} {\rm d}z)}{\beta_j}, \qquad t\geq 0.   
\end{equation}

The so-obtained formula suggests that a copy of Walsh's process can be constructed as a concatenation of several independent Brownian motions balanced locally via ties similar to these visible above. The following theorem establishes that this is indeed the case; the time $T_i(t), t \ge 0$ featured in it is a  counterpart of $  
\int_0^t \1_{\{X_i(s)>0\}} \ud s, t \ge 0$, considered above, and should be interpreted as the time the process spends at the $i$th ray of $\scan$. As already mentioned, the original idea of the theorem comes from the theory of multi-armed bandits, see  \cite{mandelbaum1987continuous, kaspi1995levy, barlow2000variably, bayraktar}; the statement below is a version of a more general result \cite[Theorem 2.1]{bayraktar}. 
  \begin{thm} 
    \label{thm:baryaktar} Let $\beta=\seq{\beta_i}$ be a vector of positive numbers such that $\sum_{i\in \kad} \beta_i=1$. Given
 a standard $\ka$-dimensional Wiener process  $\brown=\seq{\brown_i}$ with $\brown(0)\in S_{\mathsf{can}}$, we define 
\[ L_i(t)\coloneqq \max_{s\in[0,t]}(-\brown_i(s)\vee 0) \mqquad{ and } \brown\refl_i(t)\coloneqq \brown_i(t)+L_i(t), \quad i\in \kad, t \ge 0.\]     
    Then there is a unique collection of continuous, non-negative and non-de\-creas\-ing processes $T_i, i\in \kad$ such that almost surely
    \begin{equation}
        \label{eq:balance_locTime_Walsh's process1}        \frac{L_i(T_i(t))}{\beta_i}=\frac{L_j(T_j(t))}{\beta_j},\qquad  i,j\in \kad, t\geq 0, 
    \end{equation}
and $\sum_{i\in \kad}T_i(t)=t, t\geq 0$.
    Moreover, the process 
    \begin{equation}
                \label{eq:repr_Bayr}
        X\coloneqq  \seq{\brown\refl_i\circ T_i} 
        \end{equation}
is then Walsh's process  with parameter $\beta.$ 
    \end{thm}

A couple of remarks are here in order. First of all, 
only one coordinate of the Walsh's process $X$ defined in \eqref{eq:repr_Bayr} can be non-zero, $X$  has no sojourn at $\zero$, and the processes  $T_i$ from Theorem \ref{thm:baryaktar} satisfy $\sum_{i\in \kad}T_i(t)=t, t\geq 0$. Therefore, almost surely,
\[
T_i(t)=\int_0^t \1_{\{X_i(s)>0\}} \ud s,\quad t\geq 0,
\]
so that indeed $T_i(t)$ is the time spent at the $i$th ray up to time $t$. Also, the theorem shows that the local time $L(t)$ defined in Theorem \ref{thm:Walsh} is tied to the local times of Theorem \ref{thm:baryaktar} 
by 
\[
    L(t)= \frac{L_i(T_i(t))}{\beta_i}, \qquad i\in\kad, t\geq 0, \]
which agrees with \eqref{eq:ltimes}; moreover,      
   \begin{align}
        \label{eq:loctime_Bayr}
    L(t)&= \sum_{i\in\kad}  {L_i(T_i(t))}, \qquad t \ge 0.
    \end{align} 
 Finally,  formula \eqref{eq:repr_Bayr} is richer in meaning than \eqref{eq:BarlowToBayraktar}; whereas in the latter existence of a Brownian motion that underlies $X_i$ is claimed for each $i\in \kad$, the former asserts moreover that these Brownian motions are independent. 
 

Most importantly, the process $X$ of \eqref{eq:repr_Bayr}
 can be considered as a multi-parameter time change  of the multidimensional reflected Brownian motion    $\seq{\brown\refl_i}$. In fact, it was proved in  \cite[Lemma 3.3, Proposition 3.2]{bayraktar} that for \emph{multi-parameter stopping times}, that is, a class of processes that includes $\seq{T_i}$, an analogue of the 
the strong Markov property holds.

A similar argument shows that for any $t\geq 0$ the processes
\[\{\brown_i(T_i(t)+s)-\brown_i(T_i(t)), s\geq 0\}, \qquad i \in \kad \]
are mutually independent Brownian motions that are also independent of the processes describing the past: 
\[\{B_i(T_i(t)\wedge s), s\geq 0\}, \qquad i \in \kad. \]
Further, if a non-negative exponential random variable $\zeta$ is  independent of $\brown,$ then 
 the processes
\begin{equation}
    \label{eq:newBM_Bayrak}
    \{\brown_i(T_i(\sigma)+s)-\brown_i(T_i(\sigma)),s\geq 0\}, \qquad i \in \kad \end{equation}
where $\sigma\coloneqq \inf\{t\geq 0\colon L(t)=\zeta\},$ are  mutually independent Brownian motions that are also  independent of the processes describing the past 
\[\{\brown_i(T_i(\sigma)\wedge s), s\geq 0\}, \qquad i \in \kad. \]

We complete this section by remarking that 
    filtration generated by Walsh's process  is a delicate subject. Theorem \ref{thm:baryaktar} 
    shows that Walsh's process is a function of Brownian motions. However, the filtration generated by the Walsh process on a graph with more than three rays differs from the Brownian filtration \cite{tsirelson1997triple}.

\subsection{The degenerate case}\label{degenerat} 
So far, in our stochastic description, we have focused on the case where $\beta_i>0, i\in \kad$ and $\sum_{i\in \kad}\beta_i =1$. Since multiplying $F$ of \eqref{loo:1} by a positive constant does not change $F$'s kernel, the second of these conditions is just a normalization that additionally provides a natural interpretation of $\beta_i$s, but this is not the case  with the first of them. Hence, there is the need to clarify what happens when some, but not all, $\beta_i$ vanish.  

Let us thus consider the case where $\beta_i=0$ for all $i $ in a proper subset $\kal$ of $\kad$; without loss of generality we can assume that $\sum_{j\in \kad \setminus \kal}\beta_j =1$ (in other words, by definition, $W_\beta \coloneqq W_{\beta'}$ where $\beta' $ is obtained from $\beta$ by dividing its coordinates by  $\sum_{i \in \kad \setminus \kal} \beta_i$). If started from an $i$th edge with $i \in \kal$, the degenerate $W_\beta$ behaves like a one-dimensional Brownian motion until the time it reaches the origin;  at this instant it chooses, with probability $\beta_j$, to continue its motion on a $j$th edge with $j\not \in \kal$.  From that time on it never returns to the edge where it initially started, neither does it  go to any of the other edges with indices $i\in \kal$, but rather behaves like the Walsh's spider on the subset of the original graph formed by edges with indices $j\in \kad \setminus \kal$. In other words, points of edges with indices $i\in \kal$ are transient for the process, and the process is essentially the Walsh's spider on a smaller star graph with $\#(\kad \setminus \kal)$ edges. 
 
 Interestingly, the transition densities formula \eqref{gestosci} remains valid also in the degenerate case. In particular, for $i\in \kal$, the first line of  the formula says that from the perspective of the $i$th ray the process is that of minimal Brownian motion: once it reaches $\zero$ it is never seen again. At the same time, as expressed in the second line, the $i$th ray is invisible, that is, unreachable,  from the other edges.

\section{The full case of $\mum=0$; processes with continuous paths.}\label{sec:tcom=0}

Having described the process of Walsh, which, to recall, corresponds to the case of $\alpha=\gamma=0$ in \eqref{loo:1}, we now proceed to the general case of nonzero $\alpha$ and/or $\gamma$. In other words, we will now study the processes related to the boundary condition
\begin{equation}\label{contp:1} {\tfrac \alpha 2} f''(\zero) - \sum_{i\in \kad} \beta_i f'_i(0) + \gamma f(\zero) = 0 \qquad f \in \dom{\gen }. \end{equation}	
We start with the case where all betas are positive and $\sui \beta_i =1$.

We have already developed the intuition that increasing $\alpha$ forces the related process to linger at $\zero$ longer, and nonzero $\gamma $ tells of the possibility for the process to be killed, that is, to be undefined after a certain random time spent at $\zero$. The construction presented below confirms this intuition and gives deeper insight into what really happens. This construction, notably, can be carried out in a rather general context, and so, in Sections \ref{sec:sup}--\ref{sec:fno}
 we present this general picture, and specialize to the case of Walsh's process later in Section \ref{sec:sand}, where we provide a description of paths of the process related to \eqref{contp:1}.

\vspace{-0.2cm}
\subsection{Slowing down a Feller process at a point}\label{sec:sup} 

Let $X$ be a Feller process with values in a compact metric space $S$; we assume its trajectories  are  c\`adl\`ag. Moreover, let $a\in S$ be a point such that for $\sigma$ defined as the first time $X$  hits $a$ ($\sigma\coloneqq \inf\{t>0\colon  X(t)=a\}$) we have  
\begin{itemize} 
\item [(a)] $a$ is regular for itself, that is, $ \mathsf E_a \e^{-\lam \sigma}=1, \lam >0$, 
and \item [(b)] $S\ni x\mapsto \mathsf \elam (x) \coloneqq \mathsf E_x \e^{-\lam \sigma}\in [0,1]$ is continuous for all $\lam >0$. \end{itemize}
It is well-known \cite{blum}*{pp. 84--93} that even under less stringent assumptions (condition (b) is introduced in \cite{blum} much later, on p. 140) there is a local time, say,  $L$, of $X$ at $a$,  that is, a continuous additive functional (CAF) that may grow only when $X$ visits $a$. 
We recall that in fact there are many such local times, but that all of them are equivalent up to a multiplicative constant. For the present purposes it is irrelevant which local time at $a$ is used. 

To proceed, using $L$ and given $\alpha >0$, we define another CAF by  
\begin{equation}\label{aalpha}
A_{\alpha} (t) \coloneqq \begin{cases}t+\alpha L(t), & t\in [0,\tau),\\
\tau+\alpha L(\tau-), & t\in [ \tau, \infty),
\end{cases}
\end{equation}
 where $\tau$  is the life-time of $X$ (see also the remark at the end of this section).
Then, each trajectory of $A_\alpha$ is strictly increasing and continuous on $[0,\tau)$, and we have $A_\alpha(0)=0$.  It is moreover well-known that for each $t\geq 0$, the $[0,\infty]$-valued random variable $A_\alpha^{-1}(t)$ is a stopping time with respect to the filtration generated by $X$. It follows (see \cite[Theorem 10.10, Chapter 10.5]{dynkinmp} or \cite[Chapter II.6]{gands}) that $\mathcal S_\alpha X$ given by  \begin{equation}\label{lepki} \mathcal S_\alpha X (t) \coloneqq \begin{cases} X\circ A_\alpha^{-1}(t), & A_\alpha^{-1}(t)<\infty, \\ \text{undefined}, & A_\alpha^{-1}(t)=\infty, \end{cases} \end{equation} 
is a time-homogeneous Markov process. As we shall see in Section \ref{sec:fprop},  $\mathcal S_\alpha X$ has the Feller property.


For now, we will argue that $\mathcal S_\alpha X$ `slows down' when at $a$ and only at $a$, and spends positive time at this point. First of all, whenever $X$ is away from $a$, $L$ does not increase. Thus $A_\alpha$ increases linearly with slope $1$ and so does $A_\alpha^{-1}$.  This means that while away from $a$, $\mathcal S_\alpha X$ behaves precisely as $X$. However, when at $a$, $A_\alpha$ grows faster than linearly with slope $1$ and so $A_\alpha^{-1}$ grows slower than linearly with slope $1$, that is, slower than the regular time, as claimed.  This analysis shows  that excursions of $\mathcal S_\alpha X$ from $a$ coincide with excursions of $X$; it is only the parts of the trajectories where $X$ is at $0$ that are `stretched out'.

Turning to the fact that $\mathcal S_\alpha X$ spends a positive time at $a$, we write
\begin{align}
    \int_0^t\1_{\{\mathcal S_\alpha X(s)=a\}}\ud s&= \int_0^t\1_{\{X(A_\alpha^{-1}(s))=a\}} \ud s =
    \int_0^{A_\alpha^{-1}(t) }\1_{\{X(u)=a\}} \ud A_\alpha(u) \nonumber \\ &= \int_0^{A_\alpha^{-1}(t) }\1_{\{X(u)=a\}} \ud (u+ \alpha L(u)) \nonumber \\&=
    \int_0^{A_\alpha^{-1}(t) }\1_{\{X(u)=a\}} \ud u + \alpha L\circ A_\alpha^{-1}(t),  \label{dzis:1}
\end{align}
and note that the second term is positive for all $t\ge \sigma$. This establishes the claim.

We prove also that \[\mathcal S_\alpha L \coloneqq L\circ A_\alpha^{-1}\] is a continuous additive functional for $\mathcal S_\alpha X$. To this end, since the process $\int_0^t \1_{\{\mathcal S_\alpha X(s) =a\}} \ud s, t \ge 0$ is obviously a continuous additive functional, it suffices to show that for a constant, say, $c$, 
\begin{equation}\label{dzis:2} L\circ A_\alpha^{-1} (t) = c \int_0^t \1_{\{\mathcal S_\alpha X(s) =a\}} \ud s, \qquad t \ge 0, \end{equation}
almost surely with respect to any initial distribution. To prove this formula, in turn, we suppose first that $X$ has zero sojourn time at $a$, that is, by definition, that $\int_0^\infty \1_{\{ X(s) =a\}} \ud s=0$, almost surely as above. In this case, the first term in the last line of \eqref{dzis:1} is $0$, and then,  almost surely, 
\begin{equation}
    \label{eq:loctime_slow}
 \int_0^t\1_{\{\mathcal S_\alpha X(s)=a\}} \ud s = \alpha L\circ A_\alpha^{-1}(t), \qquad t \ge 0.
\end{equation}
This means that  \eqref{dzis:2} holds with $c \coloneqq \alpha^{-1}$. Also, if $X$ does not have zero sojourn at $a$, then for some $t >0$ the probability of  the event  $\{\int_0^t\1_{\{X(s)=a\}} \ud s>0\}$ is not zero.  As a result, $\int_0^t\1_{\{X(s)=a\}} \ud s, t \ge 0$ is a nontrivial local time, and so, by the uniqueness result mentioned already above, all other local times
are constant multiples of this one. In particular, $L(t) = c_1 \int_0^t\1_{X(u)=a} \ud u, t \ge 0$ for a certain $c_1>0$. But then, 
the first term in the last line of \eqref{dzis:1} is $c_1^{-1} L \circ A_\alpha^{-1} (t)$; this shows that \eqref{dzis:2} holds with $c\coloneqq (c_1^{-1} +\alpha)^{-1}$, completing the proof.

Before completing this section, we come back to formula \eqref{aalpha} to comment on its second line and its influence on the definition of the local time  $\mathcal S_\alpha L$.  Namely, given a Markov process that is not honest, it is tempting to agree that its CAF is undefined whenever the process is undefined. However, such an agreement leads to technical difficulties in interpreting the CAF's functional equation, and complicates further considerations. As it turns out (compare e.g. \cite{bandg}*{p. 149}), it is  thus more convenient to define the CAF, say, $K$, beyond the lifetime $\tau$ of the process by  
\begin{equation}
    \label{eq:agreement_LT}
    K(t):=K(\tau-), \qquad t\geq \tau,
\end{equation}  
so that, for $t\ge \tau$, $K(t)$ can be thought of as the CAF accumulated during the lifetime. 
Definition \eqref{aalpha} is in line with this more convenient approach. Throughout the paper,  we  will keep defining CAFs, and local times in particular, beyond the lifetimes of their processes as in \eqref{eq:agreement_LT}; see the next section for example.

\subsection{Killing a Feller process at a point}\label{sec:kaf} 

A Feller process can also be killed. To elaborate, let again $X$ be a process satisfying assumptions (a)--(b) of Section \ref{sec:sup} and $L$ be its local time at $a$. Given an independent exponentially distributed random variable  $\zeta$ with parameter $\gamma$, we let $\tau \coloneqq \inf \{t \ge 0\colon L(t)=\zeta\}$.  Then the process 
\begin{equation} \label{killedg} \mathcal K_\gamma X (t) \coloneqq \begin{cases} X(t), &t < \tau, \\
\text{undefined}, & t \ge \tau \end{cases}\end{equation}
will be referred to as \emph{killed} at $\tau$. It is well known that $\mathcal K_\gamma X$ is a strong Markov process. Moreover, it turns out that  \[\mc K_\gamma  L (t) \coloneqq \begin{cases} L(t), &t < \tau, \\
L(\tau-), & t \ge \tau  \end{cases}\] is a local time of $\mathcal K_\gamma X$ at $a$.

It is interesting that the transformations $\mc S_\alpha$ and $\mc K_\gamma $ in a sense commute. Whereas, in general, processes $\mc S_\alpha \mc K_\gamma X $ and  $ \mc K_\gamma \mc S_\alpha X $ obviously differ, they turn out to be the same if extra care is taken in their construction. Suppose namely that an exponential random variable $\zeta$ is fixed and  $\mc K_\gamma \mc S_\alpha X$ is defined to be equal to $\mc S_\alpha X$ up to (but not including)
the time when $\mc S_\alpha L =\zeta$, and undefined 
from that time on.  Also, let $\mc S_\alpha \mc K_\gamma X$ be the process $\mc K_\gamma X$ slowed down with the help of the local time $\mc K_\gamma L$  defined using  the same variable $\zeta$. Then, it is easy to see that $\mc S_\alpha \mc K_\gamma X= \mc K_\gamma \mc S_\alpha X $. 

\subsection{A criterion for Feller property}\label{sec:fprop} Our next aim is to prove that the strong Markov processes $\mathcal S_\alpha X$ and $\mathcal K_\gamma X$  inherit Fellerian nature of the original process.  To this end, in this section, we introduce an auxiliary criterion for Feller property, that is,  Proposition \ref{lem:feller} further down.

Let $\rla, \lam >0$ be the resolvent of a process $X$ satisfying conditions (a) and (b) of Section \ref{sec:sup}, that is, let 
\[\rla f(x)= \mathsf E_x \int_0^\infty \e^{-\lam t} f(X(t))\ud t, \qquad x\in S, \lam >0, f\in \coss .\] 
{For simplicity of exposition and without losing any essential examples, we assume furthermore that $a$ is an accumulation point of $S$ and the lifetime of $X$ is no smaller than $\sigma$, so that $X$ is defined at least to the time when it first reaches $a$.} Also, let the \emph{minimal} process $X_{\mathsf {min}}$ be defined as being identical to $X$ up to (but not including) time $\sigma$ and undefined from $\sigma$ on.  The strong Markov property of $X$ allows then showing that the resolvent $\rla^0, \lam >0$ of $X_{\mathsf {min}}$ is related to the resolvent of $X$ as follows 
\begin{equation}\label{pklucz} \rla f (x) = \rla^0 f (x) + \rla f(a) \, \elam (x), \qquad x \not = a, f \in \coss .\end{equation}
Indeed, since $X_{\mathsf {min}}$ and $X$ are identical before $\sigma$, and $X$ is strong Markov, we can write 
\begin{align*}
\rla f(x) &= \mathsf E_x  \int_0^\sigma \e^{-\lam t} f(X_{\mathsf{min}}(t)) \ud t + \mathsf E_x  \e^{-\lam \sigma}  \mathsf E_{X(\sigma)} \int_0^\infty \e^{-\lam t} f(X(t)) \ud t, \end{align*}
for $x\not =a$ and $f\in \coss$. This gives \eqref{pklucz} because $X_{\mathsf{min}}$ is undefined after $\sigma$ and $X(\sigma) = a$. Moreover, formula \eqref{pklucz} can be extended by continuity to $x=a$, but then assumption (a) on the original process forces $\rla^0f(a)=0, f \in \coss$; keeping this in mind we have \begin{equation}\label{klucz} \rla f =  \rla^0 f  + \rla f(a) \, \elam, \qquad  f \in \coss, \lam >0.\end{equation}

It should be stressed here that $\rla^0, \lam >0$ is not a Feller resolvent in the sense introduced in Section \ref{pofs}: for $f $ with $f(a)\not =0$ we cannot have $\lil \lam \rla^0 f=f$ because $\lam \rla^0 f(a)=0$ for all $\lam >0$.  Nevertheless, by \eqref{klucz}, for $f \in \mathfrak \cosso\coloneqq \{f\in \coss \colon f(a)=0\}$, 
\[ \|\lam \rla^0 f - f\|\le \|\lam \rla f - f\| + \|\elam\| |\lam \rla f(a) - f(a)| \]
and so $\lil \lam \rla^0 f =f$ because $\|\elam\|=1$ and $\rla, \lam >0$ is a Feller resolvent. Moreover, $\rla^0$ maps $\coss $ into $\mathfrak \cosso$. As a result, there is a strongly continuous semigroup in $\mathfrak \cosso$ that describes the minimal process. We note in passing that \eqref{klucz} implies 
\begin{equation}\label{kluczyk} 
\elam = \mathsf 1_S - \lam \rla^0 \mathsf 1_S.\end{equation}
Indeed, there is at least one honest process that is identical to $X_{\mathsf {min}}$ before $\sigma$ (for example, the process that stays at $a$ from $\sigma$ on), and for such a process we have $\lam \rla \mathsf 1_S =\mathsf 1_S$; multiplying both sides of \eqref{klucz} with $f=\mathsf 1_S$ by $\lam$ yields \eqref{kluczyk}.

It is an equally important point, however, that the procedure described above can be reversed: if $\wt X$ is a strong Markov process that coincides with the minimal process up to $\sigma$, then the resolvent $\wrla, \lam >0$ of $\wt X$  is related to $\rla^0, \lam >0$ as in \eqref{kluczt}  (and the reader may want to compare this result with \eqref{loo:4}):
\begin{equation}\label{kluczt} \wrla f =  \rla^0 f  + \wrla f(a) \, \elam, \qquad  f \in \coss, \lam >0.\end{equation}
This allows deducing Fellerian nature of such an $\wt X$ as follows.

\begin{proposition}\label{lem:feller}
Suppose that $\wt X$ is a strong Markov process that is identical to $X_{\mathsf {min}}$ for  $t< \sigma$.  Then $\wt X$ is a Feller process provided that  
\begin{equation}\label{warek} \lil \lam \wrla f(a) = f(a), \qquad f \in \coss.\end{equation}  
\end{proposition}
\begin{proof}  By assumption (b) of Section \ref{sec:sup}, the right-hand side of \eqref{kluczt} is a continuous function whenever $f$ is,  that is, each $\wrla$ maps $\coss $ into itself.  Moreover, since $\wt X$ is a Markov process, these operators satisfy the Hilbert equation. They are also clearly nonnegative and condition 3. of the definition of Feller  resolvent is fulfilled. 

We are thus left with showing that $\lil \lam \wrla f =f, f \in \coss$. Starting with the case where $f=\mathsf 1_S$ we note that condition  \eqref{kluczt}, when combined with  \eqref{kluczyk}, renders 
\( \|\lam \wrla  \mathsf 1_S - \mathsf 1_S\| = (1 - \lam \wrla \mathsf 1_S(a))\|\ell_\lam\| = 1 - \lam \wrla \mathsf 1_S(a)\) and this converges to $0$ as $\lam \to \infty$, by assumption \eqref{warek}. Therefore, it suffices to show that $\lil \lam \wrla f =f, f\in \mathfrak \cosso$. However, for such $f$, $\lil \lam \rla^0 f =f $, and so the desired result is a consequence of 
 \eqref{kluczt} and \eqref{warek}; note again that $\|\elam\|=1$. 
 \end{proof}

In practice, to show \eqref{warek} it suffices to check that 
\begin{equation}\label{wareczek} \lim_{t\to 0+}f(\wt X(t)) =f(a), \qquad f\in \coss\end{equation}  almost surely conditional 
on  $\wt X (0)=a$. Indeed, if this conditions is satisfied, we also have    $\lil \lam \czn \e^{-\lam t}f(\wt X (t))\ud t = f(a)$ (almost surely as above) and then  \eqref{warek} follows by the Lebesgue Dominated Convergence Theorem. In particular, \eqref{wareczek} is fulfilled if the lifetime of the process starting at $a$ is positive with probability $1$.

\subsection{Fellerian nature of slowed down and killed processes} \label{sec:fno}

Proposition \ref{lem:feller} allows proving that the slowed down process $\mc S_\alpha X$ of \eqref{lepki} has Feller property. For, as we have already stated, $\mc S_\alpha X$ is strong Markov. Moreover, for $t<\sigma$ it is identical to $X_{\mathsf {min}}$ and so \eqref{kluczt} is in force. Finally, since $X$ has c\`{a}dl\`{a}g paths, so does $\mc S_\alpha X$ and thus condition \eqref{wareczek} is satisfied.

The same proposition can be used to show that also the killed process $\mc K_\gamma X$ has Fellerian nature. Indeed, first of all, $\mathcal K_\gamma X$ is strong Markov and identical to $X_{\mathsf{min}}$ for $t< \sigma$; hence \eqref{kluczt} is again in force. Moreover, by assumption, $\lim_{t\to 0+}f(X(t)) =f(a)$ for all $f\in \coss$  almost surely conditional on $X (0) = a$.  However, $\tau >\sigma$ almost surely, because $\zeta$ is almost surely positive and a positive value cannot be reached by a continuous process starting at zero, that is, by the local time of $X$ at $a$. Hence, before $\tau >0$, each path of $\mathcal K_\gamma X$ is the same as that of $X$, and this yields    $\lim_{t\to 0+}f(\mathcal K_\gamma X(t)) =f(a)$ almost surely as above, so that \eqref{wareczek} is satisfied.

\subsection{Slowing down and killing the Walsh's process}\label{sec:sand} 

Let us come back to the Walsh's process $W_\beta$ as defined in \eqref{eq:repr_Bayr}. This process can be slowed down at $\zero$: to this end we define 
\begin{equation}\label{stickywalsh} W_{\alpha,\beta} \coloneqq W_{\beta}\circ A_\alpha^{-1}\end{equation} 
where $A_\alpha^{-1}$ is the inverse of $A_\alpha$ of  \eqref{aalpha} with $L$ of \eqref{eq:loctime_Bayr}
being the symmetric local time of $W_\beta$ at $\zero$.  We know that the so-defined $W_{\alpha,\beta}$ is a Feller process. Indeed, $W_\beta$ satisfies conditions (a)--(b) of Section \ref{sec:sup}; we have $\scan \ni x\mathfrak e_i  \mapsto \mathsf E_{x\mathfrak e_i} \e^{-\lam \sigma} = \e^{-\slam x}, x\ge 0, i \in \kad$ (this is just the restatement of the definition of $\elam^0 $ introduced in \eqref{lifetime}), where $\sigma$ is the first time $W_\beta$ reaches the graph center $\zero$. Hence, the argument of the first part of Section \ref{sec:fno} is applicable.  

This process, in turn, can be
killed with the help of an exponential random variable $\zeta$ with parameter $\gamma$ that is independent of the process $W_{\alpha,\beta}$. The procedure is summarized in
\begin{equation}\label{wkilled}
W_{\alpha,\beta,\gamma}(t)\coloneqq \begin{cases}
    W_{\alpha,\beta} (t),& L\circ A_\alpha^{-1} (t)<\zeta,\\
    \text{undefined},&  L\circ A_\alpha^{-1} (t)\geq \zeta.
\end{cases}
\end{equation}
The same result can be obtained by reversing the order of slowing down  and killing, as follows. First we introduce 
\begin{equation}\label{wobg} W_{0,\beta, \gamma} (t) \coloneqq \begin{cases} W_{\beta}(t), & L (t)< \zeta,\\
 \text{undefined}, & L(t)\ge \zeta, \end{cases}\end{equation}
 where, again, $W_\beta $ is the Walsh's process of \eqref{eq:repr_Bayr}, $L$ is the local time of \eqref{andriya}, and $\zeta$, independent of $W_\beta$, is an exponential random variable with parameter $\gamma$. Then,  we slow down $W_{0,\beta,\gamma}$ at $0$: 
 \begin{equation} \label{wabg}  W_{\alpha,\beta, \gamma} (t) \coloneqq W_{0,\beta, \gamma} \circ A_\alpha^{-1} (t), \qquad t \ge 0. \end{equation}
It is easy to check, as in the general case, that both procedures lead to the same process. Moreover,  since $L\circ A_\alpha^{-1}$ is a local time of $W_{\alpha,\beta}$, formula \eqref{wkilled} is a particular case of \eqref{killedg}. Thus, we conclude that 
$W_{\alpha,\beta,\gamma}$ has Fellerian nature.

We end this section by noting that $L\circ A_\alpha^{-1}$ is the symmetric local time of $W_{\alpha,\beta}$, that is, almost surely,  

\begin{equation}\label{wika:1} L\circ A_\alpha^{-1} (t) =  \lim_{\varepsilon\to0+}\tfrac{1}{2\varepsilon}\int_0^t\1_{\{0<|W_{\alpha,\beta}(s)|\leq \varepsilon\}}\ud s, \quad t \ge 0. 
\end{equation}
Indeed, 
\begin{align*}
\int_0^t\1_{\{0<|W_{\alpha,\beta}(s)|\leq \varepsilon\}}\ud s& = \int_0^t\1_{\{0<|W_{\beta}\circ A_\alpha^{-1} (s)|\leq \varepsilon\}}\ud s \\&= \int_0^{A_\alpha^{-1}(t)}\1_{\{0<|W_{\beta}(s)|\leq \varepsilon\}}\ud A_\alpha (s)
\\&= \int_0^{A_\alpha^{-1}(t)}\1_{\{0<|W_{\beta}(s)|\leq \varepsilon\}}\ud (s + \alpha L(s)) \\&= \int_0^{A_\alpha^{-1}(t)}\1_{\{0<|W_{\beta}(s)|\leq \varepsilon\}}\ud s, 
\end{align*}
with the last equality following by the fact that $L$ may increase only if $W_\beta$ is at $\zero$. This shows \eqref{wika:1} by \eqref{andriya}.

\subsection{The generator of $W_{\alpha,\beta,\gamma}$; an approach via Dynkin's operator}\label{sec:dynkin}

We have established that the formula
\[ T_{\alpha,\beta,\gamma} (t) f(v) = \mathsf E_v f(W_{\alpha,\beta,\gamma}(t)), \qquad v\in \scan, f\in \cscan, t \ge 0\] 
defines a Feller semigroup $\rod{T_{\alpha,\beta,\gamma}(t)}$ in $\cscan$ (see Section \ref{sec:bpy} for the definition of $\scan$); the related process is that of Walsh but slowed down at $\zero$ with parameter $\alpha$ and killed with intensity $\gamma$. Our goal will be reached once we show that 
this semigroup is generated by $\gen_{\alpha,\beta,\gamma}$ of  \eqref{gengen}.

To this end, we will use the notion of \emph{characteristic operator of Dynkin}. To recall, under quite general mild conditions  the characteristic operator extends the generator of a Markov process
(see \cite{dynkinmp}*{ \S 3 of Chapter 5}), and coincides with the generator of a Feller process (see \cite{rogers}*{p. 256}). Thus, denoting by $  \mathfrak H_{\alpha,\beta,\gamma}$ the generator of  $\rod{T_{\alpha,\beta,\gamma}(t)}$, we have 
\begin{equation}\label{dynkina} \mathfrak H_{\alpha,\beta,\gamma}f(v) \coloneqq \lim_{\eps \to 0} \frac{\mathsf E_v f(W_{\alpha,\beta,\gamma}(\tau_\eps)) - f(v)}{\mathsf E_v \tau_\eps}, \qquad f\in \dom{\mathfrak H_{\alpha,\beta,\gamma}}, v\in \scan, \end{equation}
where $\tau_\eps$ is the first time $W_{\alpha,\beta,\gamma}$ exits  the open ball centered at $v$ with radius $\eps$.  
We will argue that $\mathfrak H_{\alpha,\beta,\gamma}$ extends $\gen_{\alpha,\beta,\gamma}$. 

For, if $v$ lies away from the center and is thus of the form $v=x\mathfrak e_i$ for some $x>0$ and $i\in \kad$, $\tau_\eps$ is the time needed for a standard Brownian motion to exit from the interval $(x-\eps,x+\eps)$ (as long as $\eps < x$). Indeed, because Brownian paths are continuous, the process cannot touch the graph's center before $\tau_\eps$, and thus before that time it behaves like a one-dimensional Brownian motion on the $i$th edge; in particular, before this time it is neither slowed down nor killed. Hence, $\mathsf E_v \tau_\eps = \eps^2$ (this  can be deduced e.g. from \cite{karatzas}*{Problem 8.14, p. 100}; other ways of reaching this conclusion are listed in \cite{kostrykin2010brownian}*{p. 10}). Moreover, using continuity of Brownian paths again, we see that $W_{\alpha,\beta,\gamma} (\tau_\eps)$ is either $(x-\eps)\mathfrak e_i$ or $(x+\eps)\mathfrak e_i$, and by symmetry both cases happen with the same probability. Thus, the right-hand side of  \eqref{dynkina} reduces to $\lim_{\eps \to 0} \frac {\frac 12 f((x+\eps)\mathfrak e_i) + \frac 12 f((x-\eps)\mathfrak e_i)-f(x\mathfrak e_i)}{\eps^2}$ and for $f\in \dom{\Delta}$ this equals $\frac 12 f''(v)$. 

To treat the case where $v$ is the graph's center, we need to gather some helpful information. 
\begin{itemize} 
\item [(i) ] First of all, Lemma 3.4 in \cite{kostrykin2010brownian} says that, as long as $\gamma =0$, 
\[ \tau_\eps = \sigma_\eps + \alpha L\circ \sigma_\eps = A_\alpha (\sigma_\eps) \]
where $\sigma_\eps$ is the time needed for a reflecting Brownian motion started at $0$ to reach $\eps$ for the first time. 
\item [(ii) ] Using \cite{karatzas}*{Problem 8.14, p. 100} alluded to above we see that $\mathsf E_0 \sigma_\eps = \eps^2$. Moreover, $\mathsf E_0  L\circ \sigma_\eps = \eps$.  Indeed, the second part of \eqref{deflocal} holds for all $t$ for almost all paths and can thus be also used with $t$ replaced by $\tau_\eps$ to yield $\eps = B(\tau_\eps) + L(\tau_\eps)$. Since Brownian motion is a martingale, a~reference to Doob's Optional Sampling Theorem reveals that $\mathsf E_0 B(\tau_\eps)=0$, completing the proof of the claim.

\item [(iii) ]By (i), $L\circ A_\alpha^{-1} (\tau_\eps) = L(\sigma_\eps)$, and so $\mathsf P_\zero (L\circ A_\alpha^{-1} (\tau_\eps)\le \zeta) = \mathsf P_0 (L (\sigma_\eps)\le \zeta)$ --- see  
\eqref{wkilled}.  Since $\zeta $ is  exponentially distributed with parameter $\gamma$, the latter probability equals \[\phantom{===} \czn \gamma \e^{-\gamma t}  \mathsf P_0 (L (\sigma_\eps)\le t)\ud t = \czn \e^{-\gamma t} \ud   \mathsf P_0(L  (\sigma_\eps)\le t)= \mathsf E_0 \e^{-\gamma L(\sigma_\eps)}.\]
The Laplace transform of $L(\sigma_\eps)$, in turn, can be obtained as a special case of formula 2.3.3 on p. 356 in \cite{borodin} (this is formula 2.3.3 on p. 362 in the 2015 corrected printing of this monograph); to this end, put $z=\eps$, $x=r=0$ and take the limit as $\alpha \to 0$, to obtain 
 \[ \mathsf P_\zero (L\circ A_\alpha^{-1} (\tau_\eps)\le \zeta) = \mathsf E_0 \e^{-\gamma L(\sigma_\eps)}= \tfrac 1{1+\eps \gamma};\]
the same formula for the Laplace transform of $L(\sigma_\eps)$ can be found on p. 429 of \cite{karatzas}; a semigroup-theoretic proof of the fact that $L(\sigma_\eps)$ is exponentially distributed with parameter $\eps$ is given in \cite{bobpil1}.
 \end{itemize}
With this information under our belt, we can calculate $\mathfrak H_{\alpha,\beta,\gamma}f (\zero)$. Since, conditional on $L\circ A_\alpha^{-1} (\tau_\eps) \le \zeta$, the process $W_{\alpha,\beta,\gamma}$ started at $\zero$ is at time $\tau_\eps$ at $\eps \mathfrak e_i$ with probability $\beta_i$, \[ \mathsf E_\zero f(W_{\alpha,\beta,\gamma}(\tau_\eps)) =\mathsf P_\zero (L\circ A_\alpha^{-1} (\tau_\eps)\le \zeta) \sum_{i\in \kad} \beta_i f(\eps \mathfrak e_i)\] 
Thus, by (ii) and (iii), in the case of $v=\zero,$ the right-hand side of \eqref{dynkina} reduces to $\lim_{\eps \to 0} \frac {\frac 1{1+\eps \gamma} \sum_{i\in \kad} \beta_i f(\eps \mathfrak e_i) -f(\zero)}{\eps^2+\alpha \eps}$. If $\alpha>0$, this limit, for $f$ that have the first derivative at $\zero$ along each edge, equals $\alpha^{-1} (\sum_{i\in \kad}\beta_i f_i'(0)-\gamma f(\zero))$.
Now, for $f\in \dom{\gen_{\alpha,\beta,\gamma}}$ this coincides with $\frac 12 f''(\zero)$, establishing that $\mathfrak H_{\alpha,\beta,\gamma}$ extends $\gen_{\alpha,\beta,\gamma}$. Similarly, as can be checked by de l'Hospital's rule,  in the case of $\alpha=0$ the above limit is equal to  $\frac 12 f''(\zero)+ \gamma (\gamma f(\zero) - \sui \beta_i f'_i(0))$ for $f\in \mathfrak C^2(S)$, and reduces to  $\frac 12 f''(\zero)$ for $f \in \dom{\gen_{\alpha,\beta,\gamma}}$. 

However, $\mathfrak H_{\alpha,\beta,\gamma}$, being a Feller generator, satisfies the positive-maxi\-mum principle, and thus 
cannot be a proper extension of a generator (see e.g. \cite{kallenbergnew}*{p. 377} or \cite{rogers}*{p. 242}). Since $\mathfrak H_{\alpha,\beta,\gamma}$ extends the generator $\mathfrak G_{\alpha,\beta,\gamma}$, \[ \mathfrak H_{\alpha,\beta,\gamma} = \gen_{\alpha,\beta,\gamma},\]
as desired.   

\subsection{The generator of $W_{\alpha,\beta,\gamma}$; an approach via resolvent}
\label{seq:res_Walsh_non_jump}

The fact that $W_{\alpha,\beta,\gamma}$ is generated by $\gen_{\alpha,\beta,\gamma}$ can also be established 
by looking at its resolvent, but then we need to use the results of \cite{itop,kostrykin2010brownian} in a more fundamental way than before.

To elaborate on this succinct statement, let us begin by noting that the process of Walsh can be seen as a concatenation of $\ka$ reflected Brownian motions on the edges having common local time at the graph's center --- this is the intrinsic meaning of Theorem \ref{thm:baryaktar}.  In $W_{\alpha,\beta}$ of \eqref{stickywalsh} each of these reflected Brownian motions is slowed down at $0$, and then in $W_{\alpha,\beta,\gamma}$ of \eqref{wkilled} --- killed with intensity $\gamma$. It is one of the main findings of \cite{itop,kostrykin2010brownian} that the first procedure transforms each of those reflected Brownian motions to the process generated by the one dimensional Laplace operator (see Section \ref{sec:bco}) with the boundary condition $\frac \alpha 2 f''(0)=f'(0)$, whereas the second --- to the  process generated by the same operator with the boundary condition $\frac \alpha 2 f''(0) -f'(0)+\gamma f(0)=0$.  The resolvent of the latter operator, let us call it $\mathfrak A$,  is well known and given by (see e.g. \cite{knigazcup}*{pp. 17-18}) 
\begin{equation*}
\rez{\mathfrak A} f (x) = C(f) \e^{\sqrt {2\lam} x } + D(f) \e^{-\sqrt {2\lam} x } - \sqrt{\tfrac 2\lam} \int_0^x \sinh \sqrt {2\lam} (x-y) f(y) \ud y,\end{equation*}
for all $x\ge 0$ and $f\in C[0,\infty]$, where $C(f)$ and $D(f) $ are constants defined as follows:
\begin{align*} C(f) = \frac 1{\sqrt {2\lam} } \int_0^\infty \e^{-\sqrt {2\lam} y } f(y) \ud y, \qquad 
D (f)=  \frac{(\sqrt {2\lam} - \alpha \lam - \gamma)C + \alpha f(0)}{\alpha\lam +  \sqrt {2\lam} + \gamma}. \end{align*}
In particular, 
\begin{equation*}
\rez{\mathfrak A} f (0) = C(f)+ D(f) = \frac{2\sqrt {2\lam}C(f) + \alpha f(0)}{\alpha\lam +  \sqrt {2\lam} + \gamma}.\end{equation*}

Coming back to $W_{\alpha,\beta,\gamma}$: using the strong Markov property as in \eqref{klucz}, we obtain the following formula for its resolvent
\[ \rla f = \rla^0 f + \rla f(\zero) \elam^0, \qquad f\in \cscan \]
where $\rla^0$ and $\elam^0$ are the resolvent of the minimal process and the Laplace transform of the time needed to reach the graph center, respectively (see \eqref{loo:5}). 
Moreover, since $W_{\alpha,\beta,\gamma}$ started at $\zero$ is at the $i$th edge with probability $\beta_i$ and while at this edge behaves like the process generated by $\mathfrak A$ described above, we have
\begin{equation}\label{aux} \rla f(\zero) = \sum_{i\in \kad} \beta_i \rez{\mathfrak A} f_i(0) =  \frac{2\sqrt {2\lam}\sum_{i\in \kad} \beta_i C_i(f) + \alpha f(\zero)}{\alpha\lam +  \sqrt {2\lam} + \gamma}\end{equation} with $C_i(f)$ defined in \eqref{cee}. It is now easy to see that $\rla f(\zero)$ is a particular case of $E(f)$  introduced in \eqref{loo:3} (recall that $\sui \beta_i=1$), and thus conclude that the resolvent of $W_{\alpha,\beta,\gamma}$ coincides with the resolvent of the generator $\gen_{\alpha,\beta,\gamma}$ of \eqref{gengen}, as was expected.



\subsection{The degenerate case of some  or all $\beta_i=0$}\label{sec:sui}

Throughout Section \ref{sec:tcom=0}, we have assumed that $\beta_i>0, i\in \kad$ and, without loss of generality, have normalized them to have $\sum_{i\in \kad}\beta_i =1$. Now is the time to clarify what happens when some or all $\beta_i$s vanish.  

Let us first consider the case where $\beta_i=0$ for all $i $ in a proper subset $\kal$ of $\kad$. Then $W_{\alpha,\beta, \gamma}$ is  a function of the degenerate process $W_\beta$ of Section \ref{degenerat}: by definition, $W_{\alpha,\beta,\gamma}$ is $W_\beta$ slowed down at $\zero$ as in \eqref{stickywalsh} and killed as in \eqref{wkilled} --- the local time of the degenerate Walsh's process coincides with the local time of the regular Walsh's process on a smaller graph with $\# (\kad \setminus \kal) $ edges.


The second case is that of $\kal =\kad$. Here, we must have  $\alpha >0$ for otherwise \eqref{mainas} is violated. To describe this scenario, we begin with the process which at each edge behaves like a Brownian motion but upon hitting $\zero$ for the first time stays at this point for ever. It is easy to check that the generator of this process is $\gen_{1,0,0}=\gen_{\alpha,0,0}$. Then, we modify the process by requiring that it is killed when its local time at $\zero$ (which is the same as the regular time) exceeds an independent exponential random variable with parameter $\gamma$. Again, a simple analysis of the related Dynkin operator shows that the modified process is generated by $\gen_{1,0,\gamma}=\gen_{\alpha,0,\alpha \gamma}$. 


 \subsection{Resolvents, semigroups and cosine families}

In the preceding sections we have established that the process related to the boundary condition \eqref{loo:0}  with $\mum=0$  is Walsh's process slowed down at the graph center and killed when the local time at the graph center exceeds an independent exponential random variable with parameter $\gamma$. This fact can also be seen, but perhaps less vividly, directly from  \eqref{loo:6}.  

First of all, if $\gamma$ is zero, $\elam$ is a zero function, as a reflection of the fact that the process is then honest. Moreover, $\lam \mapsto\e^{-\slam x}$ is the Laplace transform of the time needed for a Brownian motion starting at $x$ to reach $0$ (see \cite{ito}*{ Eq. 5) p. 26}, \cite{karatzas}*{Eq. (8.6) p. 96} or the original \cite{levybook}*{pp. 221-223}). 
It can be argued, moreover, that \begin{equation}\label{aa:1} \lam \mapsto  \frac \gamma{\lam \alpha + \slam \sui \beta_i + \gamma}, \end{equation} is completely monotone, and thus is the Laplace transform of the distribution of an $\R^+$-valued random variable. Therefore, 
$\elam (x)$ describes the distribution of the sum of two nonnegative random variables: one is the time to reach $0$ from $x$, the other --- the time spent at the boundary before the process is no longer defined. In particular, for $\alpha =1$ and $\sui \beta_i =0$, \eqref{aa:1} reduces to $\lam \mapsto  \frac \gamma{\lam +  \gamma} $, which we recognize as the Laplace transform of the exponential distribution with parameter $ \gamma $ (see Section \ref{sec:sui}, above).
 Likewise, for $\alpha=0$ and $\sui \beta_i =1$, \eqref{aa:1} reads $\lam \mapsto \frac \gamma{ \slam  + \gamma}$ which is the Laplace transform of the time needed for the L\'evy local time to exceed an independent exponential random variable with parameter $\gamma$, see  \cite{ito}*{p. 45, Eq. 1}. For $\alpha > 0$ and $\sui\beta_i =1$, the set of times the process spends at the boundary has positive Lebesgue measure, but contains no intervals, see  \cite{liggett}*{p. 128}.  

An additional insight into the process can also be gained from the analysis of the related cosine families and semigroups. Notably, see \cite{ela}, there is an elegant subspace of $\mathsf F$ that plays a similar role for the cosine family related to $W_{\alpha,\beta,0}$ as $\mathsf F_\beta$ of \eqref{przef} plays for $W_\beta$. As a result, there are counterparts of \eqref{glebszy} and \eqref{wsp:1}: to obtain a formula for the cosine family, or the semigroup, related to $W_{\alpha,\beta,0}$, it suffices to replace  $C_{\mathsf{ref}}$ by $C_{\mathsf{sticky}}$ and $T_{\mathsf{ref}}$ by $T_{\mathsf{sticky}}$ in the formulae just mentioned (see also the already cited \cite{barlow}). The same is true in the case of $W_{0,\beta,\gamma}$ with $\gamma >0$, but the general case seems to be still unknown.

We note also the paper \cite{adasma}, where existence of cosine families related to $W_{\alpha,\beta,\gamma}$ was established by a method of decomposition of resolvents inspired by Example 3.59 in \cite{liggett}. 

\subsection{{Lumping edges together}}\label{sec:lumping1}

In Section \ref{sec:bpy} we have seen, as  a corollary to  Theorem \ref{thm:Walsh},
that by lumping edges  of a star graph together, we transform Walsh's process to another Walsh's process, and the two processes have the same symmetric local time at $\zero$. Here, we extend this result to the case where additionally slowing down at $\zero$ and killing are allowed. 

To this end, let  
$W_\beta$ be Walsh's process  with parameter $\beta$, and let, as in Section \ref{sec:bpy}, for some natural number $\mathpzc n\le \ka$, 
$\Psi$ map $\kad $ onto $\mathpzc N \coloneqq \{1,\dots,\mathpzc n\}$. 
We know that the process $(X_j)_{j\in \mathpzc N}$ defined by  
     \[
     X_j(t)=\sum_{\{i\colon \Psi(i)=j\}}  (W_\beta)_{i}(t), \qquad j\in \mathpzc N, t \ge 0
     \]
where $(W_\beta)_i$ is the $i$th coordinate of $W_\beta$, 
is then Walsh's process also, but with the modified parameter $\widetilde \beta = (\widetilde \beta_j)_{j\in \mathpzc N}$, where $\widetilde \beta_j \coloneqq \sum_{\{i\colon \Psi(i)=j\}} \beta_i, j \in \mathpzc N$. Moreover, the symmetric local times of $W_\beta$ and $W_{\widetilde \beta}=X$ at $\zero$ coincide.

We claim that, for any $\alpha >0$, the process $(Y_j)_{j\in \mathpzc N}$ defined by  
     \[
     Y_j(t)=\sum_{\{i\colon \Psi(i)=j\}}  (W_{\alpha,\beta})_{i}(t), \qquad j\in \mathpzc N, t \ge 0
     \]
is $W_{\alpha,\widetilde \beta}$. Indeed, $W_{\alpha,\beta} $ is $W_\beta$ slowed down at $\zero$, as revealed in \eqref{stickywalsh}. Thus, for $j \in \mathpzc N $ and $t\ge0$, 
\[ Y_j (t) = \sum_{\{i\colon \Psi(i)=j\}}  (W_{\alpha})_{i}(A_\alpha^{-1}( t)) = X_j (A_\alpha^{-1} (t)) = (W_{\widetilde \beta})_j (A_\alpha^{-1} (t)).\]
Since $L$ featured in \eqref{aalpha}  (and defined in \eqref{andriya})  coincides simultaneously with the symmetric local time of $W_{\widetilde \beta}$, this relation shows that $Z =W_{\widetilde \beta} \circ A_\alpha^{-1} = W_{\alpha,\widetilde \beta},$ as claimed. 

Similarly, for any $\gamma>0$, the process $(Z_j)_{j\in \mathpzc N}$ defined by  
     \[
     Z_j(t)=\sum_{\{i\colon \Psi(i)=j\}}  (W_{\alpha,\beta,\gamma})_{i}(t), \qquad j\in \mathpzc N, t \ge 0
     \]
is $W_{\alpha,\widetilde \beta,\gamma}$. Indeed, as explained in \eqref{wkilled}, $W_{\alpha,\beta,\gamma}$ is identical to $W_{\alpha,\beta}$ before the time, say, $\tau$,  when $L\circ A_\alpha^{-1}$  reaches the level of an independent random variable $\zeta$ that is exponentially distributed with parameter $\gamma$, and is undefined later. Hence, for $i\in \mathpzc N$, $Z_j (t)$ is either undefined (for $t\ge \tau$) or equals $\sum_{\{i\colon \Psi(i)=j\}}  (W_{\alpha,\beta})_{i}(t)=W_{\alpha,\widetilde \beta}(t)$ (for $t<\tau$). Since $L$ is also the symmetric local time of $W_{\alpha,\widetilde \beta}$, this means that $Z$ is $W_{\alpha,\widetilde \beta,\gamma}$, as announced above. 

Notably, even though above we have tacitly assumed that all $\beta_i$ are positive and sum up to $1$, it can be argued that the principle of lumping the edges together applies also to the degenerate case of Section \ref{sec:sui}; we leave the details to the readers (see also Sections \ref{lumping} and \ref{lumping2}, further down).

\section{The case of non-zero but finite $\mum$}\label{sec:finite}

In this section, we turn to the case where $\mum$ in \eqref{loo:0} is non-zero but finite, and so the condition in question can be rewritten as  
\begin{equation*}\tfrac{\alpha}2 f''(\zero) - \sum_{i\in \kad} \beta_i f'_i(0) + \gamma  f(\zero) = \delta \int_{S_{\zero}} [f(x) - f(\zero)] \,\mu (\mud x),  \end{equation*}	
where $\mu \coloneqq \frac \mum {\mum(S_{\zero})}$ is a probability measure and $\delta \coloneqq \mum (S_{\zero})$. 
Equivalently, \begin{equation}\label{loo:7} \tfrac{\alpha}2 f''(\zero) - \sum_{i\in \kad} \beta_i f'_i(0) + (\gamma +\delta) f(\zero) = \delta \int_{S_{\zero}} f \ud \mu,  \end{equation}	
showing in particular that a non-zero $\mum$ automatically contributes a term to the coefficient of $f(\zero)$.  As a result, there are two natural interpretations of $\gamma$. 
\begin{itemize}
\item [(a) ] In the first of these, the process related to \eqref{loo:7} is seen as a modification of the process related to 
\begin{equation}\label{loo:7a} \tfrac{\alpha}2 f''(\zero) - \sum_{i\in \kad} \beta_i f'_i(0) + \delta f(\zero) = \delta \int_{S_{\zero}} f \ud \mu;  \end{equation}	
when the local time at $0$ of the latter exceeds an independent exponential time with parameter $\gamma$, the former is killed, that is, no longer defined. 
\item [(b) ] In the second interpretation,  the process related to  \eqref{loo:7} is seen as a continuation of $W_{\alpha,\beta,\gamma+\delta}$: at a random time when the latter ceases to be defined, the former is `resurrected' with probability $\frac \delta{\delta +\gamma}$ and starts all over again, forgetting the past, at a random point with distribution $\mu$. 
\end{itemize}
\begin{figure}
\begin{center}\includegraphics[width =11cm,height=7cm]{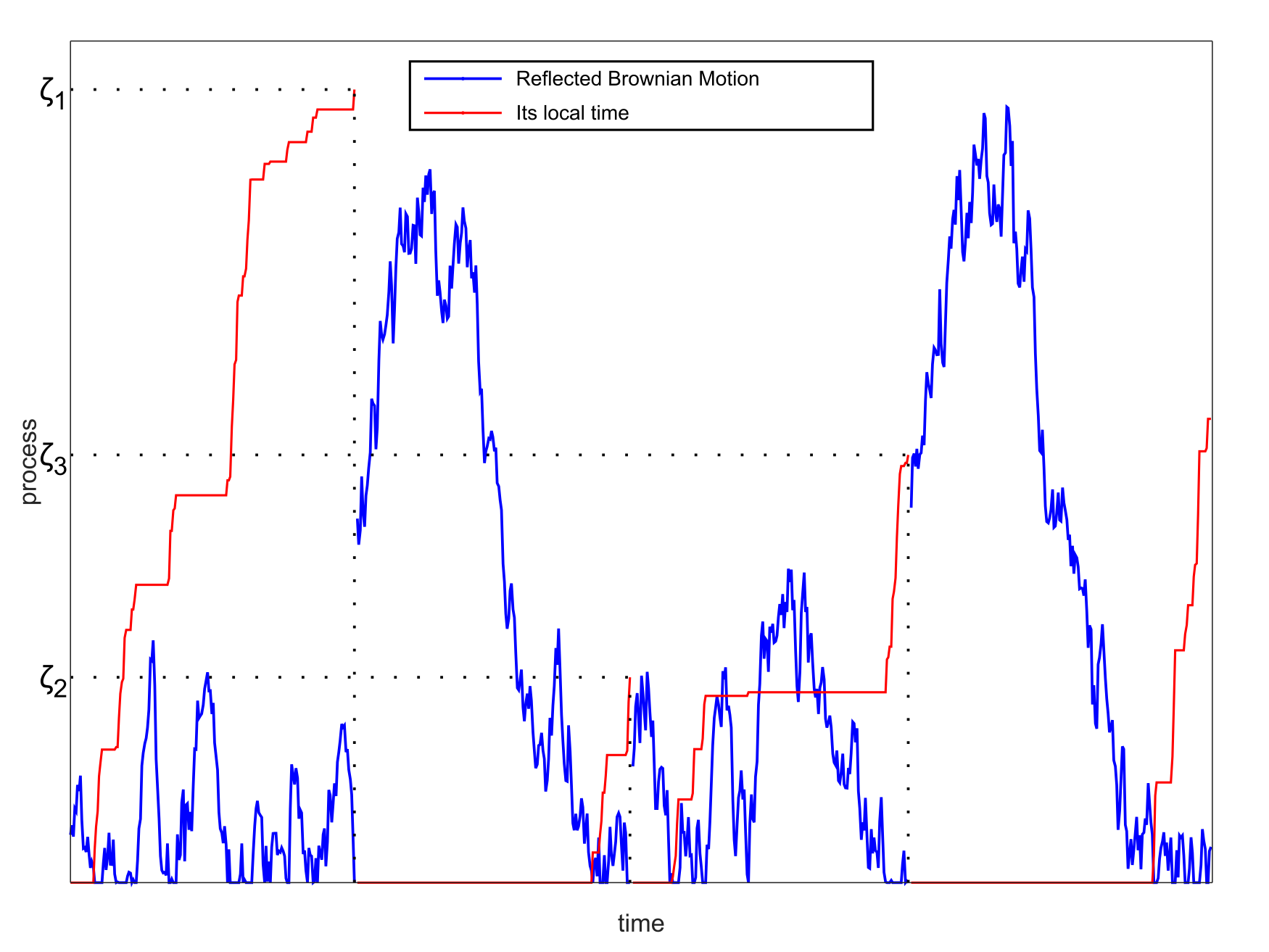}\end{center}
\caption{A Brownian motion on the nonnegative half-line with jumps from the boundary where $x=0$. When the local time at $x=0$ exceeds an independent random variable, the process starts all over again at a random point.  }
\label{fig1}
\end{figure}

Of course, interpretation (b) applies also to the process related to \eqref{loo:7a}: whenever the process generated by $\gen_{\alpha,\beta,\delta}$ ceases to be defined, the process related to \eqref{loo:7a} is resurrected and starts all over again at a random point.  Figure \ref{fig1} illustrates this procedure in the analogous situation of Brownian motion on the positive half-line. In the case of star graph the process can not only restart at any of the edges but also each excursion from zero can lead to a different edge; otherwise, the analogy is complete. 

Details of the stochastic construction are given below.
  As in the previous section, we start with a rather general situation to specialize to the case of processes on star graphs later.

\subsection{Concatenating infinitely many copies of one Feller process}\label{sec:cim}

Let the space $S$, the process $X$, the point $a$ and the local time $L$ be as in Section \ref{sec:sup}.   Guided by the interpretation (b) presented above, given  a probability measure $\mu$ on $S$ and real parameters $\gamma \ge 0$ and $\delta >0$, we construct a Feller process $\mc C_{\gamma,\delta,\mu}X$ which initially is identical to $\mc K_{\gamma+\delta} X$ (see  Section \ref{sec:kaf}), but when $\mc K_{\gamma+\delta} X$ ceases to be defined $\mc C_{\gamma,\delta,\mu}X$ can be continued: with probability $p\coloneqq \frac \delta{\gamma +\delta} $ it starts anew at a~random point distributed according to the measure $\mu$, and with probability $1-p$ is killed. 

Our construction uses the following building blocks --- the construction of the type presented below has apparently been devised by Ikeda, Nagasawa and  Watanable in \cite{ikedacon}; see also the recent \cite{fitzsi}.  
\begin{itemize} 
\item [(i) ] Processes $X_n, n \ge 1$ that are independent and have the same transition probabilities as $X$, and $L_n, n \ge 1$ that are their local times at $a$. However, whereas the initial distribution of $X_1$ is chosen arbitrarily, all the initial distributions of $X_n, n \ge 2$ are the same and identical to $\mu$. 
\item [(ii) ] Random variables $\zeta_n, n \ge 1$ that are mutually independent, and independent of $X_n$s, with common exponential distribution of parameter $\gamma+\delta$.  
\item [(iii) ] Bernoulli random variables $\epsilon_n, n \ge 2$ that are mutually independent, and independent of $X_n$s and $\zeta_n$s, such that $\mathsf P (\epsilon_n = 1) = p$ and $\mathsf P (\epsilon_n = 0) = 1-p$.  

\end{itemize}

\newcommand{\varsigmam}{\tau}
In terms of these, we first define $\mdelta_n\coloneqq \inf\{t\geq 0\colon L_n(t)=\zeta_n \}, n \ge 1$ as the moments when the local times at $a$ reach exponential levels, and then, if $X_n$s are honest, let 
\begin{equation}\label{eq:MP_jump_exit_constr}
 \mc C_{\gamma,\delta,\mu}X(t)\coloneqq  X_{n+1}(t-\varsigmam_n), \quad  t\in[\varsigmam_n, \varsigmam_{n+1}),  \qquad n=0,\dots, n_0-1, 
\end{equation} 
where 
\( \varsigmam_0\coloneqq 0, \varsigmam_n \coloneqq {\textstyle \sum_{k=1}^{n}} \mudelta_k, n \ge 1 \) and  $n_0 \ge 1$ is the smallest index such that $\epsilon_n =0$. In particular, $\mc C_{\gamma,\delta,\mu}X$ is left undefined from $\varsigmam_{n_0}$ onwards. Moreover, $\mc C_{\gamma,\delta,\mu} L $ given by 
  \begin{equation}
    \label{eq:localTime_jumps}
\mc C_{\gamma,\delta,\mu} L(t)\coloneqq \sum_{k=1}^{n} L_k(\mdelta_k)  +L_{n+1} (t-\tau_n), 
\, t\in[\varsigmam_n, \varsigmam_{n+1}),  \, n=0,\dots,n_0-1, 
\end{equation}
{together with the convention \eqref{eq:agreement_LT} that allows extending the definition beyond $\varsigmam_{n_0}$,}
is a local time of $\mc C_{\gamma,\delta,\mu}X$ at $a$. 
Also, if additionally $L_n (t)= \lim_{\varepsilon\to0+}\tfrac{1}{2\varepsilon}\int_0^t\1_{\{0<|X_n(s)|\leq \varepsilon\}}\ud s,$ for $t \ge 0$ and $n \ge 1$, then, clearly, $\mc C_{\gamma,\delta,\mu}L (t)$ $= \lim_{\varepsilon\to0+}\tfrac{1}{2\varepsilon}\int_0^t\1_{\{0<|\mc C_{\gamma,\delta,\mu} X (s)|\leq \varepsilon\}}\ud s,$ for $t \ge 0$. 

If $X_n$s are not honest, these formulae need to be modified as follows: we agree that if there is an $n<n_0$ and an $s\in [\tau_n,\tau_{n+1})$ such that $X_n(s)$ is undefined, then so is $\mc C_{\gamma,\delta,\mu}X$ from that time on.  However, $\mc C_{\gamma,\delta,\mu}L$ is continued beyond this time as in \eqref{eq:agreement_LT}.

It is easy to see that $ \mc C_{\gamma,\delta,\mu}X$ is strong Markov, because so is $\mc K_{\delta +\gamma}X$. It is also honest whenever $X_n$s are honest and $\gamma =0$.  We claim, moreover, that it possesses the Feller property. To prove this, we note first of all that $ \mc C_{\gamma,\delta,\mu}X$ extends the minimal process (the process that is undefined from the moment $\sigma$ when $X$ reaches $a$ for the first time), and thus Proposition \ref{lem:feller} can be used. 
Furthermore, in Section \ref{sec:fno} we have established that $\mathcal K_{\gamma+\delta} X$ is a Feller process, and in particular has \cadlag\ paths. It follows that, almost surely conditional on $\mathcal K_{\gamma+\delta} X$ starting at $a$, $\lim_{t\to 0+} f(\mathcal K_{\gamma+\delta} X (t)) = f(a)$; hence also $\lim_{t\to 0+} f( \mc C_{\gamma,\delta,\mu}X (t)) = f(a)$ because $ \mc C_{\gamma,\delta,\mu}X$ is identical to $\mathcal K_{\gamma+\delta} X$ for all the times smaller than $\tau\coloneqq \tau_1$, and $\tau >0$.  Thus, \eqref{wareczek} is satisfied, completing the proof.

For a future use we note that the resolvent of $\mc C_{\gamma,\delta,\mu}X$, denoted below $\wrla, \lam >0$, can be informatively expressed in terms of the resolvent $\rla^\tau, \lam >0$ of $\mathcal K_{\gamma +\delta} X$. For,  using strong Markov property of $\mc C_{\gamma,\delta,\mu}X$ as in Section \ref{sec:fprop} (this time at $\tau$, not at $\sigma$), we obtain
\begin{align*}
\wrla g(x) &=\rla^\tau g(x) + \mathsf E_x  \e^{-\lam \tau}  \mathsf E_{\mc C_{\gamma,\delta,\mu}X(\tau)} \int_0^\infty \e^{-\lam t} g(\mc C_{\gamma,\delta,\mu}X(t)) \ud t, \\ &= \rla^\tau g(x) + p \elam^\tau (x) \int_S \wrla g \ud \mu 
\end{align*}
with $\elam^\tau (x) \coloneqq  \mathsf E_x \e^{-\lam \tau}  $, because at $\tau$ the process $\mc C_{\gamma,\delta,\mu}X$ starts all over again with probability $p$ and its starting distribution is $\mu$. Next, we integrate both sides of the so-obtained relation with respect to $\mu$ to see that $\int_S \wrla g\ud \mu = \frac { \int_S \rla^\tau g\ud \mu }{1 - p \int_S \elam^\tau \ud \mu}$; the denominator here is never zero, even if $p=1$, because $\tau >0$ and so, for any $x\in S $, $\elam^\tau (x) \le  \elam^\tau (a) < 1$.   Since $\elam^\tau = \mathsf 1_S - \lam \rla^\tau \mathsf 1_S$, this yields
\begin{equation}\label{cim:5} \wrla g = \rla^\tau g + \frac{\int_S \rla^\tau g\ud \mu}{1-p + p\lam \int_S \rla^\tau \mathsf 1_S \ud \mu}\elam^\tau , \qquad g \in \coss. \end{equation}

\subsection{Two alternative constructions}\label{sec:tac}
Alternatively, we can start from processes of point (i) and two sequences, say,  $\jcg{\upzeta_{1,n}} $ and $\jcg{\upzeta_{2,n}}$ of random variables (jointly independent and independent of the processes), the first of these being exponentially distributed with parameter $\gamma$, the second --- exponentially distributed with parameter $\delta$. Then $\zeta_n \coloneqq \upzeta_{1,n}\wedge \upzeta_{2,n},n \ge 1$ are  exponential with parameter $\gamma +\delta$, as in point (ii). Moreover, $\epsilon_n \coloneqq  [\upzeta_{2,n}\le \upzeta_{1,n}], n \ge 1$ (where $[\cdot]$ is the Iverson bracket) have Bernoulli distribution with parameter $\frac \delta {\gamma +\delta}$ and are independent of $\zeta_n$s, as in point (iii). Hence, the rest of the construction proceeds as before.

We may also construct $\mc C_{\gamma,\delta,\mu}X$ in two steps,  by building $\mc C_{0,\delta,\mu}X$ first and then killing it with the help of transformation $\mc K_\gamma$ --- see interpretation (a) at the beginning of Section \ref{sec:finite}. Notably, the first step does not require the Bernoulli variables $\epsilon_n$s, for the process $\mc C_{0,\delta,\mu}X$ is always restarted; similarly, in the alternative construction, we need merely the sequence $\jcg{\upzeta_{2,n}}$, because $\gamma=0$ corresponds to  $\upzeta_{1,n}$s being infinite. Moreover, the second step requires just one additional exponentially distributed random variable. 

The situation resembles the following experiment. Let $\tau_n, n \ge 1$ be independent exponential random variables with common parameter $\gamma >0$, and let $p\in (0,1)$ be given. We wait for the exponential time $\tau_1$ to perform a random experiment which with probability $p$ results in a success. If we fail, we wait for the  additional exponential time $\tau_2$ to perform the random experiment for the second time, etc.: if we fail $n$ times, we wait for the time $\tau_{n+1}$ to have our next chance. Then, the time to the first success turns out to be exponential with parameter $p \gamma$.

To prove that \begin{equation}\label{andreya} \mc K_\gamma \mc C_{0,\delta,\mu}X\overset{d}=\mc C_{\gamma,\delta, \mu}X\end{equation} we suppose the latter process was constructed with the help of sequences   $\jcg{\upzeta_{1,n}} $ and $\jcg{\upzeta_{2,n}}$ and that for the construction of the former, besides the same copies of $X$,  the second of these sequences is used.  Also, let $\zeta$ be an independent random variable of exponential distribution with parameter $\gamma$, and assume initially that processes $X_n, n \ge 1$ are honest.

Initially, up to the first jump or killing,  $\mc K_\gamma \mc C_{0,\delta,\mu}X$ coincides with $X_1$, and the time of jump or killing is the moment when the local time of $X_1$ reaches the level of $\zeta \wedge \upzeta_{2,1}$; if $\zeta > \upzeta_{2,1}$, we see a jump, otherwise the process is killed. The same is true of  $\mc C_{\gamma,\delta, \mu}X$ except that we wait to the moment when the local time of $X_1$ reaches the level of $\upzeta_{1,1} \wedge \upzeta_{2,1}$; if $\upzeta_{1,1} > \upzeta_{2,1}$, we see a jump, otherwise the process is killed. Since $\zeta$ and $\upzeta_{1,1}$ have the same distribution, so do   $\mc K_\gamma \mc C_{0,\delta,\mu}X$  and $\mc C_{\gamma,\delta, \mu}X$ up to the first jump or killing. Suppose next that, for some natural $n$, both processes have the same distribution up to the time of $n$th jump or killing. We can also assume that neither of the processes has been killed as yet, because otherwise there is nothing to prove. This means, on the one hand, that $\upzeta_{1,k}\ge \upzeta_{2,k}, k=1,\dots,n$ and, on the other, that $\zeta > \mc C_{0,\delta,\mu}L(\tau_n)$, 
where $\tau_n$  is the moment when the local time of $X_n$ reaches the level of $\upzeta_{2,n}\le \upzeta_{1,n}$. By the memoryless property of exponential distribution, conditional on the event just described, $\wt \zeta \coloneqq \zeta- \mc C_{0,\delta,\mu}L(\tau_n)$ is still exponentially distributed with parameter $\gamma$. Moreover, the next jump or killing in $\mc K_\gamma \mc C_{0,\delta,\mu}X$ will occur at the moment when the local time reaches the level of $\wt \zeta \wedge \upzeta_{2,n+1}$; if $\wt \zeta > \upzeta_{2,n+1}$ we will see a jump, otherwise the process will be killed. Since a similar description applies to  $\mc C_{\gamma,\delta, \mu}X$, except that the role of $\zeta$ is played by $\upzeta_{1,n+1}$, this shows that the processes in question have the same distribution also to the $n+1$ jump or killing event, completing the proof by induction. To treat the case where $X_n$s need not be honest, it suffices to note that in this case additional killing events can occur between those considered above, but that their distribution is the same in both cases. 

\subsection{Operations of killing, slowing down and concatenating commute}\label{przemienne}
As explained in Section \ref{sec:kaf}, transformations $\mc K_\gamma$  and $\mc S_\alpha$ in a sense commute. An argument similar to that used to  establish \eqref{andreya} shows that so do $\mc C_{0,\delta,\mu}$ and  $\mc K_\gamma $ (and their composition amounts to $\mc C_{\gamma,\delta, \mu}$).  It can also be argued that $ \mc C_{0,\delta,\mu}$ commutes with $S_\alpha$ and that the family of transformations $\mc C_{\gamma,\delta, \mu}$ indexed by $\gamma $ and $\delta$ with fixed $\mu$ has the following semigroup property: 
\begin{equation*}
\mc C_{\gamma,\delta, \mu} \mc C_{\gamma', \delta',\mu} X \overset{d}=\mc C_{\gamma+\gamma',\delta +\delta',\mu}X.\end{equation*}


\subsection{Definition of $W_{\alpha,\beta,\gamma,\delta\mu}$}\label{sec:dow}

Having covered the general scenario, we return to the case of Brownian motions on star graphs. Let $\alpha,\beta, \gamma$ and $\mum = \delta \mu$ be given and assume additionally that $\sui \beta_i =1$. 

Then 
$W_{\alpha,\beta,\gamma,\delta\mu}$ is simply defined as $\mc C_{\gamma,\delta,\mu}W_{\alpha,\beta}$.  Therefore, we have the following special case of \eqref{cim:5}: 
\begin{equation}\label{dow:1} \wrla g = \rez{\gen _{\alpha,\beta,\gamma+\delta}}g + M(g) \elam ,\qquad g\in \coss \end{equation}
where $\wrla, \lam>0$ is a temporary notation for the resolvent of  $W_{\alpha,\beta,\gamma,\delta\mu}$,
\begin{equation}\label{dow:2} M(g) \coloneqq \frac {\delta \int_S \rez{\gen _{\alpha,\beta,\gamma+\delta}} g \ud \mu}{\gamma+ \delta \lam \int_S \rez{\gen _{\alpha,\beta,\gamma+\delta}} \mathsf 1_S \ud \mu}\end{equation}
and 
\begin{equation}\label{dow:3} 
(\elam)_i (x) = \frac {\gamma+\delta}{\lam \alpha + \slam  + \gamma +\delta} \e^{-\slam x}, \qquad x \ge 0, i\in \kad,\end{equation}
is the function introduced \eqref{loo:6}, but with $\sui{\beta_i}$ equal to $1$ and $\gamma $ replaced by $\gamma +\delta$. 

\subsection{The generator of $W_{\alpha,\beta,\gamma,\delta\mu}$}\label{sec:gow}
        

We claim that $W_{\alpha,\beta,\gamma,\delta\mu}$ is generated by  the operator \[ \gen_ {\alpha,\beta,\gamma,\delta\mu},\] 
defined as the restriction of the Laplace operator $\Delta$ to the set of functions $f\in \mathfrak C^2(S)$ such that $Ff = \delta \int_S f\ud \mu$, where $F$ was introduced in \eqref{loo:1}. To show this, it suffices to check that $\wrla$ of \eqref{dow:1} coincides with $\rez{\gen_{\alpha,\beta,\gamma,\delta\mu}}$. Moreover, 
 the solution to the resolvent equation for $\gen_ {\alpha,\beta,\gamma,\delta\mu}$, if it exists, is necessarily unique, because $\gen_ {\alpha,\beta,\gamma,\delta\mu}$ satisfies the positive maximum principle (see Appendix \ref{maks}), and is thus dissipative. Hence,
it suffices to show that $f\coloneqq \wrla g$ solves the resolvent equation for this operator.
 
Now, since both $\rez{\gen _{\alpha,\beta,\gamma+\delta}}g$ and $\elam$ belong to $C^2(S)$, so does $f$. Moreover, for the functional $F$ of \eqref{loo:1} (with $\gamma$ replaced by $\gamma +\delta$)  
we have $F \rez{\gen _{\alpha,\beta,\gamma+\delta}}g=0$ and $F\elam = \gamma + \delta$.  
Therefore, 
\begin{align*} Ff-\delta  \int_Sf \ud  \mu &= -\delta \int_S\rez{\gen _{\alpha,\beta,\gamma+\delta}}g\ud \mu + M(g)\, (\delta +\gamma - \delta \int_S \elam \ud \mu ), 
\end{align*}
this in turn is zero, because $\elam =\mathsf 1_S - \lam \rez{\gen _{\alpha,\beta,\gamma+\delta}}\mathsf 1_S$. Hence, $f$ belongs to $\dom{\gen _{\alpha,\beta,\gamma,\delta\mu}}$. Finally, since $\elam$ is an eigenvector of $\Delta$ corresponding to $\lam>0$, we check that $\lam f - \gen_ {\alpha,\beta,\gamma,\delta\mu}f = g$, completing the proof.

\subsection{$W_{\alpha,\beta,\gamma,\delta \mu}$ as a function of Brownian motion and subordinators}\label{sec:wab}

Next, we want to see that to construct a $W_{\alpha,\beta,\gamma,\delta \mu}$ is suffices to have just one $\ka$-dimensional Brownian motion, as opposed to its infinitely many independent copies used in Sections \ref{sec:cim}--\ref{sec:tac}. This will lead to a very handy formula expressing $W_{\alpha,\beta,\gamma,\delta \mu}$ in terms of this Brownian motion and certain subordinators.
 In this section  we assume that $\beta_i >0, i \in \kad$ and, for convenience, normalize our parameters so that $\sui \beta_i =1$; we will get rid of this additional assumption in Section \ref{sec:extend}.

\subsubsection{The construction} We start with the case where $\alpha=\gamma=0$, that is, with the process $W_{0,\beta,0,\delta \mu}$, where $\delta >0$ and $\mu $ is a measure on $\scan$ (see Section \ref{sec:bpy}). This process behaves like $W_\beta$ up to the moment when its local time at $\zero$ reaches the level of an independent, exponential random variable with parameter $\delta >0$, and then starts all over again at a random point with distribution  $\mu$.   Let 
 \begin{itemize}
\item [($\upalpha$)] $B$ be a $\ka$-dimensional Wiener process starting at a random point of $\scan$,
\item [($\upbeta$)] $\jcg{\zeta_n}$ be a sequence of independent, exponentially distributed random variables with parameter $\delta>0$ that are independent of $B$ also, and 
\item [($\upgamma$)] $(\xi_n)_{n\ge 1}$ be a sequence of independent random variables with common distribution $\mu$ that are independent of $B$ and $\jcg{\zeta_n}$ too. 
\end{itemize}
As a preparation for the main construction, using these building blocks, we produce a sequence $\jcg{B_n}$ of Wiener processes such that $B_1=B$ and  $B_n(0)=\xi_{n-1}, n \ge 2$. 
To proceed by induction we let $B_1\coloneqq B$, and,  assuming that $B_1,\ldots,B_n$ are already constructed, define $\mdelta_n\coloneqq \min \{t\ge 0\colon L_n(t) =\zeta_n\}$, where $L_n$ is the local time of $B_n$ at $\zero$. Theorem \ref{thm:baryaktar} tells us then that there is a unique collection of processes $T_{n,i}, i\in \kad$ such that a counterpart of \eqref{eq:balance_locTime_Walsh's process1} holds. Moreover, as noted at the end of Section \ref{sec:bay}, we can take  
$B_{n+1}\coloneqq \seq{B_{n+1,i}}$ where 
\[ B_{n+1,i} (t) \coloneqq \xi_{n,i} + B_{n,i} (T_{n,i} (\mdelta_n) +t) - B_{n,i}(T_{n,i}(\mdelta_n)), \qquad t \ge 0, i \in \kad \]
$B_{n,i}$s are coordinates of $B_n$ and $\xi_{n,i}$s are coordinates of $\xi_{n}$, as a Wiener process starting at $\xi_{n}$ and independent of $B_k(t), t\in [0,\Delta_k], 1\leq k\leq n.$ 
This completes the construction of $\jcg{B_n}$. We note that the processes $(B_n {(t\wedge \mdelta_n))}_{t\ge 0}$, $n \ge 1$ are independent. 

With this sequence under our belt, we proceed as follows. For $n \ge 1$ and $i\in \kad$,  let 
\[ L_{n,i}(t)\coloneqq \max_{s\in[0,t]}(-\brown_{n,i}(s)\vee 0) \mqquad{and} \brown\refl_{n,i}(t)\coloneqq \brown_{n,i}(t)+L_{n,i}(t), \qquad t \ge 0.\]   
Then, by Theorem   \ref{thm:baryaktar}, each   
    \begin{equation}
                \label{wab:1}
        X_n\coloneqq  \seq{\brown\refl_{n,i}\circ T_{n,i}}, \qquad n \ge 1,
        \end{equation}
is a Walsh's process  with parameter $\beta$, which for $n\ge 2$ starts at $\xi_{n-1}$. Therefore,  $W_{0,\beta,0,\delta \mu} $
can be constructed by means of the following formula, being a particular case of \eqref{eq:MP_jump_exit_constr}: 
\begin{equation*} W_{0,\beta,0,\delta \mu} (t) \coloneqq  X_{n+1}(t-\varsigmam_n), \quad  t\in[\varsigmam_n, \varsigmam_{n+1}),  \qquad n\ge 0 ,
\end{equation*} 
where 
\( \varsigmam_0\coloneqq 0$ and $ \varsigmam_n \coloneqq {\textstyle \sum_{k=1}^{n}} \mudelta_k, n \ge 1 \). Moreover, 
  \begin{equation}
    \label{wab:2}
L(t)\coloneqq \sum_{k=1}^{n} L_k(\mdelta_k)  +L_{n+1} (t-\tau_n), 
\qquad t\in[\varsigmam_n, \varsigmam_{n+1}),  n\ge 0, 
\end{equation}
is its local time, and since $L_n (t)= \lim_{\varepsilon\to0+}\tfrac{1}{2\varepsilon}\int_0^t\1_{\{0<|X_n(s)|\leq \varepsilon\}}\ud s,$ for $t \ge 0$ and $n \ge 1$, we have
\begin{equation}
    \label{eq:L_loc_W}
    L (t) = \lim_{\varepsilon\to0+}\tfrac{1}{2\varepsilon}\int_0^t\1_{\{0<|W_{0,\beta,0,\delta \mu}(s)|\leq \varepsilon\}}\ud s,\qquad    t \ge 0. 
\end{equation} 
Finally, for each $i\in \kad$, 
\begin{equation} T_i (t) \coloneqq \sum_{k=1}^{n} T_{k,i} (\mdelta_i)  +T_{n+1,i} (t-\tau_n), \qquad t\in[\varsigmam_n, \varsigmam_{n+1}),  n\ge 0, 
\label{wab:2a}\end{equation}
can be interpreted as the time spent by $W_{0,\beta,0,\delta \mu} $ at the $i$th edge of $\scan$ up to $t\ge 0$.  An induction argument shows that $\sum_{i=1}^\ka T_i(t) =t, t \ge 0$. 

\subsubsection{Representation via subordinators} 
Let, here and in what follows,  by abuse of notations, $B_{i},i\in \kad $ denote the $i$th coordinate of the original \linebreak
$\ka$-dimensional  process $B$; previously, $B_n, n\ge 1$ denoted the $n$th  copy of $B$, but we do not need this notation any longer. Also, let  $(W_{0,\beta,0,\delta \mu})_i$ be the $i$th coordinate of $W_{0,\beta,0,\delta \mu}$. 
Since, for each $n\ge 1$ and $i\in \kad$, $L_{n,i} =\beta_i L_n$, an induction argument can be used to show that
\[ (W_{0,\beta,0,\delta \mu})_i (t)  = B_{i}(T_i(t)) + \beta_i L(t) + \sum_{\{k\ge 1\colon \tau_k \le t\}}\xi_{k,i},\qquad  t \ge 0, i\in \kad.\]
This relation, in turn, can be rewritten as 
\[ (W_{0,\beta,0,\delta \mu})_i (t)  = B_{i}\circ T_i(t)  + \beta_i L(t) + {\textstyle \sum_{k=1}^{N\circ L(t)}}\xi_{k,i},\qquad  t \ge 0, i\in \kad.\]
where $N$ is the Poisson process constructed by means of $\jcg{\zeta_n}$, that is,  $N(t) $ is the largest integer $n\ge 1$ such that $\sum_{k=1}^n \zeta_k \le t$; if there is no integer with this property, $N(t) =0$. Indeed, since $L_{{j}}(\mdelta_{{j}}) =\zeta_{{j}}, {{j}} \ge 1$, we have $L(\tau_k) = \sum_{{{j}}=1}^k L_{{j}}(\mdelta_{{j}}) = \sum_{{{j}}=1}^k \zeta_{{j}}, k \ge 1$ and thus, $L$ being nondecreasing, condition $\tau_k \le t$ implies $\sum_{{{j}}=1}^k \zeta_{{j}}\le L(t)$. Moreover, since $\tau_k$ is the smallest of $t$ such that $L(t)=\sum_{{{j}}=1}^k \zeta_{{j}}$, the converse is also true, and so the conditions in question are equivalent. But this means that 
$\tau_k \le t$ iff $k\le N\circ L(t)$, establishing the relation. 

We further see that  
 \[ (W_{0,\beta,0,\delta \mu})_i = B_{i}\circ T_i  + U_i \circ L \qquad  i\in \kad,\]
where $U_i$ is the subordinator given by 
\begin{equation}\label{wab:3} U_i(t) \coloneqq  \beta_i t + {\textstyle \sum_{k=1}^{N(t)}}\xi_{k,i}, \qquad t\ge 0. \end{equation}
This relation reveals that (for each path) the pair $((W_{0,\beta,0,\delta \mu})_i, L)$ is a solution to the generalized Skorokhod problem for $B_i\circ T_i$ with noise $U_i$. Indeed, $(W_{0,\beta,0,\delta \mu})_i$ is nonnegative and, clearly,  $\int_0^\infty   \1_{\{B_i \circ T_i(t) >0\}}\ud L(t) =0$. Since $B_i \circ T_i$ does not have jumps, and $U_i$ is strictly increasing with $\grat U_i(t)=\infty$, the solution is unique and thus we have (see Section \ref{sec:inverses}) $L(t)=U_i^{-1} \circ K_i(t), t \ge 0$ where \( K_i(t)\coloneqq \sup_{s\in [0,t]} (-B_i \circ T_i (s) \vee 0)=\sup_{u\in [0,T_i(t)]} (-B_i(u) \vee 0)\).   
Introducing \begin{equation}\label{wab:4} \mc L_i(t) \coloneqq \sup_{s\in [0,t]} (-B_i(s) \vee 0), \qquad t \ge 0, i \in \kad \end{equation} (to be distinguished from $L_{{j}}$ featured in e.g. \eqref{wab:2}), we see finally that $K_i= \mc L_i \circ T_i$ and 
\begin{equation}\label{wab:5}(W_{0,\beta,0,\delta \mu})_i = (B_{i}  + U_i \circ U_i^{-1} \circ \mc L_i ) \circ T_i,  \qquad  i\in \kad. \end{equation}
Our discussion is summarized in the following result. 

\begin{thm}\label{thm:wba} There is a family $T_i, i \in \kad$ of nondecreasing nonnegative processes with continuous paths such that \begin{itemize}
\item [1. ]  $\sum_{i=1}^\ka T_i(t)=t, t \ge 0$, 
\item [2. ] $W_{0,\beta,0,\delta \mu}$ is represented as in \eqref{wab:5} for $U_i$ and $\mc L_i$ defined in \eqref{wab:3} and \eqref{wab:4}, respectively, and
\item [3. ] $U_i^{-1}\mc \circ \mc L_i \circ T_i$ does not depend on $i\in \kad$.  
\end{itemize}
Moreover, conditions 1. and 3. determine $T_i$s uniquely (a.s.), and the common value of $U_i^{-1}\circ \mc L_i \circ T_i$ coincides with the local time $L$ of $ W_{0,\beta,0,\delta \mu}$ --- see \eqref{wab:2} and {\eqref{eq:L_loc_W}.}
\end{thm}

The existence of the family described above has already been established; all we need to prove is the uniqueness. To this end, we present the following lemma that seems to be of some interest in itself. 

 \begin{lemma}
 \label{lem:balance_uniq}
     Let $\ell_i, i \in \kad $ be nondecreasing functions $\ell_i\colon[0,\infty)\to [0,\infty)$ that do not have common levels of constancy (this, by definition, means that there are no indexes $i\not =j$ and numbers $0\le a<b$ and $0\le c<d$ such that  $\ell_i(a) =\ell_i(b)=\ell_j(c)=\ell_j(d)$). Then
there is at most one collection $t_i, i\in \kad$ of nondecreasing functions $t_i\colon[0,\infty)\to [0,\infty)$ such that $\sum_{i\in \kad} t_i(s)=s, s\ge 0$ and $\ell_i \circ t_i$ does not depend on $i\in \kad$. 
 \end{lemma}
 \begin{proof}
Striving for a contradiction, suppose that there are two different collections, say, $t_i,i\in \kad$ and $\hat t_i, i \in \kad$, of functions that satisfy the conditions listed in the lemma.  Then, for some $i\in \kad $ and $s \ge 0$, we have $a\coloneqq  t_i(s) \not = \hat t_i(s)=:b$, and without loss of generality we may assume that $0\le a<b$. Moreover, since  $\sum_{i\in \kad} t_i(s)= \sum_{i\in \kad} \hat t_i(s)$, there is a $j\not =i$ such that $d\coloneqq t_j(s) > \hat t_j (s)\eqqcolon c\ge 0$. Monotonicity of $\ell_i$ and $\ell_j$ yields therefore 
 $ \ell_i(t_i(s))\leq  \ell_i(\hat t_i(s))$ and $ \ell_j(t_j(s))\geq  \ell_j(\hat t_j(s))$. Since, by assumption, $\ell_i(t_i (s))= \ell_j(t_j (s)) $ and  $\ell_i(\hat t_i (s))= \ell_j(\hat t_j (s)) $, this is possible only if all these four numbers are the same, that is, when 
 \[
\ell_i(a)= \ell_i(b) = \ell_j(c)= \ell_j(d).
 \]
Now we have the required contradiction since we assumed that $\ell_i$ and $\ell_j$ do not have common levels of constancy.  
\end{proof}

\begin{proof}[Proof of uniqueness in Theorem \ref{thm:wba}]

By Lemma \ref{lem:balance_uniq} it suffices to show that the processes $U_i^{-1}\circ \mc L_i, i \in \kad $ do not have common levels of constancy, that is, that their generalized inverses $(U_i^{-1}\circ \mc L_i)^{-1} = \mc L_i^{-1}\circ U_i$ do not have common points of jumps (a.s.). But, each $U_i$ is clearly a subordinator and the well-known result  says that so is $\mc L_i^{-1}$ ({see \cite[Chapter IV, Theorem 8]{bertoin}}).  
Therefore,  $\mc L_i^{-1}\circ U_i, i \in \kad $ are subordinators. This reduces our task to showing that they are independent, because 
independent subordinators do not have common moments of jumps. Indeed, jumps of a subordinator are precisely the jumps of the corresponding Poisson point process, and the Disjointness Lemma in \S 2.2 of  
\cite{kingmanpoisson} says that independent Poisson point processes do not have common points of jumps.

Since $\mc L_i$s are independent among themselves and independent of $U_i$s, we are left with establishing that $U_i$s are independent, that is, that the compound Poisson processes, say,  $V_i, i \in \kad$, featured in \eqref{wab:3} are independent. These processes may seem to be dependent, being driven by the same Poisson process $N$, but in fact are not. This is because random variables $\xi_k,k \ge 1$ have values in $\scan \setminus \{\zero\}$ so that their coordinates $\xi_{k,i}$ are nonnegative, and for each index $k$ and event $\omega$ there is precisely one $i$ such that $\xi_{k,i}(\omega) >0$; this implies that the actual driving processes of $V_i$s are obtained from $N$ by (independent) coloring (see \cite{kingmanpoisson}*{p. 53}). \end{proof}

\subsubsection{Notes on Theorem \ref{thm:wba}} We note, first of all,  that Theorem \ref{thm:wba}  generalizes Theorem \ref{thm:baryaktar}. Indeed,  if there is no jump part in  \eqref{wab:3}, $U_i$ takes the form $U_i(t) = \beta_i t, t \ge 0$; hence, $U_i \circ U_i^{-1} (t)=t, t \ge 0$ for each $i\in \kad$ and consequently \eqref{wab:5} reduces to \eqref{eq:repr_Bayr}. At the same time, $U_i^{-1} (t)= \frac t{\beta_i},$ $t \ge 0, i \in \kad$ and so the requirement that   $U_i^{-1}\mc \circ \mc L_i \circ T_i$ does not depend on $i\in \kad$ simply means that balance conditions \eqref{eq:balance_locTime_Walsh's process1} are satisfied. We leave it to the reader to check that Theorem \ref{olepkimizabitym} generalizes the formula for $W_{\alpha,\beta,\gamma}$ introduced in Section \ref{sec:sand}.

Secondly, as a direct consequence of Theorem \ref{thm:wba} and commutation results of Section \ref{przemienne}, we obtain the following statement. 

\begin{thm}\label{olepkimizabitym} Let $W_{0,\beta,0,\delta\mu}$ be the process of \eqref{wab:5}. Then the distributions of $\mc S_\alpha\mc K_\gamma W_{0,\beta,0,\delta \mu}$ and $W_{\alpha,\beta,\gamma,\delta \mu}$ coincide, provided that in the construction of the former process one uses the  common value of $U_i^{-1}\mc \circ \mc L_i \circ T_i$,  $i\in \kad$ as the local time of $W_{0,\beta,0,\delta\mu}$. \end{thm}

Our final remark here concerns the symmetric local time $L$ of $W_{0,\beta,0,\delta \mu}$. Theorem \ref{thm:wba} includes, as its integral part, the representation
\[ L(t) = \grae \tfrac 1{2\eps}\int_0^t \1_{\{0< |W_{0,\beta,0,\delta \mu}(s)| <\eps \}} \ud s, \qquad t \ge 0,\]
and the fact that $L$ coincides with each and every process $U_i^{-1}\mc \circ \mc L_i \circ T_i,$ $i\in \kad$. 
The following corollary to the theorem provides an alternative description of $L$ issuing from these two statements. 

\begin{corollary}For each $i\in \kad$, 
\[ L(t)=\grae \tfrac{1}{2\beta_i\eps}\int_0^t\1_{\{0<(W_{0,\beta,0,\delta \mu})_i(s) <\eps\}}\ud s, \qquad t \ge 0\]
\end{corollary}
\begin{proof}
We start by noting that in the case of $\ka=1$, the family $T_i, i \in \kad$ of time-changes reduces to just one process $T_1$ and this process is deterministic: $T_1(t)=t, t\ge 0$. Moreover, the vector $\beta$ becomes a scalar and is necessarily equal to $1$. Finally,  $W_{0,1,0,\delta \mu}$ has only one coordinate, and  \eqref{wab:5} becomes 
\begin{equation}\label{dzis:3}W_{0,1,0,\delta \mu} = B  + U \circ U^{-1} \circ \mc L ,\end{equation}
where (a) $B$ is a Brownian motion and $\mc L$ is its {running minimum}
defined by $\mc L(t) \coloneqq -\min_{s\in [0,t]} (B(s) \wedge 0),  t \ge 0$,  and (b) $U$ is an independent subordinator of the form 
\[ U(t) \coloneqq t + {\textstyle \sum_{k=1}^{N(t)}}\xi_{k}, \qquad t\ge 0. \]
where, as in \eqref{wab:3},  $N$ is a Poisson process with parameter $\delta$ and $\xi_k, k \ge 1$ are independent random variables with distribution $\mu$. $W_{0,1,0,\delta \mu} $ is evidently a Brownian motion on the right half-line with jumps from $0$.

Coming back to the general case of $\ka \ge 1$, let $i$ be fixed, and let $V(t) \coloneqq U_i(\beta_i^{-1}t), t \ge 0$, so that, by \eqref{wab:3},  
\[ V(t) = t + {\textstyle \sum_{k=1}^{N(t)}}\xi_{k,i}, \qquad t\ge 0, \]
but this time $N$ is a Poisson process with parameter $\delta \beta_i^{-1}$.  Then 
\[X \coloneqq B_{i}  + V \circ V^{-1} \circ \mc L_i ,\]
being of the form \eqref{dzis:3},
is a Brownian motion on the right half-line with jumps from $0$, as described above. More precisely, $X= W_{0,1,0,\beta_i^{-1}\delta \mu_i}$ where $\mu_i$ is the restriction of $\mu$ to the $i$th ray. Theorem \ref{thm:wba} says therefore that the local time $\grae \tfrac{1}{2\beta_i\eps}\int_0^t\1_{\{0<X(s)<\eps\}}\ud s,  t \ge 0$ of $X$ exists and equals $V^{-1}\circ \mc L_i
= \beta_i U_i^{-1} \circ \mc L_i$: 
\[ \beta_i U_i^{-1} \circ \mc L_i (t) = \grae \tfrac{1}{2\eps}\int_0^t\1_{\{0<X(s)<\eps\}}\ud s, \qquad t \ge 0.\]
Also, since  $ V \circ V^{-1}=U_i \circ U_i^{-1}$, \eqref{wab:5} says that $(W_{0,\beta,0,\delta \mu})_i =X\circ T_i$, and thus our task reduces to showing that, on the one hand, 
\begin{equation}\label{dzis:4} \int_0^t\1_{\{0<X \circ T_i(s) <\eps\}}\ud s= \int_0^{T_i(t)}\1_{\{0<X_i(u)   <\eps\}}\ud u, \qquad t \ge 0 \end{equation}
because, to recall, $L = U_i \circ \mc L_i \circ T_i$.

To prove this relation, we note first that, by Lemma \ref{det} (a), for all $j \in \kad$, $T_j(t)\geq \int_0^t\1_{\{(W_{0,\beta,0,\delta \mu})_j(s) >0 \}}\ud s, t \ge 0$; indeed, the lemma says that $T_j$ grows linearly with slope $1$ whenever $(W_{0,\beta,0,\delta \mu})_j$ is positive (the lemma is presented in Section \ref{sec:stocha}, but its proof does not use any of the results obtained in the meantime; in fact, the lemma is purely deterministic). On the other hand, $\sum_{j \in\kad }T_j(t)=t $ and  $\sum_{j \in\kad }\int_0^t\1_{\{ (W_{0,\beta,0,\delta \mu})_j(s) >0\}}\ud s=\int_0^t\1_{\{ |W_{0,\beta,0,\delta \mu}(s)| >0\}}\ud s = t, t\geq 0$. Hence, 
\(
T_j(t)=\int_0^t\1_{\{(W_{0,\beta,0,\delta \mu})_i(s) >0\}}\ud s 
\), for all indexes $j\in\kad$ and in particular for $j=i$:  
\[ 
T_i(t)=\int_0^t\1_{\{ X \circ T_i(s) >0\}}\ud s, \qquad t \ge 0
.\]
Also, since $\1_A= (\1_A)^2$  for any indicator function $\1_A$, the left-hand side of \eqref{dzis:4} equals
\(
\int_0^t (\1_{\{0<X\circ T_i(s) <\eps\}})^2 \ud s\) and this, by the relation just established is $\int_0^t\1_{\{0<X\circ T_i(s) <\eps\}}\ud T_i(s)\).  This completes the proof of \eqref{dzis:4} because the last expression can be rewritten, by the change of variables, as 
\(\int_0^{T_i(t)}\1_{\{0<X_i(u)   <\eps\}}\ud u
\). 
\end{proof}
\subsection{Extensions}\label{sec:extend}

\newcommand{\obeta}{\overline \beta}
Throughout Section \ref{sec:wab} we have assumed that (a) all betas are positive, and (b) they add up to $1$ (in Section \ref{sec:dow} only (b) is used). Again, (b) is made without loss of generality, for convenience and to have a natural interpretation of the betas. Indeed,   
the boundary condition \eqref{loo:0} with $\mum=\delta \mu$ remains the same if $\alpha, \beta, \gamma$ and $ \delta $ are multiplied by a positive constant. Hence, if (a) holds   we can let $\obeta \coloneqq \sum_{i\in \kad} \beta_i >0$ and 
\begin{equation}\label{marek:1} W_{\alpha,\beta,\gamma,\delta \mu} \coloneqq W_{\alpha',\beta',\gamma',\delta' \mu},  \end{equation}
where $\alpha',\beta',\gamma'$ and $\delta'$ are obtained by dividing $\alpha,\beta,\gamma$ and $\delta$ by $\obeta$.  

Notably, as a result of this definition, Section \ref{sec:dow} now applies also to the case where all betas are positive but $\overline \beta$ need not be $1$. Theorems \ref{thm:wba} and \ref{olepkimizabitym} similarly remain true without any essential changes even if (b) is violated.  Elaborating on this succinct statement, suppose (a) is satisfied but $\obeta \not =1$. Then, by definition, as in \eqref{wab:5} and \eqref{wab:3},
 \begin{equation*}
 (W_{0,\beta,0,\delta \mu})_i = (B_{i}  + U_i^{\sharp} \circ (U_i^\sharp)^{-1} \circ \mc L_i ) \circ T_i,  \qquad  i\in \kad. \end{equation*}
 where $U_i^\sharp (t) \coloneqq  (\beta_i/\obeta) t + {\textstyle \sum_{k=1}^{N^\sharp(t)}}\xi_{k,i}, t\ge 0 $
 and $N^\sharp$ is a Poisson process with rate $\delta/\obeta$. However, the formula $N (t) = N^\sharp (\obeta t), t \ge 0$ defines a Poisson process with rate $\delta$, and introducing $U_i (t) = \beta_i t +  {\textstyle \sum_{k=1}^{N(t)}}\xi_{k,i}, t\ge 0 $, we see that $U_i^\sharp (t)= U_i (t/\obeta), t \ge 0$. Consequently, $(U_i^\sharp)^{-1}=\obeta U_i^{-1}  $ and $U_i \circ U_i^{-1} = U_i^{\sharp} \circ (U_i^\sharp)^{-1}$, $i\in \kad$. Moreover, because of the first of these relations, the requirement that $(U_i^\sharp)^{-1} \circ \mc L_i \circ T_i$ does not depend on $i$  is equivalent to the requirement that $U_i^{-1} \circ \mc L_i \circ T_i$ does not depend on $i$, and thus, together with other conditions listed in Theorem \ref{thm:wba}, determines the same $T_i$s. As a result, 
\begin{equation}\label{mzoe:1} (W_{0,\beta,0,\delta \mu})_i = (B_{i}  + U_i \circ U_i^{-1} \circ \mc L_i ) \circ T_i,  \qquad  i\in \kad; \end{equation}
this means that, even if $\obeta \not =1$, formula \eqref{wab:5} is still in force, and neither $\beta_i$ nor the rate of the Poisson process need to be modified by dividing them by $\obeta$ --- this is what we mean by saying that Theorem \ref{thm:wba} remains essentially the same. 

However, it should be noted that, by the relation $(U_i^\sharp)^{-1}=\obeta U_i^{-1}  $ established above, the new local time $L \coloneqq U_i^{-1} \circ \mc L_i \circ T_i$ does not coincide with the old $L^\sharp \coloneqq (U_i^\sharp)^{-1} \circ \mc L_i \circ T_i$. Rather, we have, 
\begin{equation}\label{ele} L^\sharp = \obeta L  .\end{equation} But this state of affairs has positive consequences.  First of all, to obtain $W_{\alpha,\beta,0,\delta \mu}$ from $W_{0,\beta,0,\delta \mu}$ of \eqref{mzoe:1} we should, by definition, slow it down at $\zero$  by means of the inverse of $A_\alpha^\sharp $ defined by $A_\alpha^\sharp (t)\coloneqq t + (\alpha/\obeta )L^\sharp (t), t \ge 0$.   Since, by \eqref{ele}, $A_\alpha^\sharp$ coincides with $A_\alpha $ defined by $A_\alpha (t) = t +\alpha L(t),  t\ge 0$ the description of $W_{\alpha,\beta,0,\delta \mu}$ given in Theorem \ref{olepkimizabitym} applies also to the case where $\obeta \not =1$. Finally, to obtain $W_{\alpha,\beta,\gamma,\delta \mu}$, we should kill $W_{\alpha,\beta,0,\delta \mu}$ at the instant when $L^\sharp$ reaches the level of a random variable $\zeta$ that is exponentially distributed with parameter $\gamma/\obeta$. Since this is the instant when $L$ reaches the level of $\zeta/\obeta$ and $\zeta/\obeta$ is exponentially distributed with parameter $\gamma$, we conclude that also the description of $W_{\alpha,\beta,\gamma,\delta \mu}$ given in Theorem \ref{olepkimizabitym} applies to the case where $\obeta$ is not necessarily~$1$.

To summarize our discussion:  as long as all $\beta_i$s are positive, $W_{\alpha,\beta, \gamma, \delta \mu}$ can be characterized as follows. Let $\jcg{\xi_n}$ be a sequence of independent $\scan$-valued random variables with common distribution $\mu$ and $N$ be an independent Poisson process with intensity $\delta$.   Also, let 
\begin{equation}\label{radek:1} U_i(t) =\beta_i t + {\textstyle \sum_{k=1}^{N(t)}}\xi_{k,i},  \qquad t\ge 0,i \in \kad . \end{equation} 
Then there is a $\ka$-dimensional time change process $\seq{T_i}$ uniquely determined  by the requirements that $\sum_{i=1}^\ka T_i(t)=t, t \ge 0$ and $U_i^{-1}\mc \circ \mc L_i \circ T_i$ does not depend on $i\in \kad$ ($\mc L_i$s are defined in \eqref{wab:4}). Moreover, the symmetric local time $L$ of $W_{0,\beta,0,\delta\mu}$ (see \eqref{eq:L_loc_W}) coincides with $U_i^{-1}\mc \circ \mc L_i \circ T_i$ and
$W_{0,\beta,0,\delta \mu}$ has the following form:
\begin{equation}\label{radek:2} (W_{0,\beta,0,\delta \mu})_i = B_{i}\circ T_i  + U_i \circ L=(B_{i}  + U_i \circ U_i^{-1} \circ \mc L_i ) \circ T_i, \qquad i \in \kad. \end{equation}
Finally, if $\zeta$ is an independent exponentially distributed random variable with parameter $\gamma$, then
\begin{equation}\label{mzoe:2} 
W_{\alpha,\beta, \gamma,\delta \mu} (t) =
\begin{cases}  W_{0,\beta,0,\delta \mu} \circ A_\alpha^{-1} (t), &  U_i^{-1} \circ \mc L_i \circ T_i(t) < \zeta,\\ 
\text{undefined}, &   U_i^{-1} \circ \mc L_i \circ T_i(t) \ge \zeta ,
\end{cases}
 \end{equation}
where $A_\alpha (t) = t+\alpha L(t), t \ge 0$,  and the symmetric local time of $W_{\alpha,\beta, \gamma,\delta \mu} $ is given by 
\begin{equation}\label{mzoe:3}
K(t) =
\begin{cases}  U_i^{-1} \circ \mc L_i \circ T_i(t), &  U_i^{-1} \circ \mc L_i \circ T_i(t) < \zeta,\\ 
U_i^{-1} \circ \mc L_i \circ T_i(\zeta), &   U_i^{-1} \circ \mc L_i \circ T_i(t) \ge \zeta .
\end{cases}
 \end{equation}

\subsection{Lumping edges together}\label{lumping}
Let, as in Sections \ref{sec:bpy} and \ref{sec:lumping1}, for some natural number $\mathpzc n \le \ka$, 
$\Psi$ map $\kad $ onto $\mathpzc N \coloneqq \{1,\dots,\mathpzc n\}$. Also,  let, by a slight abuse of notation, $\Psi$ denote simultaneously the map $S_\ka \ni (i,x)\mapsto  (\Psi (i), x) \in S_{\mathpzc n}$, and the associated transformation  $D([0,\infty),S_\ka)\to D([0,\infty),S_{\mathpzc n})$ given by $(\Psi \omega) (t) = \Psi (\omega (t)), t\ge 0, \omega \in D([0,\infty),S_\ka)$. 
We claim that 
\begin{equation}\label{let:2} \Psi W_{\alpha,\beta, \gamma, \delta \mu} = W_{\alpha,\wt \beta, \gamma, \delta \wt \mu},\end{equation} 
where 
 $\widetilde \beta = (\widetilde \beta_j)_{j\in \mathpzc N}$ with $\widetilde \beta_j \coloneqq \sum_{\{i\colon \Psi(i)=j\}} \beta_i, j \in \mathpzc N$, and $\wt \mu $ is the image of the measure $\mu$ via $\Psi$. To clarify, \eqref{let:2} means that the distributions of $\Psi W_{\alpha,\beta, \gamma, \delta \mu} $ and $W_{\alpha,\wt \beta, \gamma, \delta \wt \mu}$, as vectors with values in $D([0,\infty), S)$, are the same.

To see this, we recall first of all that, by \eqref{marek:1}, without loss of generality, it can be assumed that $\sui \beta_i=1$.  Then, see Section \ref{sec:dow}, $W_{\alpha,\beta, \gamma, \delta \mu}$ is defined as $\mc C_{\gamma, \delta, \mu}W_{\alpha,\beta}$. This means that (a) before the time $\tau$ when $L\circ A_\alpha^{-1}$ reaches the level of an independent exponential random variable with parameter $\gamma +\delta$, $W_{\alpha,\beta, \gamma, \delta \mu}$ is identical to $W_{\alpha,\beta}$, and (b) at $\tau$, with probability $\frac \delta{\gamma +\delta}$, $W_{\alpha,\beta, \gamma, \delta \mu}$ starts all over again, forgetting its past, at a random point of $S_\ka$, its distribution being $\mu$. Accordingly, (a) as established in Section \ref{sec:lumping1}, before $\tau$, $\Psi W_{\alpha,\beta, \gamma, \delta \mu}$ is identical to $W_{\alpha,\wt \beta}$, and (b) at $\tau$, with probability $\frac \delta{\gamma +\delta}$, $\Psi W_{\alpha,\beta, \gamma, \delta \mu}$ starts all over again, forgetting its past, at a random point of $S_\mathpzc n$, its distribution being $\wt \mu$. Since $L\circ A_\alpha^{-1}$ is a common symmetric local time for $W_{\alpha,\beta}$ and $W_{\alpha,\wt \beta}$, this establishes the claim.

\subsection{Degenerate cases}\label{sec:degene}

As in Section \ref{sec:sui} we need to describe the process also in the degenerate case where some or all betas vanish. Basically, the related process is a concatenation of processes described in Section \ref{sec:sui}. Again, if started at the edge with $i\in \kal$, the process of Section \ref{sec:sui} rather quickly leaves this part of the state-space via origin, never to return. However, as long as $\mu$ has a nonzero mass at this edge, the process related to $\gen_ {\alpha,\beta,\gamma,\delta\mu}$ can come back here by a jump when the local time of the process (which from the time of reaching origin behaves like a possibly slowed down Walsh's process on a subgraph) at the origin reaches an exponential level.  Edges with $\beta_i=0$ and of zero measure $\mu$, form the transient part of the state-space.

If all betas vanish, $\alpha$ is necessarily positive and, as we know from Section \ref{sec:sui}, the process generated by $\gen_ {\alpha,0,\gamma}=\gen_ {1,0,\gamma/\alpha}$ ceases to be defined when it first reaches $0$ to be stopped there, and then stays there for an exponential time. However, that of    
$\gen_ {\alpha,0,\gamma,\delta \mu}$ is resurrected with probability $\frac {\delta}{\gamma +\delta}$ and starts all over again at a random point of the graph, distributed according to $\mu$. If it so happens that some edges are of zero measure $\mu$, they again form the transient part of the state-space, as the process can never return there. We will write more extensively on such cases in our analysis of infinite $\mum$. 

{We leave it to the readers to check that the principle of lumping edges together applies also to these degenerate cases --- in Section \ref{lumping2} we will give an independent proof of this fact.}

\newcommand{\afu}{\gen_ {\textnormal{full}}}
\newcommand{\afup}{\gen_ {\alpha,\beta,\gamma,\mum}}
\newcommand{\wfup}{W_ {\alpha,\beta,\gamma,\mum}}

\section{The case of infinite $ \mum$, analytic part}
\label{sec:tcoim}

We come to the case where $\mum (S) =\infty$; as throughout the paper, we assume that $\int (1\wedge x)\mum (\mud x)$ is finite. Given such a measure, nonnegative real parameters $\alpha$ and $\gamma$, and a nonnegative vector $\beta = \seq{\beta_i}$, we consider the restriction, say, 
\[\afup,\]
of $\Delta$ to the domain of $f\in \coss $ such that 
\begin{equation}\label{zet:1} \tfrac{\alpha}2 f'' (\zero) - \sum_{i\in \kad} \beta_i f_i'(0) + \gamma f(\zero) = \int_S [f(x) - f(\zero)] \, \mum (\mud x)
; \end{equation} note that for $f\in \mathfrak C^2(S)$, the function $S\ni x \mapsto f(x) - f(\zero)$ is $ \mum$ is  integrable.   We stress that, in contrast to the previous sections, $\beta_i$ need not be positive; in fact, in the extreme case we may even have $\alpha + \sui\beta_i =0$. Our main theorem in this section, Theorem \ref{thm:full}, states that $\afup$ is a Feller generator; the related process will be constructed using natural building blocks in Section \ref{sec:stocha}.

The idea of the proof is to approximate the searched-for semigroup generated  by $\afup$  by the semigroups described in Section \ref{sec:finite}. This approximation involves probability measures $\mu_\eps $ constructed from $\mum$ as follows. Since, by assumption, $\int (1\wedge x)\mum (\mud x)$ is finite, we have $\delta (\eps) \coloneqq \mum (\Gamma_\eps) < \infty$ for any $\eps >0$, where $\Gamma_\eps \subset S$ is the set of points in $S$ that lie at a distance $\ge \eps$ from the graph's center. Therefore, for every $\eps>0$, the formula \[  \mu_\eps (\cdot ) = \frac{\mum(\mc \cdot \cap \Gamma_\eps)}{\mum(\Gamma_\eps)}\] defines a Borel probability measure on $S$.  Notably, by assumption, we have $\grae \delta (\eps)$ $= \infty$ and at the same time the measures $\mu_\eps$ converge to the Dirac delta at the graph's center. 

We also need vectors $\beta(\eps)= \seq{\beta_i(\eps)}, \eps >0$ such that
\[ \beta_i(\eps) >0, \qquad i \in \kad, \eps >0 \mqquad{and} \grae \beta_i (\eps) = \beta_i, i \in \kad. \]
This assumption allows thinking of the semigroups generated by 
\[ \gen_\eps \coloneqq \gen_{\alpha,\beta(\eps), \gamma, \delta(\eps)\mu_\eps},\qquad \eps >0,\] described in Section \ref{sec:finite}.

\subsection{The generation theorem}

In terms of the objects introduced above, our main theorem says that the semigroup  generated by $\afup$ can be obtained as a limit, as $\eps \to 0$, of  the semigroups generated by $\gen_\eps$. The following lemma is a first step in this direction.


\newcommand{\rlae}{R_{\lam,\eps}}



\begin{lemma} \label{lem1} As $\eps \to 0$, the resolvents $\rez{\gen _\eps}, \lam>0$ converge strongly to the Feller resolvent $\rla, \lam >0$ given by  
\begin{equation}\label{loo:10} \rla g = \rla^0 g + \overline C (g) \elam^0, \qquad \lam >0, g\in \coss,  \end{equation}
where 
\begin{itemize}
\item [(a) ] $\rla^0, \lam >0$ is the resolvent of the minimal Brownian motion, and $\elam^0=\mathsf 1_S - \lam \rla^0\mathsf 1_S, \lam >0$ is the Laplace transform of its lifetime (see \eqref{lifetime}), 
\item [(b) ] $\overline C(g) \coloneqq \frac{\int_S \rla^0 g\ud \mum + C(g)}{\gamma + \lam \int_S \rla^0 \mathsf 1_S \ud \mum  +  \lam C(\mathsf 1_S)}$, $C(g)\coloneqq  2\slam \sui \beta_iC_i(g)  + \alpha g(\zero)$, and $C_i(g)$s were introduced in \eqref{cee}. 
\end{itemize}
\end{lemma}
\begin{proof}As explained in Section \ref{sec:extend}, the generator of the process characterized by the vector of parameters $(\alpha,\beta (\eps), \gamma,\delta(\eps)\mu_\eps)$ is, by definition, the generator of the process characterized by the vector $(\obeta (\eps))^{-1}(\alpha,\beta (\eps), \gamma,\delta(\eps)\mu_\eps)$, where $\obeta (\eps) =\sui \beta_i(\eps) $. 
Moreover, in \eqref{dow:1} the resolvent of the latter generator is expressed in terms of the resolvent of the generator characterized by 
$(\obeta (\eps))^{-1}(\alpha,\beta (\eps), \gamma +\delta(\eps))$, that is, of $\gen_{\alpha,\beta (\eps), \gamma+\delta(\eps)}$. In the light of this formula we have namely 
\begin{align}\label{dod}
\rez{\gen _\eps} g = \rlae g + M_\eps (g)\ell_{\lam,\eps} \end{align}
where 
\begin{itemize} 
\item $\rlae \coloneqq \rez{\gen_{\alpha,\beta (\eps), \gamma+\delta(\eps)}}$, $\lam >0$,
\item $\ell_{\lam,\eps} \coloneqq \mathsf 1_S - \lam \rlae \mathsf 1_S$, so that --- see \eqref{loo:6} --- for $x \ge 0$ and $ i\in \kad$ we have $(\ell_{\lam,\eps})_i(x) = \frac {\gamma+\delta (\eps)}{\lam \alpha + \slam {\sui} \beta_i (\eps) + \gamma+\delta (\eps)} \e^{-\slam x}$, and 
\item $M_\eps (g) \coloneqq \frac {\delta (\eps) \int_S \rlae g \ud \mu_\eps}{\gamma+ \delta (\eps) \lam \int_S \rlae  \mathsf 1_S \ud \mu_\eps}$.
\end{itemize}
Furthermore, by \eqref{loo:4},
\[ \rlae g = \rla^0 g + E_\eps (g) \ell_{\lam}^0, \qquad \lam,\eps >0, g\in \coss, \]
where 
$E_\eps(g) \coloneqq \frac {2\slam \sui \beta_i (\eps) C_i(g) + \alpha g(\zero)}{\lam \alpha + \slam \sui \beta_i(\eps) + \gamma +\delta (\eps)}$. 

Since, clearly,  $\grae E_\eps (g)=0$, we have $\grae \rlae g= \rla^0g $ and, in particular, $\grae \ell_{\lam,\eps} = \elam^0$. Thus, to establish that $\grae \rez{\gen_\eps}g = \rla g$, it suffices to show that \begin{equation}\label{zoe:1} \grae M_\eps (g) = \overline C(g), \qquad g \in \coss,\end{equation} and this, of course, involves the analysis of $\grae \delta (\eps) \int_S \rlae g \ud  \mu_\eps $. 
For the latter, we note first that, by \eqref{loo:5}, the limit $\lim_{x\to 0+} x^{-1}(\rla^0g)_i (x) $ exists (and equals $2C_i(g)\slam$), $i \in \kad $. This implies that $\rla^0g$ is integrable with respect to $ \mum$, and therefore, by the Lebesgue dominated convergence theorem, $\grae \delta (\eps) \int_S \rla^0 g \ud  \mu_\eps = \grae \int_{\Gamma_\eps} \rla^0 g \ud \mum  = \int_S \rla^0 g\ud \mum $. Moreover, \[ \grae \delta (\eps) E_\eps (g) \int_S \ell_{\lam}\ud  \mu_\eps= C(g)  \grae \int_S \ell_{\lam,\eps} \ud  \mu_\eps =C(g),\] because the measures $ \mu_\eps$ converge to the Dirac delta measure at $\zero$.  This shows that the numerator of $M_\eps(g)$ converges to $\int_S \rla^0 g\ud \mum + C(g)$, whereas its denominator converges to $\gamma + \lam \int_S \rla^0 \mathsf 1_S \ud \mum  +  \lam C(\mathsf 1_S)$. This establishes \eqref{zoe:1}, for the denominator is larger than zero: $\gamma +\lam C(\mathsf 1_S)$ is nonnegative and $ \lam \int_S \rla^0 \mathsf 1_S \ud \mum  = \int_S (\mathsf 1_S -\elam^0) \ud \mum  $ is strictly positive because $\elam^0 <1$ everywhere outside of $\zero$, that is, almost everywhere with respect to $\mum$. 

Having established $\grae \rez{\gen_\eps}g= \rla g, g \in \coss $, we note next that $\rla, \lam >0$,
as a limit of Feller resolvents,  is a pseudoresolvent of nonnegative operators with $\|\lam \rla \|\le 1, \lam >0$, and check that $\lam \rla \mathsf 1_S \le \mathsf 1_S$ with equality for $\gamma =0$. Therefore, our task comes down to proving that 
\begin{equation}\label{lilek} \lil \lam \rla g =g , \qquad g\in \coss.\end{equation}
To this end, we claim first that 
\begin{equation}\label{wikusia} \lil \lam \int_S \rla^0 \mathsf 1_S \ud \mum  =\infty.\end{equation} 
Indeed,   introducing $\overline \mum$, a Borel measure on $[0,\infty)$, 
as the image of $\mum$ via the map $S\ni (x,i) \mapsto x$, we  see that $\lam \int_S \rla^0 \mathsf 1_S \ud \mum $ equals $\int_{[0,\infty)} (1-\e^{-\slam x})\overline \mum (\mud x)$, and thus converges to $\infty$ as $\lam \to \infty$, because by assumption $\mum$ is infinite, and thus so is $\overline \mum$.  

Next, we write   
\(\mathsf 1_S - \lam \rla \mathsf 1_S = (  1-\lam \overline C(\mathsf 1_S)) \elam^0 = \frac{\gamma}{\gamma + \lam \int_S \rla^0 \mathsf 1_S \ud \mum +C(\mathsf 1_S)} \elam^0\) and recall that $\|\elam^0\|_{\coss} = 1$. Since this, when combined with \eqref{wikusia}, yields $\lil \lam \rla \mathsf 1_S = \mathsf 1_S$,  in proving \eqref{lilek} we may restrict ourselves to $g\in \cosso$, that is, to $g$ with $g(\zero)=0$. But $\rla^0, \lam >0$ is the resolvent of the minimal Brownian motion semigroup, and hence we have $\lil \lam \rla^0 g =g $ for $g\in \cosso$. Because of $\|\elam^0\|_{\coss}=1$, this reduces our task to establishing that 
\begin{equation}\label{dom:2}
    \lil \frac {\lam \int_S \rla^0 g \ud \mum }{\gamma + \lam \int_S \rla^0 \mathsf 1_S \ud \mum } = 0, \qquad g \in \cosso.
\end{equation}
We note furthermore that, since $\|\lam\rla \|\le 1 $, it suffices to prove $\lil \lam \rla g$ $=g$, or, equivalently,  
the formula above, for $g$ in a dense subset of $\cosso$, and in particular, we can choose this subset to be the common range of $\rla^0, \lam >0$. 

However, for $g$ of the form $g=\rmi^0 h,$ where $ h \in \cosso$ and $\nu >0$ is fixed, the numerator in \eqref{dom:2} equals
\[ \int_S \nu \rla^0 \rmi^0 h \ud \mum  + \int_S \rmi^0 h \ud \mum  - \int_S \rla^0 h \ud \mum   \]
{by the Hilbert equation.}
If $h$ is nonnegative, the first and the third terms here are nonincreasing functions of $\lam$, and thus have finite limits, whereas the second does not depend on $\lam $ (and is finite). It follows that this expression has a finite limit as $\lam \to \infty$ for all $h\in \cosso$.  To complete the proof, therefore, it suffices to note that, by \eqref{wikusia}, the denominator in \eqref{dom:2} converges to $\infty$ as $\lam \to \infty$. 
\end{proof}

Our next result identifies $\afup$ as the generator of the Feller semigroup related to $\rla , \lam >0$ of \eqref{loo:10}. 

\begin{thm}\label{thm:full} The operator $\afup$ is a Feller generator, and its resolvent coincides with $\rla, \lam >0$. Moreover,
$\grae \e^{t\gen_ \eps} = \e^{t\afup}$ strongly and uniformly for $t$ in compact intervals.
\end{thm}
\begin{proof}
  Since $\rla, \lam >0$ is a Feller resolvent, there is an operator, say, $\gen$, such that $\rez{\gen}=\rla, \lam >0$. This operator's domain is the common range of $\rla, \lam >0$. We claim that 
\[ \dom{\gen} \subset \dom{\afup}.\]  
  
First of all, by \eqref{loo:10}, $\dom{\gen}\subset C^2(S)$, because both $\rla^0 g, g \in \coss $ and $\elam^0$ belong $C^2(S)$.  Hence, it suffices to check that for $f\in \dom \gen$, that is, for $f$ of the form $f=\rla g$ for some $g\in \coss$, condition \eqref{zet:1} is satisfied. Let $f$ be of this form. We already know (see the proof of Lemma \ref{lem1}) that $S\ni x\mapsto \rla^0g (x) - \rla^0g (\zero) = \rla^0g (x)$ is $ \mum$   integrable. Since, $\lim_{x\to 0+}(\e^{-\slam x} - 1)/x=-\slam $ the same is true of $S\ni x\mapsto \elam^0 (x) - \elam^0 (\zero)$, and we have, on the one hand, 
\begin{align}
\int_S [f(x) - f(\zero)]\mum (\mud x)& = \int_S \rla^0g \ud \mum  + \overline{C}(g) \int_S (\elam^0(x) - \elam^0 (\zero) )\mum (\mud x)\nonumber \\
&= \int_S\rla^0g \ud \mum  - \lam \overline{C} (g)\int_S \rla^0 \mathsf 1_S \ud \mum  \label{loo:11}.
 \end{align}  
On the other hand, $(\rla^0g)_i'(0)= 2C_i(g) \slam $ and $(\rla^0g)''(0)= -2g(\zero)$, showing that, for $F$ of \eqref{loo:1}, $F\rla^0g = - \alpha g(\zero) - 2\slam \sui \beta_i C_i(g) = - C(g)$. In particular,  because of $\ell_\lam^0 = \mathsf 1_S - \lam \rla^0 \mathsf 1_S,$ we have $F\ell_\lam^0=\gamma + \lam C(\mathsf 1_S)$, and, consequently,
\(
Ff = - C(g) +  \overline C(g)(\gamma + \lam C (\mathsf \mathsf 1_S)). \) A look at the definition of $\overline C(g)$ reveals now that the  last expression coincides with the right-most side in \eqref{loo:11}, that is, that $f$ satisfies \eqref{zet:1}, as claimed, and so the desired inclusion is established.   

Next, we will argue that $\afup$ extends $\gen$. Indeed, any $f \in \dom{\gen}$ is of the form $f =\rla g, g\in \coss $ and so, by definition, $\gen  f=\gen \rez{\gen}g= \lam \rez{\gen} g - g= \lam f - g$. At the same time, this $f$ belongs, as we have just proved, to $\dom{\afup}$, and a calculation, based on \eqref{loo:10} and $\Delta \elam^0 =\lam \elam^0$, shows that $\lam f - \afup f = g$, that is, $\afup f= \lam f - g$. This implies that $\afup f = \gen f, f \in \dom{\gen}$, as heralded above.

However, $\afup$ cannot be a proper extension of $\gen$, because $\afup$ satisfies the positive-maximum principle (see Appendix \ref{maks}) and is thus dissipative, and so we conclude that $\afup =\gen$. In particular, $\afup$ is a Feller generator with resolvent $\rla, \lam >0$.   
The rest therefore follows by the Trotter--Kato--Neveu theorem \cite{abhn,kniga,knigazcup,ethier,goldstein,kallenbergnew,pazy}. 
\end{proof}

\subsection{Continuous dependence on parameters}\label{sec:continuous}
As a direct corollary to Lemma \ref{lem1} and Theorem \ref{thm:full}, we obtain the following formula for the resolvent of $\afup$. We have 
\begin{equation}\label{joj:1} \rez{\afup}g = \rla^0 g + \overline C (g) \elam^0, \qquad \lam >0, g\in \coss
\end{equation}
where, as before, \[ \overline C(g) = \frac{\int_S \rla^0 g\ud \mum + C(g)}{\gamma + \lam \int_S \rla^0 \mathsf \mathsf 1_S \ud \mum  +  \lam C(\mathsf 1_S)} \text{ and } C(g)=  2\slam \sui \beta_iC_i(g)+ \alpha g(\zero),\] 
or, in the expounded form,
\[ \rez{\afup}g = \rla^0 g + \frac{ \alpha g(\zero) + 2\slam \sui \beta_iC_i(g)+\int_S \rla^0 g\ud \mum }{\lam \alpha + \slam \sui \beta_i +\gamma + \lam \int_S \rla^0 \mathsf \mathsf 1_S \ud \mum  } \elam^0.
\]

Notably, as we will now argue, this formula is valid for all possible choices of parameters. Indeed, for $\mum=0$, $\overline C(g) $  coincides with $E(g)$ of \eqref{loo:3} and so \eqref{joj:1} reduces to \eqref{loo:4}. Also in the limit case of $\alpha=\beta_i=0, i \in \kad $ and $\mum =0$ but $\gamma >0$, we have  $\overline C(g)=0$ and thus the right hand side in \eqref{joj:1} reduces to the minimal Brownian motion. Finally, to prove that \eqref{joj:1} includes also the case where $\mum$ is a nonzero but finite measure, one could transform \eqref{dow:1} to the desired form with the help of \eqref{loo:4}, but this approach turns out to be {computationally demanding}. It is relatively easier to note that the argument presented in the proof of Theorem \ref{thm:full} and showing that the right-hand side of \eqref{loo:10}, that is, the right-hand side of \eqref{joj:1}, is a solution to the related resolvent equation, works well also in the case of our interest.  Since such a solution is necessarily unique, we conclude that, as claimed, \eqref{dow:1} is a special case of \eqref{joj:1} but in a disguise.

Once we note that \eqref{joj:1} encompasses all the possible choices of $\alpha,\beta,\gamma$ and $\mum$, it is rather easy to see that the solution to the resolvent equation depends continuously on these parameters, and thus so do the corresponding semigroups and processes. Indeed, $C(g)$ depends continuously on $\alpha\ge 0$ and $\beta \ge 0$ and therefore $\overline C(g)$ depends continuously on these two parameters and on $\gamma \ge 0$ in the region where $\overline C(g)$'s denominator is nonzero. This region is composed of two subsets: in the first of these  $\alpha +\sui \beta_i>0$, in the second $\alpha +\sui \beta_i=0$ but $\gamma >0$ and/or $\mum$ is infinite (in fact, it suffices to assume that $\mum$ is nonzero, but the case of finite $\mum$ and  $\alpha +\sui \beta_i=0$ does not correspond to a Feller resolvent, see below). Moreover, $\overline C(g)$ depends continuously on $\mum$ in the following sense: if for certain L\'evy measures $\mum$ and $\mum_n, n \ge 1$, $\gra \int\rla^0g \ud \mum_n = \int\rla^0 g \ud \mum, g \in \coss, \lam >0$, then the same is true for the corresponding $\overline C(g)$s,  $g\in \coss$.

\begin{remark}
    It follows from   Theorems 2.6 and 2.8, and the proof of Theorem 2.10 in
Chapter V of \cite{blum} that  if $ \gamma=\beta_i=0,i\in\kad$ and $ \mum$   is an infinite measure on $\scan \setminus\{\zero\}$ such that $\int_{\scan} (1\wedge |x|) \mum (\mud x)<\infty,$ then \eqref{joj:1}  defines a resolvent of a Feller process whose excursion measure is equal up to a multiplicative constant to $\int_{\scan} \bar P_x \mum (\mud x).$ Here $\bar P_x$ is the distribution of a Brownian motion on $\scan$ that is stopped upon hitting $\zero$. 
\end{remark}

\subsection{Comparison to Rogers' work} 
\newcommand{\enla}{\mathpzc n_\lam}
\newcommand{\enmu}{\mathpzc n_\mu} It is worthwhile to look at  \eqref{joj:1} in the light of the seminal work of L.C.G. Rogers \cite{rogersik}, devoted to his exposition of It\^o excursion theory via resolvents.
In our context, the main result of this paper  can be stated as follows (see Rogers' Theorems 1 and 2): the formula 
\begin{equation}\label{rogers} \rla g = \rla^0 g + \overline C (g) \elam^0, \qquad \lam >0, g\in \coss
\end{equation}
defines a Feller resolvent $\rla, \lam >0$ iff there are  nonnegative constants $\alpha$ and $\gamma$, and  a family $\enla, \lam >0$ of finite measures concentrated on $S \setminus \{\zero\}$ such that
\begin{equation}\label{caod:6} \enla \mathsf 1 >0, \quad (\mu-\lam)\enmu \rla^0 = \enla- \enmu  , \qquad \lam, \mu >0 \end{equation} 
and $\overline C(g)$ is of the form $\overline C (g)= \frac {\enla g+\alpha g(0)}{\lam \alpha  + \lam \enla \mathsf 1 +\gamma}$ .
In \eqref{caod:6}, measures $\enla, \lam >0$ are  seen as functionals on $\coss$; as explained by Rogers, they are Laplace transforms of entrance laws for the minimal process. 

In this light, \eqref{joj:1} can be summarized by saying that in the case of Brownian motion on a star graph the family $\enla,\lam>0$ is necessarily of the form 
\begin{equation}\label{caod:7} \enla g =  \sui  \beta_i C_i (g) + \int_{S\setminus \{\zero\}} \rla^0 g \ud \mum, \qquad \lam >0, g \in \coss . \end{equation} 
This representation can also  be proved directly by arguing as in Section 3 of \cite{rogersik}; to this end, not surprisingly, one needs to use the fact (visible from \eqref{loo:5} and already referred to above) that, for each $i\in \kad$ and $g\in \coss$, the limit $\lim_{x\to 0+} x^{-1}(\rla^0g)_i (x) $ exists and equals $2C_i(g)\slam$.

\subsection{Lumping edges together} \label{lumping2} 
Formula \eqref{joj:1}, besides sharing a similar form with \eqref{loo:4}, inherits also from its predecessor a property that is worth recording again here, when we deal with the more general situation.  To wit, in the notations of Section \ref{sec:lump0}, we have
\begin{equation}\label{let:1a} \mc I^{-1} \rez{\gen_ {\alpha,\beta,\gamma,\mum}} \mc I = \rez{\gen_{\alpha,\widetilde \beta, \gamma,\wt \mum}}, \qquad \lam >0, \end{equation}
where $\gen_{\alpha,\widetilde \beta, \gamma,\wt \mum}$ is the generator in $\mathfrak C(S_\mathpzc n)$ with $\widetilde \beta = (\widetilde \beta_j)_{j\in \mathpzc N}$ given by  $\widetilde \beta_j 
 \coloneqq \sum_{\{i\colon \Psi(i)=j\}} \beta_i, j \in \mathpzc N$, and $\wt \mum$ is the image of the measure $\mum$ via the map $\Psi$.

This suggests that formula \eqref{let:2}, established for finite $\mum$, extends to infinite $\mum$ as well. That is, 
\begin{equation}\label{lump} \Psi W_{\alpha,\beta, \gamma, \mum} = W_{\alpha,\wt \beta, \gamma, \wt \mum},\end{equation} 
whatever the parameters are, and, again, this means that the distributions of both processes involved, seen as having values in $D([0,\infty), S_\ka)$, are the same.

To prove \eqref{lump}, we fix $\alpha, \beta, \gamma $ and $\mum$ (such that at least one of the conditions $\alpha +\sui \beta_i >0$ and $\mum (S_\ka)=\infty$ holds), and for each integer $n$ choose a positive vector $\beta_n$, a constant $\delta_n$ and a probability measure $\mu_n$ such that \begin{equation}\label{wnuk1} \gra \beta_n=\beta  \mqquad{and}  \gra \delta_n \int\rla^0g \ud \mu_n = \int\rla^0 g \ud \mum, g \in \mathfrak C(S_\ka), \lam >0.\end{equation} As explained in Section \ref{sec:continuous}, the semigroups generated by $\gen_{\alpha,\beta_n, \gamma, \delta_n \mu_n}$ converge then to the semigroup generated by $\afup$. Hence, by the Trotter--Sova--Kurtz--Macke\-vi\v cius Theorem (see e.g. \cite{kallenbergnew} p. 385, see also Theorem 2.5 in Chapter 4 in \cite{ethier}), the
processes $W_{\alpha,\beta_n, \gamma, \delta_n \mu_n}$    converge in distribution to $W_{\alpha,\beta,\gamma,\mum}$ in $D([0,\infty),S)$ (provided that they start from the same point), that is, for any continuous, bounded, real-valued function $\mathsf f$ on $D([0,\infty),S_\ka)$ and $x\in S_\ka$, 
\begin{equation}\label{wnuczek} \gra \mathsf E_x \mathsf f(W_{\alpha,\beta_n, \gamma, \delta_n \mu_n} ) =   \mathsf E_x \mathsf f(W_{\alpha,\beta, \gamma, \mum}).\end{equation}
Similarly, since \eqref{wnuk1} forces 
 \begin{equation*} \gra \wt {\beta_n}=\wt \beta  \mqquad{and}  \gra \delta_n \int\rla^0g \ud {\wt \mu_n} = \int\rla^0 g \ud \wt \mum, g \in \coss, \lam >0,\end{equation*}
we  have also 
\[ \gra \mathsf E_x \mathsf g(W_{\alpha,\wt{\beta_n}, \gamma, \delta_n \wt{\mu_n}} ) =   \mathsf E_x \mathsf g(W_{\alpha,\wt \beta, \gamma, \wt \mum}),\]
for any bounded, real-valued continuous function $\mathsf g$ on $D([0,\infty),S_{\mathpzc n})$ and $x\in S_{\mathpzc n}$.

Since \eqref{let:2} applies to  all the approximating processes, the obvious idea to complete the proof of \eqref{lump} is by approximation. To this end, we take a closer look at $\Psi$. To begin with, $S_\ka$ is equipped with the following natural metric: the distance between $(i,x)$ and $(j,y)$ is $x+y $ if $i\not =j$ and $|x-y|$ otherwise (we have used this metric in Section \ref{sec:dynkin} without spelling out the details). It is clear that $\Psi\colon S_\ka \to S_{\mathpzc n} $ is a non-expansive mapping with respect to this metric --- to see this it suffices to note that $|x-y|\le x+y $ for all nonnegative $x$ and $y$. The definition of the metric in $D([0,\infty), S_\ka)$ (see formulae (12.13) and (16.4) in \cite{bill2}) thus implies that $\Psi$ is also non-expansive as a map from 
$D([0,\infty),S_\ka)$ to $D([0,\infty),S_{\mathpzc n})$. It follows, that $\mathsf g \circ \Psi $ is a real-valued, bounded, continuous function on $D([0,\infty),S_{\ka})$ whenever $\mathsf g$  is a real-valued, bounded,  continuous function on $D([0,\infty),S_{\mathpzc n})$. 

This allows writing, for any $\mathsf g$ described above and $x\in S_{\mathpzc n}$, 
\begin{align*}\mathsf E_x \mathsf g(\Psi W_{\alpha,\beta, \gamma, \mum})& =  \mathsf E_x \mathsf g\circ \Psi( W_{\alpha,\beta, \gamma, \mum}) = \gra \mathsf E_x \mathsf g\circ \Psi (W_{\alpha,\beta_n, \gamma, \delta_n \mu_n} )\\& = \gra \mathsf E_x \mathsf g(\Psi W_{\alpha,\beta_n, \gamma, \delta_n \mu_n} ) = \gra \mathsf E_x \mathsf g( W_{\alpha,\wt {\beta_n}, \gamma, \delta_n \wt{\mu_n}} )  \\&= \mathsf E_x \mathsf g(W_{\alpha,\wt \beta, \gamma, \wt \mum}) . \end{align*}
Since $\mathsf g$ and $x$ are arbitrarily chosen, this completes the proof of \eqref{let:2}.



\subsection{The singular case}\label{sec:degenerate}
It remains to comment on the right-hand side of \eqref{joj:1} in the case where $\alpha +\sui \beta_i=0$ and $\mum$ is a nonzero but finite measure, so that $\mum = \delta \mu$ for a probability measure $\mu$ and $\delta =\mum (S) >0$; moreover, $C(g)=0$ and $\overline C(g)= \frac{\delta \int_S \rla^0 g\ud \mu}{\gamma + \lam \delta \int_S \rla^0 \mathsf \mathsf 1_S \ud \mu}$. The formula 
\begin{equation}\label{zez:2} \rla g \coloneqq \rla^0 g + \overline C (g) \elam^0, \qquad \lam >0, g\in \coss
\end{equation}
then still defines a pseudo-resolvent, that is, the operators $\rla, \lam >0$ satisfy the Hilbert equation. Moreover, $\lam\rla$s are positive contractions and $\lam \rla \mathsf 1_S \le \mathsf 1_S$ (with equality for $\gamma =0$). Nevertheless, $\rla , \lam >0$ is not a Feller resolvent, because condition 2. of the definition fails. 

To see this  we note that 
the calculation presented in the proof of Theorem \ref{thm:full} carries over to the case of interest with no substantial changes, and thus shows that the range of $\rla, \lam >0  $ is contained in the kernel of the bounded linear functional  $\mathpzc F\in (\coss )^*$ given by 
\[ \mathpzc F g = \delta \int_S g \ud \mu - (\gamma+\delta) g(\zero), \qquad g \in \coss;\]
this excludes the possibility of $\lil \lam \rla g =g$ for $g$ with $\mathpzc F g \not =0$. 

As expounded in the following proposition, one can still find a semigroup corresponding to $\rla, \lam >0$, but this semigroup is defined merely on 
the kernel of $\mathpzc F$, that is, on a (closed) subspace of $\coss$. 

\begin{proposition}\label{thm:dege}  Let $\cossf$ be the subspace of $g\in \coss$ such that $\mathpzc F g=0$.  
There is a (necessarily unique) strongly continuous semigroup $\rod{T(t)}$ of positive contractions in $\cossf$ such that 
\begin{equation} \rla g = \int_0^\infty \e^{-\lam t} T(t)g \ud t, \qquad \lam >0, g\in \cossf.\label{zez:3} \end{equation}
Moreover, $T(t)\mathsf 1_S = 1_S, t\ge 0$ iff $\gamma =0$.

\end{proposition}
\begin{proof} We know that the restrictions of the operators $\rla, \lam >0$ to $\cossf$ form a pseudo-resolvent in this space, and that the norm of $ 
\lam\rla |_{\cossf} $ does not exceed $1$ for $\lam >0 $. Hence, by the Hille--Yosida Theorem, it  suffices to show that 
\begin{equation}\label{zez:5} \lil \lam \rla g = g, \qquad g \in \cossf.\end{equation} 

To this end, we fix a $\nu>0$ and choose a constant $c$ in such a way that $\mathsf 1^\sharp \coloneqq \mathsf 1_S  - c\ell_\nu^0$ belongs to $\cossf$, that is, we take $c\coloneqq \frac {\gamma }{\gamma +\delta -\delta \int_S \ell_\nu^0\ud \mu }$; note that $ \int_S \ell_\nu^0\ud \mu<1$ and so the denominator above is positive. Since $\elam^0, \lam >0$ is an exit law for $\rla^0, \lam >0$, we have 
\begin{equation}\label{zez:6} 
(\lam - \nu) \rla^0\ell_\nu^0 = \ell_\nu^0 - \elam^0, \qquad \lam >0. \end{equation}
This equation (a counterpart of the Hilbert equation), together with $\ell_\lam^0 = \mathsf 1_S - \lam \rla^0 \mathsf 1_S$ yields
\[\ell_\nu^0 - \lam \rla \ell_\nu^0 = (1-\lam C(\ell_\nu))\ell_\lam^0 - \nu R_\nu \ell_\nu^0 \mqquad{and} \mathsf 1_S - \lam \rla \mathsf 1_S = (1-\lam C(\mathsf 1_S))\ell_\lam^0, \]
for all $\lam >0$. Therefore, 
\begin{align*} \mathsf 1^\sharp - \lam \rla \mathsf 1^\sharp &= [1-\lam C(\mathsf 1_S) - c(1-\lam C(\ell_\nu^0))] \elam + c\nu \rla^0 \ell_\nu \\
&= \frac {\gamma - c(\gamma +\delta - \int_S (\nu \rla^0 \ell_\nu^0 + \ell_\nu^0)\ud \mu)}{\gamma +  \delta - \delta \int_S\elam^0 \ud \mu}\ell_\lam^0 + c \nu \rla^0\ell_\nu^0,\quad \lam >0,\end{align*}
where \eqref{zez:6} was used once more. Also $\lil \int_S\elam^0\ud \mu =0$ (because $\mu (\{\zero\})$ is $0$), implying that the denominator in the fraction next to $\ell_\lam^0$ above converges to $\gamma +\delta>0$; at the same time, by $\|\rla^0\|\le \lam^{-1}$, the numerator converges to $\gamma - c(\gamma +\delta - \int_S \ell_\nu^0\ud \mu) =0$.  The same estimate on $\|\rla^0\|$ together with   
$\|\elam\|_{\coss} =1$ allows us thus to conclude that the last expression in the display converges to $0$, as $\lam \to \infty$, that is, that $\lil \lam \rla \mathsf 1^\sharp = \mathsf 1^\sharp$. 

The problem therefore reduces to establishing that \eqref{zez:5} holds for  $g\in \cossf $ such that $g(\zero)=0$. For such $g$, $\lil \lam \rla^0 g = g$ and thus, because of \eqref{zez:2} and  $\|\elam\|_{\coss} =1$, we need to show that $\lil \frac{\delta \lam \int_S \rla^0 g\ud \mu}{\gamma + \delta \lam \int_S \rla^0 \mathsf \mathsf 1_S \ud \mu}=0$. Now, $\lil \lam \int_S \rla^0 \mathsf 1_S \ud \mu = 1 - \lil \int_S \ell_\lam \ud \mu =1$, and so the denominator in this fraction converges to $\gamma +\delta$.  At the same time, due to $\lil \lam \rla^0 g =g$, the numerator converges to $\delta \int_S g \ud \mu$. This integral equals zero,  because $g $ belongs to the kernel of $\mathpzc F$ and we have $g(\zero)=0$. This completes the proof of \eqref{zez:5}. 

Finally, if $\gamma >0$ there is no sense of talking about $T(t)\mathsf 1_S$ because $\mathsf 1_S$ does not belong to $\cossf$. On the other hand, if $\gamma =0$, a straightforward calculation shows that  $\overline C(\mathsf 1_S) = \lam^{-1}$. This implies $\lam \rla \mathsf 1_S =\mathsf 1_S, \lam >0$ and, hence, $T(t)\mathsf 1_S = \mathsf 1_S, t \ge 0$.  \end{proof}

We end this section by noting that the semigroup of Proposition \ref{thm:dege} is related to the following `singular' process. 
On each edge it behaves like a Brownian motion, but at the instant when it touches $\zero$, is killed with probability $\frac \gamma{\gamma +\delta}$ or, with probability $\frac \delta{\gamma +\delta}$,  starts all over again   
at a random point distributed according to $\mu$. It should be stressed here that, strictly speaking, $\zero$ thus does not belong to the state-space of the process. Also, the reader will not find it difficult to check that  the semigroup is generated by $\Delta$ with domain composed of $f\in C^2(S)$ such that both $f$ and $f''$ belongs to $\cossf$. 

\section{The case of infinite $ \mum$: stochastic analysis}\label{sec:stocha}

Let, as in the previous section $\mum(S)=\infty $ and $\int (1\wedge x)\mum (\mud x)$ be finite.  As a direct consequence of Theorem \ref{thm:full}, there exists an $S$-valued Feller process, say, $W_{\alpha,\beta, \gamma, \mum}$ generated by $\afup$, the restriction of $\Delta$ to the domain where \eqref{loo:0} holds, provided that $\alpha,\gamma \ge 0$ and $\beta \ge 0$.  It is our goal in this section to find an explicit formula for $W_{\alpha,\beta, \gamma, \mum}$, seen as having values in $\scan$, that would be analogous to that presented in Section \ref{sec:finite} for $W_{\alpha,\beta, \gamma, \delta \mu}$. 

We are working under the standing assumption saying that, for all $i\in \kad$,  
\begin{equation}\label{zez:4} \text{ the $\mum$ measure of the $i$th ray is infinite whenever }  \beta_i =0. \end{equation}
The case where this assumption is violated will be covered in Section \ref{sec:dom2}.  

\subsection{A representation of $W_{\alpha,\beta, \gamma, \mum}$}\label{sec:scprim}


To find a counterpart of \eqref{wab:5} we need, besides a $\ka-$dimensional Brownian motion $B$, an independent Poisson point process $\Pi$ with values in $[0,\infty) \times \scan$ and mean measure $\ud t \times \mum $.  To recall, see \cite{kingmanpoisson}, this means that $\Pi$ is a collection of points that are randomly scattered over $[0,\infty)\times \scan$ in such a way that their number in any measurable $A\subset [0,\infty)\times \scan $ is Poisson distributed with mean $(\mud t\times \mum) (A)$, and their numbers in disjoint subsets are independent (and independent of $B$). It will also be convenient to think of $\Pi$ as a random measure, being the sum of Dirac measures concentrated at the points of $\Pi$: 
\[ \Pi = N(\mud t,\mud x) \coloneqq \sum_{(t,x)\in \Pi}\updelta_{(t,x)}.\]

With the help of $\Pi$, we can define an $(\R^+)^\ka$-valued process 
\begin{equation}\label{repka}
\mc U(t)\coloneqq t \beta +\int_0^t\int_{\scan} x N(\mud s, \mud x) =t\beta +  \sum_{(s,x)\in \Pi, s\in [0,t]}x, \qquad t \ge 0. 
\end{equation}
As in the proof of uniqueness in Theorem \ref{thm:wba}, one can argue that coordinates of $\mc U$, denoted in what follows by $U_i, i\in \kad$, are independent subordinators. Because of our standing assumption \eqref{zez:4}, each $U_i$ is strictly increasing and so its generalized inverse $U_i^{-1}$ is continuous.

Our main result in this section says that, as long as \eqref{zez:4} is satisfied, $W_{\alpha,\beta,\gamma,\mum}$ can be represented in terms of $B$ and $\mc U$. 

\begin{thm}\label{thm:wbam} There is a family $T_i, i \in \kad$ of nondecreasing nonnegative processes with continuous paths such that \begin{itemize}
\item [1. ]  $\sum_{i=1}^\ka T_i(t)=t, t \ge 0$, 
\item [2. ] $W_{0,\beta,0,\mum}$ defined by requiring that its $i$th coordinate equals 
\begin{equation}(W_{0,\beta,0,\mum})_i = (B_{i}  + U_i \circ U_i^{-1} \circ \mc L_i ) \circ T_i,  \qquad  i\in \kad,\label{rep1} \end{equation}
(where $\mc L_i$ is defined in \eqref{wab:4}) is a Feller process in $\scan$ generated by $\gen_{0,\beta,0,\mum}$,
 and
\item [3. ] $U_i^{-1}\mc \circ \mc L_i \circ T_i$ does not depend on $i\in \kad$.  
\end{itemize}
Moreover, conditions 1. and 3. determine $T_i$s uniquely (a.s.), and the common value $L$ of $U_i^{-1}\circ \mc L_i \circ T_i, i \in \kad$ is a  local time of $ W_{0,\beta,0,\mum}$ at $\zero$.  Finally, $W_{\alpha,\beta,\gamma,\mum}$ can be represented as 
\begin{equation}\label{mainrep} W_{\alpha,\beta,\gamma,\mum} (t) = \begin{cases}
               W_{0,\beta,0,\mum} (  A^{-1}_{ \alpha}(t)), & L \circ  A^{-1}_{\alpha}(t)< \zeta ,\\
               undefined, & L \circ A^{-1}_{\alpha}(t)\geq \zeta , 
           \end{cases}\end{equation} 
           where $A_\alpha^{-1}$ is the inverse to $A_\alpha$ defined by $A_\alpha (t) = t +\alpha L(t), t \ge 0$ and $\zeta$ is an exponentially distributed random variable with parameter $\gamma$ that is independent of $B$ and $\Pi$.  
\end{thm}

Several comments are here in order. First of all, in this theorem, as before, by a local time of $ W_{0,\beta,0,\mum}$ at $\zero$ we mean a continuous additive functional that grows only when  $ W_{0,\beta,0,\mum}$ is at $\zero$. Also, once we establish that $L$ is a local time $ W_{0,\beta,0,\mum}$ at $\zero$, as a direct consequence we will get that $L \circ  A^{-1}_{\alpha}$  is a local time of $ W_{\alpha,\beta,0,\mum}$ at $\zero$. Similarly, $K$ defined by 
\begin{equation}\label{tildek} K (t)= \begin{cases} L \circ  A^{-1}_{\alpha}(t), &  L \circ  A^{-1}_{\alpha}(t)< \zeta, \\
 \zeta, & L \circ  A^{-1}_{\alpha}(t)\ge  \zeta.\end{cases} \end{equation}
will turn out to be a local time for $\wfup$.

As a further comment, we recall that, by Theorem \ref{thm:full}, the process $W_{0,\beta,0,\mum}$ 
 is obtained as a limit of processes studied in Section \ref{sec:finite} which have larger and larger intensities of `resurrections' (equal to $\gamma +\delta(\eps)$) whereas distributions of their position after these `resurrections' (equal to $\mu_\eps$) are concentrated more and more around the graph's origin. This leads to the intuition that jumps of $W_{0,\beta, 0, \mum}$  occur `with infinite intensity' but are `infinitesimally small'. Representation \eqref{rep1} confirms this. Indeed, for each $i\in \kad$, both $B_{i} \circ T_i$ and $ U_i^{-1} \circ \mc L_i  \circ T_i$ have continuous paths, and so the jumps of $W_{0,\beta, 0, \mum}$  come only from those of $\mc U$. These, in turn are the jumps of $\int_0^t\int_{\scan} x N(\mud s, \mud x)$. Now, \eqref{ocniczek} (further down) shows that $\int_0^t\int_{\scan} x N(\mud s, \mud x)$ is a sum of an absolutely convergent series, which by nature consists of infinitely many extremely small terms. 

Finally, Theorem \ref{thm:wbam} generalizes Theorem \ref{thm:wba} in a similar way as the latter generalizes Theorem \ref{thm:baryaktar}.

\subsection{Approximation in $D([0,\infty), \R)$; four lemmas}\label{sec:appro}

The main idea of the proof of Theorem \ref{thm:wbam} is to approximate $W_{0,\beta,0,\mum}$ by processes $W_{0,\beta (\eps),0,\delta(\eps)\mu_\eps}$ (where $\delta (\eps)$ and $\mu_\eps$ are those of the beginning of Section~\ref{sec:tcoim}) as constructed in Theorem \ref{thm:wba}. We know from Theorem \ref{thm:full} that the semigroups of processes $W_{0,\beta (\eps),0,\delta (\eps) \mu_\eps}$ converge to the semigroup of $W_{0,\beta,0,\mum}$, and so, by the Trotter--Sova--Kurtz--Macke\-vi\v cius Theorem (see e.g. \cite{kallenbergnew} p. 385, see also Theorem 2.5 in Chapter 4 in \cite{ethier}), the
processes $W_{0,\beta (\eps),0,\delta (\eps) \mu_\eps}$    converge in distribution to $W_{0,\beta,0,\mum}$ in $D([0,\infty),\R^\ka)$, provided that they start from the same point.  Since this in turn implies that each coordinate of $W_{0,\beta (\eps),0,\delta (\eps) \mu_\eps}$ converges in distribution to the corresponding coordinate of $W_{0,\beta,0,\mum}$ in the space $D([0,\infty),\R)$, to prove representation \eqref{rep1} it suffices to show that, for each $i\in \kad$, the $i$th coordinate of $W_{0,\beta (\eps),0,\delta (\eps) \mu_\eps}$ converges in $D([0,\infty),\R)$ to the right-hand side of \eqref{rep1} almost surely. This will be done in Proposition \ref{thm:convergence1} (see further down) which, however, needs to be preceded by
four lemmas. The proof of Theorem \ref{thm:wbam} will be completed in Section \ref{sec:dowodzik}.

To fix notation, we write 
\begin{equation}\label{appro:1} (W_{0,\beta (\eps),0,\delta (\eps) \mu_\eps})_i = (B_i + U_{\eps,i} \circ U_{\eps,i}^{-1} \circ \mc L_i ) \circ T_{\eps,i}, \qquad i \in \kad \end{equation}
where  $U_{\eps,i}$s are subordinators defined as counterparts of \eqref{wab:3}, and 
$T_{\eps,i}$s are nondecreasing nonnegative processes characterized in Theorem \ref{thm:wba}. Our first lemma reveals that all $U_{\eps,i}$s can be constructed by means of the Poisson process $\Pi$ introduced above and that, as $\eps \to 0$, the so-constructed processes converge to the coordinates of $\mc U$ of \eqref{repka}.  

\begin{lemma}\label{lem:convU}
The role of $U_{\eps,i}$ can be played by 
\begin{equation}\label{repusia}\seq{U_{\eps,i}(t)} \coloneqq t\beta (\eps) + \int_0^t \int_{\Gamma_\eps} x N(\mud t,\mud x), \qquad t \ge 0, \end{equation}
where, as before, $\Gamma_\eps$ is the set of points of $\scan$ that lie at a distance $\ge \eps$ from the graph's center. Moreover, for any $t>0$ and $i\in \kad$, almost surely,
\[ \grae \sup_{s\in [0,t]} |U_{\eps,i}(s)- U_i(s)| =0 \mqquad {and} \grae \sup_{s\in [0,t]} |U_{\eps,i}^{-1}(s)- U_i^{-1}(s)| =0.\]
\end{lemma}
\begin{proof} Definition \eqref{repusia} implies that the number of jumps of $\seq{U_{\eps,i}(t)}$ in an interval of length $\varDelta t$ is Poisson distributed with parameter $\delta (\eps)\varDelta t \mu_\eps (\R) = \delta (\eps)\varDelta t$, and that the numbers of these jumps in disjoint intervals are independent. This shows (see e.g. \cite{bill}*{Ch. 23}) that the moments of jumps are the moments of jumps of a Poisson process with parameter $\delta (\eps)$. Since, moreover, the distribution of the size of each jump is $\mu_
\eps$, this shows the first part of the lemma (see \eqref{wab:3}).

To establish the second, we start by noting that, for any measurable $A\subset \scan$, $\int_0^t \int_{\scan} 1_A (x) N(\mud s,\mud x)$ is 
a Poisson random variable with parameter $t\mum(A)$, and so $\mathsf E \int_0^t \int_{\scan} 1_A (x) N(\mud s,\mud x) = t \mum (A), t \ge 0$. A standard procedure shows thus that, for any measurable $f\colon \scan \to \R^+$, $\mathsf E \int_0^t \int_{\scan} f (x) N(\mud s,\mud x)= t\int_{\scan} f(x)\mum(\mud x), t \ge 0$. In particular, 
\begin{align*} \mathsf E \int_0^t \int_{\scan} |x| \1_{\{|x|<1\}}  N(\mud s,\mud x)&= t\int_{|x|<1} |x| \mum(\mud x)\\&\le t \int_{\scan} (|x|\wedge 1) \mum (\mud x)< \infty,\end{align*}
where $|x|$ is the length of $x$. 
Together with the fact that $\mum (\Gamma_1) <\infty$ this shows that, for $t >0$, the sum 
\begin{equation}\label{ocniczek} \int_0^t \int_{\scan} x N(\mud s,\mud x) = \sum_{(s,x)\in \Pi, s\in [0,t]}x \end{equation}
has at most countably many terms, and thus can be thought of as a series; moreover, the series converges absolutely. But, for each $i\in \kad$ and $s\in [0,t]$, 
\begin{align*}
0\le  U_i(s) - U_{\eps,i} (s) &= |\beta_i - \beta_i(\eps)| + \int_0^s \int_{\{|x|<\eps\}} x_i N(\mud u,\mud x) \\
& \le |\beta_i - \beta_i(\eps)| + \int_0^t \int_{\{|x|<\eps\}} x_i N(\mud u,\mud x), \end{align*}
where $x_i$ is the $i$th coordinate of $x$. 
The last expression converges to $0$ as $\eps \to 0$, the integral being the partial sum of the absolutely convergent series of \eqref{ocniczek} that contains only the terms with norm smaller than $\eps$. This shows that $U_{\eps,i}\rightrightarrows U_i$ on $[0,t], i \in \kad$. The rest is now a direct consequence of  \cite[Theorem 13.6.3]{Whitt}, because each $U_i$ is strictly increasing. 
\end{proof}

Our next lemma establishes convergence of $T_{\eps,i}$s. 

\begin{lemma}\label{lemacikladny} As $\eps \to 0$, $T_{\eps,i}$s converge, almost surely, uniformly with respect to $t$ in finite intervals, to nondecreasing, continuous and nonnegative $T_i$s uniquely characterized by the following conditions: 
\begin{itemize}
\item [1. ] $\sum_{i\in \kad} T_i(t) =t, t \ge 0$,
\item [2. ] $U_i^{-1}\mc \circ \mc L_i \circ T_i$ does not depend on $i\in \kad$. 
\end{itemize}
 \end{lemma}
\begin{proof} By Lemma \ref{lem:convU}, $U_{\eps,i}^{-1}$s converge uniformly in compact subintervals to $U_i^{-1}$s on a set of probability $1$. In what follows we consider only events from this set. 

Let $t>0$ be fixed. Since $\sum_{i\in \kad} T_{\eps,i} (s)=s, s \in [0,t], \eps \in (0,1]$ and $T_{\eps,i}$s are nonnegative, $\max_{s\in [0,t]}T_{\eps,i}(s)\le t, i \in \kad$. Moreover, for $s'>s$, we have 
\( 0\le T_{\eps,i} (s') - T_{\eps,i}(s) = s'-s - \sum_{j\not =i} [T_{\eps,j}(s') - T_{\eps,j}(s)] \le s'-s,i \in \kad\),  because $T_{\eps,i}$s are nondecreasing. This shows that $T_{\eps,i}$s are Lipschitz continuous with Lipschitz constant $1$. Hence, the Arzela--Ascoli theorem applies, showing that $T_{\eps,i}$s, as restricted to $C[0,t]$, form a pre-compact set.

Hence, for any sequence $\jcg{\eps_n}$ of positive numbers not exceeding $1$ and converging to $0$, there is a subsequence $\left(\eps_{n_k}\right )_{k\ge 1}$ such that, for each $i$, $T_{\eps_{n_k},i}$s converge, as $k\to \infty$, uniformly in $[0,t]$,  to a $T_i$. Clearly, $T_i$s are nonnegative, nondecreasing and continuous, and satisfy $\sum_{i\in \kad} T_i(s) =s, s \in [0,t]$. Moreover, since $U_{\eps,i}^{-1} \circ \mc L_i \circ T_{\eps,i}$ does not depend on $i$, neither does $U_{i}^{-1} \circ \mc L_i \circ T_{i}$. Therefore, one can argue as in the proof of uniqueness in Theorem \ref{thm:wba} (see Lemma \ref{lem:balance_uniq} in particular) that $T_i$s are determined uniquely. Since $\jcg{\eps_n}$ was chosen arbitrarily, this completes the proof.  \end{proof}

Before presenting the third lemma we recall that our aim is to show that the right-hand sides of \eqref{appro:1} converge, as $\eps \to 0$, to the right-hand side of \eqref{rep1} in $D([0,\infty),\R)$, $i\in \kad$. The difficulty lies in the fact that neither addition nor composition is a continuous operation in $D([0,\infty),\R)$, and both formulae involve such manipulations. Fortunately, Lemma 
\ref{lemacikladny} shows that, as $\eps \to 0$, $B_i \circ T_{\eps,i}$ converge uniformly with respect to $t$ in finite intervals, to a continuous process (equal to $B_i\circ T_i$), and in such a case  Theorem 4.1 in \cite{Whitt:1980} can be used  to establish that our task comes down to proving that
  \begin{equation}\label{suffices} \grae U_{\eps,i} \circ U_{\eps,i}^{-1} \circ \mc L_i  \circ T_{\eps,i} = U_{i} \circ U_{i}^{-1} \circ \mc L_i \circ T_{i} \qquad \text{in } 
D([0,\infty),\R). \end{equation} 
(The theorem reveals that sums of elements of the Skorokhod space converge to the sum of limits if the latter exist and at least one of them is continuous.) Moreover, to deal with compositions, we recall the following lemma; its part (a) comes from \cite[Lemma 3.3]{IksanovPilipenkoPrykhodko2021}, and its part (b) --- from \cite[Proposition 2.3]{Straka+Henry:2011}.
\begin{lemma} \label{lemma:both} Let $f_\eps,g_\eps, \eps \in [0,1]$ be members of $D([0,\infty), \mbR)$ 
such that we have $\grae f_\eps = f_0$ and $\grae g_\eps =g_0$ in this Skorokhod space. Assume also that each $g_\eps$ is nonnegative and nondecreasing. 
\begin{itemize}
\item [(a) ] Assume that $g_0$ is continuous and that if, for some $t\ge 0$, $g_0(t)$ is a point of discontinuity of $f_0$ then $g_0(t')\not = g_0(t)$ for all other $t'\ge 0$. Then $\grae f_\eps \circ g_\eps = f_0 \circ g_0$ in  $D([0,\infty), \mbR)$.
\item [(b) ] Assume that $g_\eps$ are unbounded and that $g_0$ is strictly increasing. Then $\grae g_\eps \circ g_\eps^{-1} = g_0 \circ g_0^{-1}$ in  $D([0,\infty), \mbR)$.

\end{itemize}

\end{lemma}

Lemma \ref{lemma:both} will be used in conjunction with the following result. 
\newcommand{\epsi}{\epsilon}
\begin{lemma}\label{det}
Let $\ell_i, i \in \kad $ be nondecreasing functions $\ell_i\colon[0,\infty)\to [0,\infty)$ that do not have common levels of constancy. Furthermore, let $t_i, i\in \kad$ be nondecreasing functions $t_i\colon[0,\infty)\to [0,\infty)$ such that $\sum_{i\in \kad} t_i(u)=u, u\ge 0$ and $\ell_i \circ t_i$ does not depend on $i\in \kad$. 

\begin{itemize} 
\item [(a) ] Suppose that for some $i\in \kad$, $a\ge 0$ and $\epsi >0$ we have $\ell_i (a+\epsi)=\ell_i(a)$, and, for some $s\ge 0$, $t_i(s)=a$. Then, for $u\in [s,s+\epsi]$, 
\[ t_i(u) = a +u -s \mquad{and} t_j(u) = t_j(s), \quad j\not =i,\]
that is, in this small interval, $t_i$ grows linearly with slope $1$ whereas the other $t_j$s remain constant. 
\item [(b) ] Suppose that, for an $s>0$, we have $\ell_i(v) < \ell_i (t_i(s))$ for all $i\in \kad$ and $v< t_i(s)$. Then $t_i(u) < t_i(s)$ for all $i\in \kad$ and $u<s$. 

\end{itemize} 
\end{lemma}

\begin{proof}\ \ 

 (a) Since $t_j$s are nondecreasing and $\sum_{j\in \kad} t_j (u) =u, u \ge 0$, we have $t_i(s) \le t_i(u)\le t_i(s) +u-s=a+u-s, u\ge s$. Thus, $\ell_i\circ t_i$ stays equal $\ell_i(a)$ for a short while: we have
\begin{equation}\label{stalosc} \ell_i \circ t_i (u) = \ell_i(a), \qquad u\in [s,s+\epsi], \end{equation}
because the inequality just established shows that, for $u$ as above,
 $t_i(u)$ lies in $[a,a+\epsi]$ and $\ell_i$ is nondecreasing. 
 
 Next, recalling that $\ell_j\circ t_j$ does not depend on $j\in \kad$, we see that $\ell_j \circ t_j(s) = \ell_i \circ t_i (s) = \ell_i (a), j \in \kad$.  It follows that 
 \[ \ell_j (v) > \ell_j(t_j(s)), \qquad v > t_j(s), j\not =i;\]
 indeed, $\ell_j(t_j(s)) = \ell_i(a)$ cannot be a level of constancy of $\ell_j, j\not =i$, being a level of constancy of $\ell_i$. This in turn proves the second part of our thesis, for the inequality $t_j(u) > t_j(s)$ for some $j\not =i$ and $u\in [s,s+\epsi]$ would imply $\ell_j\circ t_j (u) > \ell_j\circ t_j (s)=\ell_i(a)$  and so, in view of \eqref{stalosc}, contradict the fact that $\ell_j\circ t_j$ coincides with $\ell_i\circ t_i$.  Finally, to see that the first part is a consequence of the second, we write 
\( u ={\textstyle \sum_{j\in \kad}} t_j(u) = t_i(u) + {\textstyle \sum_{j\not = i}} t_j(s)= t_i(u) +s -t_i(s) = t_i(u) +s -a , u\in [s,s+\epsi]. \)  

(b) Striving for a contradiction, assume that there is a $u<s$ and an $i\in \kad$ such that $t_i(u) = t_i(s)$. Since $\sum_{j\in \kad} t_j$ is a strictly increasing function, there is a $j\in \kad$ such that $t_j(u) < t_j(s)$ and so $\ell_j\circ t_j(u) < \ell_j \circ t_j(s)$ whereas, clearly, $\ell_i \circ t_i(u) = \ell_i \circ t_i(s)$.  Now we have the required contradiction since by assumption $\ell_i\circ t_i$ and $\ell_j\circ t_j $ do not differ. \end{proof}

\subsection{Approximation in $D([0,\infty), \R)$; the main proposition}\label{sec:approprim}
Here is the approximation theorem heralded before.

\begin{proposition} \label{thm:convergence1}For $i \in \kad$, the right-hand sides of \eqref{appro:1} converge almost surely, as $\eps \to 0$, to the right-hand side of \eqref{rep1} in $D([0,\infty),\R)$.
\end{proposition}
\begin{proof} As we have seen in Section \ref{sec:appro}, it suffices to show \eqref{suffices}. Let $i\in \kad$ be fixed. 

\bf Step 1. Convergence of $U_{\eps,i} \circ U_{\eps,i}^{-1} \circ \mc L_i.$ \rm 

In Lemma \ref{lem:convU} we have established that $\grae U_{\eps,i}(t) = U_i(t)$  and   $\grae U_{\eps,i}^{-1}(t) = U_i^{-1}(t)$ for all $t\ge 0$ almost surely, and the limit is uniform with respect to $t$ in all compact subsets of $[0,\infty)$. It follows that $f_\eps \coloneqq U_{\eps,i} \circ U^{-1}_{\eps,i}$ converge to $f_0 \coloneqq U_{i} \circ U^{-1}_{i}$ in the sense of $D([0,\infty),\R)$, by Lemma \ref{lemma:both} (b), because $U_i$ is strictly increasing. At the same time, $g_\eps \coloneqq \mc L_i,\eps \in [0,1]$ converge  trivially, as $\eps \to 0$, to $g_0$. Since, moreover, $g_0$ is continuous, to apply  Lemma \ref{lemma:both} (a) we need to check that, almost surely, for any point $u$ where $U_i\circ U_i^{-1}$ jumps, there is at most one time $t$ such that $\mc L_i(t)=u$. But, if there are two distinct times with this property, $u$ is a level of constancy of $\mc L_i$, and thus a point of jump of $\mc L^{-1}_i$. Hence, it suffices to show that, with probability $1$, the sets of times of jumps of $U_i\circ U_i^{-1}$ and $\mc L_i^{-1}$ are disjoint. However, $\mc L_i^{-1}$ is known to be a subordinator (see \cite{bertoin}*{Thm. 8 p. 114}) and thus the probability that it jumps at a given time is zero. Moreover, $U_i\circ U_i^{-1} $ is \cadlag \, and thus has at most countable number of jumps. Since $\mc L^{-1}_i$ is independent of $U_i\circ U_i^{-1}$, the probability that they have a common point of jump is zero also. This completes the proof. 

\newcommand{\andrija}{
At first we take $f_\eps\coloneqq U_{\eps,i}\circ U^{-1}_{\eps,i} ,$ $g_\eps\coloneqq   L_i$ in
 Then  $U_{\eps,i}\circ U^{-1}_{\eps,i}\to U_i\circ U^{-1}_i, \eps\to0$ a.s., see . Note and $s$ is a    level of constancy  of $L_i$ if $L_i^{-1}$ has a jump at $s.$ Recall that  $L_i^{-1}$ is an inverse subordinator and hence $\Pb(L_i^{-1} \text{has a jump at a fixed point } s)=0.$ Any c\`adl\`ag process has   at most countable set of jumps.   Since the process $U_i\circ U^{-1}_i $ and  $L_i^{-1}$ are  independent, 
     \[
     \Pb(L_i^{-1} \text{and } U_i\circ U^{-1}_i \text{have common instants of jump} )=0.
     \] 
     Further we will assume without noticing that we consider only $\omega$s from the set of probability 1 where $U_i$s, $L_i^{-1}$s doesn't have common points of jumps, we have convergence stated in Lemma \ref{lem:convU}, etc.}

\bf Step 2. Convergence of $U_{\eps,i} \circ U_{\eps,i}^{-1} \circ \mc L_i \circ T_{\eps,i}.$ \rm 

Our strategy is to use Lemma \ref{lemma:both} (a) again, this time with $f_\eps \coloneqq U_{\eps,i}\circ U^{-1}_{\eps,i}\circ \mc L_i$ and  $g_\eps \coloneqq T_{\eps,i}, \eps \in [0,1]$. Thus, by the first step and Lemma \ref{lemacikladny}, the task comes down to showing that, almost surely, if $u$ is a point of jump of $U_{ i}\circ U^{-1}_{ i}\circ \mc L_i$ and, for a $t\ge  0$, $T_i(t)=u$, then $T_i(t')\not =T_i(t)$ for all nonnegative $t'\not =t$.

We start our argument by recalling (see the proof of Theorem \ref{thm:wba}) that processes $U_j^{-1}\circ \mc L_j, j \in \kad$ almost surely  do not have common points of constancy.  Hence, by disregarding a set of probability $0$, we can thus assume that the paths of $\ell_j \coloneqq U_j^{-1} \circ \mc L_j$ and $t_j \coloneqq T_j, j \in \kad $ satisfy the main assumption of Lemma \ref{det}.

\newcommand{\ble}{

This is clear intuitively: An $s$ described above is a point of jump of $U_{ i}\circ U^{-1}_{ i}\circ \mc L_i \circ T_i$ and thus a point of jump of the $i$th coordinate of $W_{0,\beta,0,\mum}$. Since such a jump occurs when $\mc L_i$ reaches a specified level for the first time, $\mc L_i$ must increase for at least a short time before $u$ and so $T_i$ must increase for at least a short time before $s$. Furthermore, at $s$, the $i$th coordinate of  $W_{0,\beta,0,\mu}$ jumps to a point of positive distance from origin and thus $T_i$, the time spend at the $i$th  edge, must increase for some time after $s$, because the paths of the process on each edge are continuous. 

Now take $f_\eps \coloneqq U_{\eps,i}\circ U^{-1}_{\eps,i}\circ L_i$,  $g_\eps \coloneqq \widetilde T_{\eps,i}$.
Let $s_0$ be a point of jump of $U_{ i}\circ U^{-1}_{ i}\circ L_i.$ Then   $ U^{-1}_{ i}\circ L_i( s_0) $ is a point of jump of $U_i$, and, consequently, there is $\delta>0$ such that $ U^{-1}_{ i}(L_i(s_0)+v)= U^{-1}_{ i}(L_i(s_0 +v))=U^{-1}_{ i}(L_i(s_0)), \ v\in [0,\delta]$. 

Recall that  $ U^{-1}_{j}\circ L_j, 1\leq j\leq d,$ do not have same level of constancy a.s.
The proof of \eqref{eq:increase of T}, and therefore, the proof of the Theorem follows from next statement, where $l_i:=U_i^{-1}\circ L_i.$}
 
 For each such path, to show that the set $\{t\ge 0\colon T_i(t)=u\}$, which we know is nonempty, is a singleton, we let $t$ be its minimum; this $t$ belongs to the set because $T_i$ is a continuous function. Moreover, since $u$ is a level of constancy of $\ell_i$, Lemma \ref{det} (a) tells us that $T_i(t') > T_i(t)$ for $t'>t$. Moreover, $\ell_i(s)< \ell_i(t_i(t))$ for  $s< t_i(t)$ because, by continuity, $t_i(t) = \min \{s\colon \ell_i(s)=u\}$. Finally, we have $\ell_j(s)< \ell_j(t_j(t))$ for $s<t_j(t)$ since $u=\ell_j (t_j(t))$ is not a level of constancy of $\ell_j, j\not =i$. Lemma \ref{det} says therefore that $T_i(t')< T_i(t)$ for $t'<t$, completing the proof.  \end{proof}

 \subsection{Proof of Theorem \ref{thm:wbam}}\label{sec:dowodzik}   

\subsubsection{The case of $\gamma =0$.}\label{8.4.1} \ \ 

\rm Theorem \ref{thm:full} says that the semigroups related to $W_{\alpha,\beta, 0, \delta (\eps) \mu_\eps} $ converge to that of $W_{\alpha,\beta, 0, \mum} $. Thus, by the Trotter--Sova--Kurtz--Macke\-vi\v cius Theorem,  processes $W_{\alpha,\beta, 0, \delta (\eps) \mu_\eps} $  converge in distribution to $W_{\alpha,\beta, 0, \mum} $ in $D([0,\infty), \R^\ka)$ and, as a consequence, in $[D([0,\infty), \R)]^\ka $ as well.  

Next, let 
\begin{equation}\label{czasy} A_{\eps,\alpha} (t) \coloneqq  t + \alpha  L_\eps(t) \mqquad{and} L_\eps (t) \coloneqq U_{\eps,i}^{-1} \circ \mc L_i \circ T_{\eps,i} (t), \quad t \ge 0,\eps >0\end{equation}  (recall that  $U_{\eps,i}^{-1} \circ \mc L_i \circ T_{\eps,i} $ does not depend on $i\in \kad$).  By combined Lemmas \ref{lem:convU} and \ref{lemacikladny}, \(\grae A_{\eps,\alpha} = A_\alpha\)  ($A_\alpha$ is defined in Theorem  \ref{thm:wbam}) almost surely and uniformly in compact subsets of $[0,\infty)$. Hence, as a corollary to Proposition \ref{thm:convergence1}, we see that, for each $i\in \kad$, $(W_{0,\beta,0,\delta (\eps) \mu_\eps})_i \circ A_{\eps,\alpha}^{-1}$ (see \eqref{appro:1})  converges, as $\eps \to 0$, to $(B_{i}  + U_i \circ U_i^{-1} \circ \mc L_i ) \circ T_i \circ A_\alpha^{-1} $ in $D([0,\infty),\R)$; 
indeed, since $A_\alpha$ is strictly increasing and continuous, so is its inverse, and thus Lemma \ref{lemma:both} (a) is in force. In other words, 
$W_{\alpha,\beta, 0, \delta (\eps) \mu_\eps} $, which, by Theorem \ref{olepkimizabitym}, coincides with  
$W_{0,\beta, 0, \delta (\eps) \mu_\eps} \circ A_{\eps,\alpha}$, converges almost surely in  $[D([0,\infty),\R)]^\ka$ to the process whose $i$th coordinate is $(B_{i}  + U_i \circ U_i^{-1} \circ \mc L_i ) \circ T_i \circ A_\alpha^{-1} $. Since almost sure convergence implies convergence in distribution, the process just described is identified as $W_{\alpha,\beta, 0, \mum} $. This completes the proof in  the case of $\gamma =0$. 

\subsubsection{The case of $\gamma >0$.}\label{sec:gamma}\ \  %

\rm 

The difficulty with this case is technical in nature: the Trotter--Sova--Kurtz--Macke\-vi\v cius Theorem is usually stated for processes that are honest, whereas $W_{\alpha,\beta, \gamma,\mum}$ with $\gamma >0$ is clearly undefined from a random time on. A natural way to circumvent this difficulty is to extend $\scan$ by adjoining to it an additional point (sometimes called the cemetery or coffin state,  see \cites{kniga,bass,kallenbergnew}), and redefine  $W_{\alpha,\beta, \gamma,\mum}$ to make it honest by requiring that the new process is at the coffin state whenever the old is undefined. Below are the details. 

\newcommand{\scand}{\scan^\dagger}

\bf (A) Analysis of resolvents. \rm 
We choose a  point $\dagger $ in $\R^\ka \setminus \scan$, let $\scand$ be the union of $\scan $ and $\{\dagger\}$, and let $W_{\alpha,\beta,\gamma,\delta(\eps) \mu_\eps}^\dagger$ be the process with values in $\scand$, given by 
\begin{equation}\label{givenby}
W_{\alpha,\beta,\gamma,\delta(\eps) \mu_\eps}^\dagger (t) \coloneqq \begin{cases}
               W_{\alpha,\beta,0,\delta(\eps)\mu_\eps} (t), & L_\eps \circ  A^{-1}_{\eps,\alpha}(t)< \zeta ,\\
               \dagger, & L_\varepsilon \circ A^{-1}_{\eps, \alpha}(t)\geq \zeta ; 
           \end{cases}
           \end{equation}  as before, $\zeta$ is an exponentially distributed random variable   with parameter $\gamma$, which is independent of $B$ and $\mc U$, and $L_\eps$ and $A_{\eps,\alpha}$ are defined in \eqref{czasy}.  It is clear, see Theorem \ref{olepkimizabitym}, that $W_{\alpha,\beta,\gamma,\delta(\eps) \mu_\eps}^\dagger$ equals $\dagger$ whenever $W_{\alpha,\beta,\gamma,\delta(\eps) \mu_\eps}$ is undefined, but otherwise these processes coincide (needless to say, the former process started at $\dagger$ stays at this point forever).
 
 It follows that the related semigroup, say, $T_\eps^\dagger$, in  $\mathfrak C(\scand)$ is given by $T_\eps^\dagger (t)f (\dagger)=f(\dagger)$ and 
 \( T_\eps^\dagger (t)f (x) = \mathsf E_x f(W_{\alpha,\beta,\gamma,\delta(\eps) \mu_\eps}) + \mathsf P_x [L_\eps \circ A_{\eps,\alpha}^{-1} (t) \ge \zeta ] f(\dagger),  x\in \scan\) for $ t\ge 0$ and $f\in \mathfrak C(\scand);$ this semigroup is clearly honest. There is also a handy counterpart of this equality in terms of resolvents. To wit, denoting by $\rlae^\dagger$ the Laplace transform of $T_\eps^\dagger$, and viewing $\mathfrak C(\scan)$ as the subspace of $\mathfrak C(\scand)$ of functions that vanish at $\dagger$, we have 
\[ \rlae^\dagger f = \rez{\gen_\eps} f + \lam^{-1}f(\dagger) \ell_{\lam,\eps}, \qquad f \in \mathfrak C(\scan)\subset \mathfrak C(\scand)\]
and $\lam \rlae^\dagger \mathsf 1_{\{\dagger\}}=  \mathsf 1_{\{\dagger\}}, \lam >0$. Here, 
\begin{itemize}
\item  as in Section \ref{sec:tcoim}, $\gen_\eps$ is the generator of $ W_{\alpha,\beta (\eps),\gamma,\delta(\eps) \mu_\eps}$, 
\item $\mathsf 1_{\{\dagger\}}\in \mathfrak C(\scand)$ equals $1$ at $\dagger$ and is $0$ otherwise, 
\item  $\rez{\gen_\eps}$ is seen as an operator mapping the subspace $\mathfrak C(\scan)\subset \mathfrak C(\scand) $ to itself,
\item  $\ell_{\lam,\eps}\in \mathfrak C(\scan)$ is the Laplace transform of the lifetime of  $ W_{\alpha,\beta,\gamma,\delta(\eps) \mu_\eps}$, that is,  $\ell_{\lam,\eps}=\mathsf 1_{\scan} - \lam \rez{\gen_\eps} \mathsf 1_{\scan}$, where $\mathsf 1_{\scan}\in \mathfrak (\scand)$ equals $1$ for $x\in \scan$ and is $0$ otherwise.  
\end{itemize}
This formula, in turn, shows that the strong limit $\rla^\dagger \coloneqq \grae \rlae^\dagger $ exists and is given by 
\[ \rla^\dagger f = \rez{\afup} f + \lam^{-1}f(\dagger) \ell_{\lam}, \qquad f \in \mathfrak C(\scan)\subset \mathfrak C(\scand)\]
and $\lam \rlae^\dagger \mathsf 1_{\{\dagger\}}=  \mathsf 1_{\{\dagger\}}, \lam >0$, where $\ell_\lam \coloneqq \mathsf 1_{\scan} - \lam \rez{\afup} \mathsf 1_{\scan}$ is the Laplace transform of the lifetime of $ W_{\alpha,\beta,\gamma,\mum}$. Indeed, in Lemma \ref{lem1} and Theorem \ref{thm:full} we have established that $\grae \rlae = \rez{\afup}, \lam >0$. Furthermore, it is easy to see that $\rla^\dagger, \lam >0$ is a Feller resolvent: for example, the condition $\lil \lam \rla^\dagger f = f, f \in \mathfrak C(\scand)$ can be established by noting that $\mathfrak C(\scand)$ is a direct sum of $\mathfrak C(\scan)$ and the subspace spanned by $\mathsf 1_{\{\dagger\}}$, where convergence is obvious. It follows thus that  there is a Feller process $W_{\alpha,\beta,\gamma,\mum}^\dagger$ with values in $\scand$, that is a limit in the sense of distributions, of $W_{\alpha,\beta,\gamma,\delta(\eps)\mu_\eps}^\dagger$.  

\newcommand{\mat}{\mathpzc t}
\bf (B) Analysis of paths. \rm 
In order to describe the limit process in more detail, let us first take a closer look at the time, say, $\mat$, at which $W_{\alpha,\beta,\gamma,\mum}^\dagger$ reaches $\dagger$. For $x\in \scan$ this is the lifetime of $W_{\alpha,\beta,\gamma,\mum}$, and for $x=\dagger$ this is $0$. Since the Laplace transform of the lifetime of $W_{\alpha,\beta,\gamma,\mum}$ started at an $x\in 
\scan$ is $\lam \mapsto \ell_\lam(x)$, the Laplace transform of the time at which $W_{\alpha,\beta,\gamma,\mum}^\dagger$ started at an $x\in \scand$ reaches $\dagger$ is $\lam \mapsto \ell_\lam^\dagger(x)$, where $\ell_\lam^\dagger \in \mathfrak C(\scand)$ is obtained from $\ell_\lam\in  \mathfrak C(\scan)$ by assigning it value $1$ at $\dagger$.  Analogously, the Laplace transform of the time, say, $\mat_\eps$, at which $W_{\alpha,\beta,\gamma,\delta(\eps)\mu_\eps}$ started at an $x\in \scand$ reaches $\dagger$ is $\lam \mapsto \ell_{\lam,\eps}^\dagger (x)$, where $\ell_{\lam,\eps}^\dagger$ coincides with $\ell_{\lam,\eps}$ on $\scan$, and equals $1$ at $\dagger$. Since, as we have seen in point (A), $\grae \ell_{\lam,\eps} =\ell_\lam$ in $\mathfrak C(\scan)$, we have also 
$\grae \ell_{\lam,\eps}^\dagger =\ell_\lam^\dagger$ in $\mathfrak C(\scand)$, and this means that $\grae \mat_\eps=\mat$ weakly. 

Formula \eqref{givenby} can be used to provide yet more specific information on $\mat$, and, in particular, to prove that $\mat_\eps $ converge to $\mat$ almost surely.  First of all, the formula reveals that \begin{equation}\label{maletepsy}\mat_\eps = \inf \{t\ge 0\colon L_\eps \circ  A^{-1}_{\eps,\alpha}(t)=\zeta\}.\end{equation} To proceed, we recall the following real analysis result. Suppose that,  for some $t_0$, we are given nondecreasing, continuous 
functions $f_\eps\colon [0,t_0] \to \R^+$ starting at $0$ that converge uniformly to a function $f$. Suppose also that, for an $a>0$ 
there is a $\mat \in (0,t_0)$ such that $f(\mat)=a$, but $a$ is not a level of constancy of $f$. Then for $\mat _\eps \coloneqq \inf \{t\ge 0\colon f_\eps (t)=a\}, \eps >0$ the limit of  $\mat_\eps$ exists and equals $\mat$. Indeed, for any $s \in (\mat, t_0]$, we have $f(s) > a$, 
because $a>0$ is not a level of constancy of $f$. Therefore, for $\eps $ small enough, $f_\eps (s) >a$, and this, by the Darboux property, implies that $\mat_\eps <s$. Since a similar argument shows that, for $s \in [0,\mat)$, $\mat_\eps > s$ as long as $\eps $ is small enough, the result is established. (To see that the claim is no longer true if the assumption that $a$ is not a level of constancy is dropped, consider $t_0=2,$ $a=1$, $f(t)= \min (t,1), t \in [0,2]$ and piecewise linear $f_\eps $ with the graph connecting points $(0,0), (1,1-\eps)$ and $(2,1)$.)

Armed with this information we come back to $\mat_\eps$ of \eqref{maletepsy}. By combined Lemmas \ref{lem:convU} and \ref{lemacikladny}, 
\begin{equation}\label{supl}
\grae L_\eps \circ  A^{-1}_{\eps,\alpha} (t) = L \circ A_\alpha (t)\end{equation} uniformly with respect to $t$ in compact intervals. Moreover, a level of constancy of $ L \circ A_\alpha$ is a level of constancy of $L$ and thus a point of jump of $\mc U$. Since $\mc U$ is independent of $\zeta$ and this random variable has continuous distribution, with probability one $\zeta$ is not a level of constancy of $L \circ A_\alpha $. It follows that the real analysis result proved above applies, and we infer that $t_\eps $ converge, almost surely, to the time when $ L \circ A_\alpha$ reaches the level of $\zeta$, that is, to 
\begin{equation}\label{malet} 
\mat = \min\{t\ge 0\colon  L \circ A_\alpha (t)=\zeta \}.\end{equation}

Finally, we will show that the paths of $W_{\alpha,\beta,\gamma,\mum}^\dagger$ are obtained from those of  $W_{\alpha,\beta,0,\mum}$ by requiring that they are equal to $\dagger$ from $\mat$ on. In this argument we will be aided by the following result, which can be thought of as an exercise in using Skorokhod's metric. Assume that $f_\eps, \eps > 0$ are members of $D([0,\infty),\scan)$, and that each coordinate of $f_\eps $ converges, as $\eps\to 0$, to the corresponding coordinate of $f$ in the Skorokhod topology. Suppose, moreover, that $\mat_\eps$ are positive numbers converging to a $\mat>0$ and that $\mat $ is a point of continuity of $f$. Then, for  $\widehat f_\eps, \eps > 0$ defined as \[\widehat f_\eps (t)
\coloneqq  \begin{cases}
        f_\varepsilon(t),& t<\mat_\eps,\\
        \dagger,& t\geq \mat_\eps, 
    \end{cases}\]
each coordinate of $\widehat f_\eps $ converges, as $\eps \to 0$, to the corresponding coordinate of $\widehat f$ in the Skorokhod metric. This result has a bearing on the issue at hand, because (a) we know from the case of $\gamma =0$ that for almost all paths, each coordinate of $W_{\alpha,\beta, 0, \delta (\eps) \mu_\eps} $, converges in the Skorokhod topology to the corresponding coordinate of  
$W_{\alpha,\beta, 0, \mum} $, (b) we have established that, almost surely, $\grae \mat_\eps =\mat$, and (c)~points of discontinuity of $W_{\alpha,\beta, 0, \mum} $ are precisely points of jumps of $\mc U$, and since $\mc U$ is independent of $\zeta$, a continuous random variable, the probability that $\mat $ is a point of discontinuity of $W_{\alpha,\beta, 0, \mum} $ is zero. Thus, paths of $W_{\alpha,\beta, \gamma, \delta (\eps) \mu_\eps}^\dagger$ converge, almost surely, to those given by  $W_{\alpha,\beta,0,\mum}(t)$ for  $L \circ A_\alpha^{-1}(t) < \zeta$ and $\dagger $ otherwise. 

On the other hand, we know from the analysis of resolvents and the Trotter--Sova--Kurtz--Macke\-vi\v cius Theorem that the processes $W_{\alpha,\beta, \gamma, \delta (\eps) \mu_\eps}^\dagger$ converge in distribution to $W_{\alpha,\beta, \gamma, \mum}^\dagger$. Hence, 
\[ W_{\alpha,\beta, \gamma, \mum}^\dagger (t)= \begin{cases} W_{\alpha,\beta,0,\mum}(t), & L \circ A_\alpha^{-1}(t) < \zeta, \\ \dagger, & L \circ A_\alpha^{-1}(t) \ge \zeta.\end{cases} \]  
Since this establishes \eqref{mainrep}, our proof is completed, save for the fact that $L$ is a local time of $W_{0,\beta,0,\mum}$.

To reach our last goal, let $i\in \kad$ be such that $\mum_i$, the restriction of $\mum$ to the $i$th ray, is infinite, and let us take a look at the process \[ X \coloneqq   B_{i}  + U_i \circ U_i^{-1} \circ \mc L_i  .\]
Similarly as the process in \eqref{dzis:3}, $X$ is a Brownian motion on the right half-line with jumps from $0$, but, since $\mum_i$ is infinite, these jumps tend to be `infinitely small' and occur `infinitely often'. Such processes are described in detail in \cite{blum}. In particular, it follows from  \cite[p. 68]{blum}  that the local time of $X$ is $U_i^{-1} \circ \mc L_i$, and that, almost surely with respect to any initial distribution, 
\[
  U_i^{-1} \circ \mc L_i(t)=\grae \frac{\sum_{s\leq t}\1_{\{X(s-)=0,X(s)>\eps\}}}{\mum_i([\eps,\infty))}, \qquad t\geq 0. 
\]
Hence 
\begin{align*}
L(t)=  U_i^{-1} \circ \mc L_i\circ T_i(t)&= \grae \frac{\sum_{s\leq T_i(t)}\1_{\{X(s-)=0, X(s)>\eps\}}}{\mum_i([\eps,\infty))}\\
&=\grae \frac{\sum_{s\leq t}\1_{\{X(T_i(s-))=0,\  X(T_i(s))>\eps\}}}{\mum_i([\eps,\infty))}
\qquad \text{a.s.}  
\end{align*}
 because function $T_i$ is continuous and monotonous. Since $X\circ T_i$ is the $i$th coordinate of $W_{0,\beta,0,\mum}$, this formula shows that $L$ is a continuous additive functional of $W_{0,\beta,0,\mum}$. Moreover, $L$ does not grow except when $W_{0,\beta,0,\mum}$ is at $\zero$. It follows that $L$ is a local time for this processes and thus we are done.

\subsection{$W_{\alpha,\beta, \gamma, \mum}$ without assumption \eqref{zez:4}}\label{sec:dom2}

In this section we sketch a picture of $W_{\alpha,\beta, \gamma, \mum}$ in the case where \eqref{zez:4} is violated.  To begin with, since $\mum$ is assumed to be infinite, we have a nonempty set of indices $\mathpzc J$ such that 
\begin{equation}\label{was:1} \text{ the $\mum$ measure of the $i$th ray is infinite whenever }  \beta_i =0, \qquad i\in \mathpzc J. \end{equation}
Let $S_{\mathpzc J}$ and  $S_{\kad \setminus \mathpzc J}$ be the subgraphs of $S$ formed by edges with indices in $\mathpzc J$ and $\kad \setminus \mathpzc J$, respectively. Also, let $\mum_{\mathpzc J}$  and $\mum_{\kad\setminus \mathpzc J}$ be the restrictions of $\mum $ to $S_{\mathpzc J}$ and  $S_{\kad\setminus \mathpzc J}$, respectively.

This scenario naturally splits into two sub-cases: the first, a bit simpler one, where $\mum$ is concentrated on  $\mathpzc J,$ and the one where it is not. In the first of these, $W_{\alpha,\beta, \gamma, \mum}$ is in a sense identical with a process of the type described in Theorem \ref{thm:wbam} having values in the subgraph $S_\mathpzc J$ of $S$. To wit,
\begin{equation}\label{was:2} W_{\alpha,\beta, \gamma, \mum} \overset {\text{ident.}}= W_{\alpha,\beta^\natural, \gamma, \mum_\mathpzc J}, \end{equation}
where $\beta^\natural$ is obtained from $\beta$ by throwing away all its zero coordinates with indexes in $\mathpzc {K\setminus J}$. 
To elaborate: When started in $S_{\mathpzc J}$, $W_{\alpha,\beta, \gamma, \mum}$ is indistinguishable from  $W_{\alpha,\beta^\natural, \gamma, \mum_\mathpzc J}$; in particular, despite frequently going through $\zero$, it never enters the proper subgraph $S_\mathpzc {K\setminus J}$. Also, when started in one of the edges of $S_{\mathpzc{K \setminus J}}$, the process behaves like a one-dimensional Brownian motion until the time when it reaches $\zero$; from that time on, it behaves like  $W_{\alpha,\beta^\natural, \gamma, \mum_\mathpzc J}$ and `does not see' the subgraph $S_{\mathpzc{K \setminus J}}$ anymore.

When  $\mum\not =\mum_{\mathpzc J}$, the situation changes: the fact that $\mum_{\mathpzc{K\setminus J}}>0$ allows the process to return to the previously omitted subgraph $S_{\mathpzc {K\setminus J}}$ (or at least its part; there may be edges of $S_{\mathpzc{K\setminus J}}$ with $\mum$ measure $0$).  More precisely,  $W_{\alpha,\beta, \gamma, \mum}$
turns out to be a concatenation of an infinite number of copies of processes on $S$ that can be identified with  $W_{\alpha,\beta^\natural, \gamma + |\mum_{\kad\setminus \mathpzc J}|, \mum_\mathpzc J}$ as in \eqref{was:2}.  This means that 
\begin{itemize} 
\item [(a) ] when started in $S_\mathpzc {K\setminus J}$,  $W_{\alpha,\beta, \gamma, \mum}$ initially behaves like a one-dimensional Brownian motion, \item [(b) ] from the moment it reaches $\zero $ (or if started at $S_\mathpzc J$), it  behaves like $W_{\alpha,\beta^\natural, \gamma + |\mum_{\kad\setminus \mathpzc J}|, \mum_\mathpzc J}$, 
\item [(c) ] when the local time of $W_{\alpha,\beta^\natural,0, \mum_\mathpzc J}$ (the process $L\circ A_\alpha^{-1}$ in Theorem \ref{thm:wbam}) reaches the level of an independent exponential random variable with parameter  $ \gamma + |\mum_{\kad\setminus \mathpzc J}|$,   $W_{\alpha,\beta, \gamma, \mum}$  is killed with probability $\frac \gamma { \gamma + |\mum_{\kad\setminus \mathpzc J}|}$ or resurrected with probability $\frac {|\mum_{\kad\setminus \mathpzc J}|} { \gamma + |\mum_{\kad\setminus \mathpzc J}|}$; in the latter case, its new starting point lies in $S_{\mathpzc {K\setminus J}}$ and  is distributed according to  $\frac{\mum_{\kad\setminus \mathpzc J}}{|\mum_{\kad\setminus \mathpzc J}|}$; of course, the process forgets its past and starts all over again. 
\end{itemize}

\section{$\lam$-potential of the local time \eqref{tildek}}\label{sec:ltot}

This section is devoted to a further information on the local time of $\wfup$. This information will allow us to conclude in particular that $\overline \beta K$ is the symmetric local time. 

 Notably, local times of a Markov process at a given point differ by multiplicative constants, but are fully determined by their $\lambda$-potentials, see \cite[\S III.3]{blum}. We will calculate the $\lam$-potential $u_\lam$ of the local time $K$ (the latter being spelled out in  \eqref{tildek}); to recall, the $\lam$-potential of $K$ is defined, for $x\in S$  and $\lam >0$, by 
  \[
  u_\lambda(x)\coloneqq \E_x\int_0^\infty \e^{-\lambda s} \ud K(s).  
  \]
Our result is summarized in the following theorem.   
  \begin{thm}\label{thm:resLoc}
  Suppose that assumption \eqref{zez:4} is satisfied.  Then
  \begin{align}\label{eq:res_formula_localtime}
    u_\lambda 
    =\frac{1}{\lambda \alpha +\sqrt{2\lambda} \, \overline \beta +\gamma+\lambda \int_SR_\lambda^01_S\ud \mum} {\mathpzc l}_\lambda^0, \qquad  \lam >0.
\end{align}
\end{thm}
As a corollary we also obtain:  

\begin{thm}\label{thm:Loc}
  Assume that $\overline \beta >0$.  Then
the process $\overline \beta K$ is the symmetric local time of  $\wfup, $ that is,
\[
\overline \beta K(t)=\lim_{\eps\to0}\tfrac{1}{2\eps}\int_0^t \1_{\{0<|W_{\alpha,\beta, \gamma, \mum} (s)|<\eps\}}\ud s
\]
in probability with respect to all underlying measures $\mathsf P_x$. 
\end{thm}

\subsection{Background material}

For the proof of these two theorems  we need a background material, including some facts from the general theory of $\lam$-potentials (Lemmas \ref{lem:conv-potent} and \ref{lem:dzik1}, and Theorem \ref{thm:lim_of_potentials}) and a useful estimate (Lemma \ref{onezymus}). We start with the first of the lemmas.

\begin{lemma}\label{lem:conv-potent}
 For $\eps\in[0,1]$,  let  
 $K_\eps (t),t \ge 0  $ be  a  family of random variables defined on a probability space that is common for all these families. Suppose, 
 furthermore, that the following conditions are satisfied. 
\begin{itemize}  
\item [(a)] For all $\omega$, the paths $[0,\infty)\ni s\mapsto K_\eps (s,\omega)$ are continuous and nondecreasing with $K_\eps(0,\omega)=0, \eps\in[0,1]$.  
\item [(b)] There is a function $\chi\colon [0,\infty) \to \R^+$ such that $\sup_{\eps \in (0,1]} \E  [K_\eps (s)]^2 \leq \chi (s),$ $ s\geq 0,$ whereas $\int_0^\infty \e^{-\lam s} \sqrt{\chi (s)} \ud s < \infty$ for all $\lam >0$. 
    \item [(c)] For all $s\ge0$,  $K_\eps(s)$ converges, as $\eps \to 0$,  to $ K_0(s)$ in probability. \end{itemize}
Then
    \[
    \grae  \E \int_0^\infty \e^{-\lambda s} \ud K_\eps (s) =  \E \int_0^\infty \e^{-\lambda s} \ud K_0(s). 
    \]   
     \end{lemma}
    \begin{proof}
        The monotone convergence theorem, integration by parts formula and the Tonelli's theorem, yield, for $\lam >0$ and $\eps \in (0,1]$,
\begin{align}\nonumber \E  \int_0^\infty \e^{-\lam s} \ud K_\eps (s) &=\lim_{u\to \infty} \E  \int_0^u \e^{-\lam s} \ud K_\eps (s) = \lim_{u\to \infty} \E \int_0^u \lam \e^{-\lam s} K_\eps (s) \ud s\\& =  \E \int_0^\infty \lam \e^{-\lam s} K_\eps (s) \ud s =\int_0^\infty \lam \e^{-\lam s} \E K_\eps (s) \ud s, \quad\label{welka:21} \end{align}
because the estimate $\e^{-\lam u} \mathsf E K_\eps (u) \le \e^{-\lam u} \sqrt{\mathsf E [K_\eps (u)]^2} \le \e^{-\lam u}\sqrt {\chi (u)}$ shows that  $\e^{-\lam u} \E  K_\eps (u)$ converges to $0$ as $u\to \infty$. 

Next, for each $u\ge 0$, since $K_\eps (u)$ converges to $K_0(u)$ in probability, there is a sequence $\jcg{\eps_n}$ with $\gra \eps_n =0$ such that $K_{\eps_n} (u)$ converges to $K_0(u)$ almost surely. The Fatou lemma implies thus that $\e^{-\lam u} \mathsf E K_0 (u) \le \e^{-\lam u} \sup_{\eps \in (0,1]}\mathsf E K_\eps  (u)\le \e^{-\lam u} \sup_{\eps \in (0,1]} \sqrt{\mathsf E [K_\eps (u)]^2} \le \e^{-\lam u}\sqrt {\chi (u)}$, and this allows extending the calculation \eqref{welka:21} to $\eps =0$ also.

Finally, for any $u\ge 0$, by (b), the second moments of 
random variables $K_\eps (u), \eps\in(0,1]$ are uniformly bounded, and so the variables are
uniformly integrable (see e.g. \cite{gut}*{Theorem 4.2, p. 215}). Hence, their  expectations converge: $\lim_{\eps\to0}\E  K_\eps (u)=\E  K_0 (u)$,  see e.g. \cite{gut}*{Theorem 5.4, p. 221}. In the light of this information, our thesis is a consequence of the Lebesgue dominated convergence theorem combined with \eqref{welka:21} extended to include $\eps =0$ --- here the estimate of point (b) is used again. 
    \end{proof}

Turning to the rudiments of the theory of $\lambda$-potentials, let $K$ be a continuous additive functional of a Markov process taking values in a complete separable metric space $S$. The function $\varphi=\varphi_t$ given by  \[ \varphi_t(x)\coloneqq \E_x K(t), \qquad  t\ge  0, x \in S\] is said to be its \emph{characteristic}, see \cite{dynkinmp}. The following proposition gathers basic estimates of the growth of the characteristic.

\begin{lemma}\label{lem:dzik1}
Assume that  there is a $t>0$  such that $\kappa_0\coloneqq \sup_{x\in S} \varphi_t(x)<\infty$. Then
\begin{itemize} \item [(a)] $\sup_{x\in S} \varphi_s(x)\leq \kappa (1+s), s \ge 0,$ where $\kappa \coloneqq \max (\kappa_0,\kappa_0/t),$ and
\item [(b)] $\sup_{x\in S} \E_x [K(s)]^n \leq n! (\sup_{x\in S} \varphi_s(x))^n$ for $n \ge 1, s \ge 0$ (in what follows we will use this inequality only with $n=2$).\end{itemize}
\end{lemma}

\begin{proof} The functional equation that characterizes continuous additive functionals (see e.g. \cite{rosen}*{p. 83}) implies that 
\[ \mathsf E_x K (t+h) \le \mathsf E_x K(t) + \sup_{y\in S} \mathsf E_y K(h), \qquad x\in S, t,h \ge 0,  \]  
that is, 
\[ \varphi_{t+h} (x) \le \varphi_t (x) + \sup_{y\in S} \varphi_h (y), \qquad x\in S, t,h \ge 0. \]  

Let now $t>0$ be the real number described in the assumption. Given $s\ge 0$ one can find the natural $n$ such that $s\in [nt,(n+1)t)$. Then, by the relation established above and the fact that $[0,\infty)\ni t \mapsto K(t)$ is nondecreasing with $K(0)=0$, 
\begin{align*}
\varphi_s(x) &\le \varphi_{(n+1)t} (x) =\sum_{i=0}^n [\varphi_{(i+1)t}(x) -\varphi_{it} (x)] \le (n+1) \kappa_0 \\& \le \kappa_0 (1 + \tfrac st) \le \kappa (1+s).
\end{align*}
Since $s$ is arbitrary, this shows (a). 
Point (b) follows from \cite{gands}*{Ch II, \S 6, Lemma 3}. \end{proof}

{Lemma \ref{lem:dzik1} shows that in the case of $\lam$-potentials there is a natural and rather simple condition  guaranteeing that assumption (b) of Lemma \ref{lem:conv-potent} is satisfied. In what follows we will combine these lemmas to deduce convergence of $\lambda$-potentials of local times of interest from the convergence of local times themselves (see the proofs of  Theorems \ref{thm:resLoc} and \ref{thm:Loc} given further down).
For the proof of Theorem \ref{thm:Loc} we will also need the following theorem showing that, vice versa, sometimes convergence of local times can be inferred from the convergence of their $\lambda$-potentials.}

\begin{thm}
\label{thm:lim_of_potentials}
    Let  $u_{\eps,\lambda}, \eps \in (0,1]$ be $\lambda$-potentials of continuous additive functionals $K_\eps$  of Feller processes $X_\eps$ with values in a locally compact metric space $S$: 
    \[ u_{\eps,\lam } (x) = \E_x \int_0^\infty \e^{-\lam s} \ud K_\eps (s), \qquad x \in S.\] Assume that 
\begin{itemize}
    \item [(a)]  $\sup_{\eps \in (0,1]}\sup_{x\in S} u_{\eps,\lambda}(x)<\infty$, $\lambda>0$, and
\item [(b)] $u_{\eps,\lambda}$ converge, as $\eps \to 0$,  uniformly with respect to $x\in S$  to a function $u_\lambda, \lam >0$. 
\end{itemize}
Then $u_\lambda,\lam >0 $ is a $\lambda$-potential of a continuous additive functional $K$, and 
    for all $t\ge0$, $K_\eps (t)$ converges in probability, as $\eps \to 0$,  to $K(t)$  with respect to all measures $\mathsf P_x, x\in S$.
\end{thm}
Unfortunately, a detailed proof of this theorem is quite long and requires discussing a number of notions we would rather avoid discussing in this paper. Neither could we find such a proof in the existing literature. Hence, we need to restrict ourselves to stating that the theorem can be established by following the arguments presented in the proof of Theorem 3.8 on page 162  in \cite{bandg}; 
we are grateful to Prof. Patrick J. Fitzsimmons for this observation. We note also that a similar result, stating that uniform convergence of characteristics entails convergence of continuous additive functionals, is contained in  Theorem 6.4 Ch. VI, \S 3 of  \cite{dynkinmp}.

\subsection{Proof of Theorem \ref{thm:resLoc}}

To complete the preparations for the proof of Theorem  \ref{thm:resLoc}, as mentioned earlier, we need estimates that will allow us to use Lemma \ref{lem:conv-potent}
and Theorem \ref{thm:lim_of_potentials} in the particular case of additive continuous functionals of interest to us --- these estimates are gathered in Lemma \ref{onezymus}, further down. To set the stage for the lemma, let $M_t, t\ge 0$ be the minimal Brownian motion on the star graph $S$ with $\ka $ edges. The relation 
\begin{align}\rla^0g(i,x)&\coloneqq  \E_{(i,x)} \czn \e^{-\lam t} g(M_t) \ud t \nonumber \\ & = \tfrac 1{\slam} \czn (\e^{-\slam |x-y|} - \e^{-\slam (x+y)}) g(i,y)\ud y\nonumber\\&= 2C_i(g) \sinh \slam x 
- \sqrt {\tfrac 2{\lambda}}\int_0^x \sinh \slam (x-y) g(i,y)\ud y, \label{atka:1} \end{align}
an incarnation of \eqref{loo:5}, has been established for $g\in C(S)$, but a standard procedure extends it to all integrable $g$ (for $x\ge 0$ and $i \in \kad$). Hence, it makes sense to speak of $\rla^0 g $ for integrable functions $g$ also; for such $g$,  $\rla^0g $ is the continuous function defined by the right-hand side of \eqref{atka:1}.  In the specific instance of $g_\eps, \eps \in (0,1]$ defined by 
\begin{equation}\label{beatka:4} g_\eps(x)\coloneqq \tfrac 1{2\eps} \1_{\{0<|x|<\eps\}}, \qquad x \in S, \end{equation}
we can similarly extend \eqref{joj:1}. Indeed,  each $g_\eps$ is a pointwise limit of a sequence of continuous functions vanishing at $\zero$ and 
having a common integrable bound. Hence,  \eqref{joj:1} together with the dominated convergence theorem yield 
\begin{equation}\label{joj:1prim} \E_x \czn \e^{-\lam t} g_\eps(\wfup (t)) \ud t = \rla^0 g_\eps (x) + \overline C (g_\eps) \elam^0(x),\end{equation} for $\, x\in S, \lam >0$ and  $\eps \in (0,1]$, 
where 
\[ \overline C(g_\eps) = \frac{2\slam \sui \beta_iC_i(g_\eps)+\int_S \rla^0 g_\eps \ud \mum }{\lam \alpha + \slam \sui \beta_i +\gamma + \lam \int_S \rla^0 \mathsf \mathsf 1_S \ud \mum  }, \]
whereas $C_i(g_\eps) = \frac 1{2\eps \slam}\int_0^\eps \e^{-\slam y}\ud y$, $i \in \kad, \eps \in (0,1], \lam >0$. 
\begin{lemma}\label{onezymus} For each $\lam >0$, there is a constant $\kappa_\lam$ such that 
\begin{equation}
    \label{eq:bound_resol}
|R^0_\lambda g(i,x)|\leq \kappa_\lambda (1\wedge x) \|g\|_{L^1(S)}, \qquad x \ge 0, i \in \kad,
\end{equation}
for all integrable $g$ on $S$, where, of course, $\|g\|_{L^1(S)}\coloneqq\sui \int_0^\infty |g(i,y)|\ud y$.  
Moreover, for each $t>0$, 
\[ \max_{i\in \kad}\,  \sup_{x\ge 0}\sup_{\eps\in(0,1]} \E_{(i,x)} \int_0^t g_\eps (\wfup (s))\ud s <\infty .\]

\end{lemma} 
\begin{proof}The expression  in the second line of  \eqref{atka:1} makes it clear that $|\rla^0 g(i,x)|$  does not exceed $ \tfrac 1{\slam} \|g\|_{L^1(S)}$ for all $x\ge 0, i\in \kad$ and $\lam >0.$ However, for $x\in (0,1]$ we can estimate differently: using the other expression, we see that 
\begin{align*} |\rla^0 g(i,x)| & \le \left (2C_i(g) + \sqrt{\tfrac 2\lambda} \int_0^x |g(i,y)| \ud y\right )\sinh \slam x \\& \le 2 \sqrt{\tfrac 2\lambda} \|g\|_{L^1(S)}\sinh \slam x \le  2x \sqrt{\tfrac 2\lambda}  \sup_{y\in (0,1]} \tfrac{\sinh \slam y}y \|g\|_{L^1(S)},  \end{align*}
and the supremum is finite because so is $\lim_{y\to 0+}\tfrac{\sinh \slam y}y$.   These two inequalities render \eqref{eq:bound_resol} with $\kappa_\lam $ being the larger of $\frac 1{\slam}$ and   $2 \sqrt{\tfrac 2\lambda}  \sup_{y\in (0,1]} \tfrac{\sinh \slam y}y$. 
(The case of $x=0$ is immediate since $\rla^0 g(i,0) =0$, $i \in \kad, \lam >0$.)

To establish the rest, we write 
\begin{align}
 \E_{(i,x)} \int_0^t g_\eps (\wfup (s))\ud s &\le  \e^{\lam t} \E_{(i,x)} \int_0^t \e^{-\lam s} g_\eps (\wfup (s))\ud s\nonumber\\
 &\le \e^{\lam t} (\rla^0 g_\eps (x) + \overline C (g_\eps) \elam^0(x)) \nonumber \\
 &\le \e^{\lam t} (\tfrac {\ka  \kappa_\lam}2 + \overline C(g_\eps)),  \label{beatka:2}\end{align}
where first \eqref{joj:1prim} and then \eqref{eq:bound_resol} was used (together with $\|g_\eps\|_{L^1(S)} = \frac \ka 2$). This reduces our task to showing that $\sup_{\eps \in (0,1]} \overline C(g_\eps)$ is finite, and for this it suffices to show that the supremum over $\eps \in (0,1]$ of the numerators of $\overline C(g_\eps)$ is finite. However, since $0\le C_i(g_\eps)\le \frac 1{2\slam}$, the latter supremum does not exceed $\overline \beta + \sup_{\eps \in (0,1]} \int_S \rla^0g_\eps \ud \mum$. This is finite by \eqref{eq:bound_resol}.  
 \end{proof}

We are finally ready to present the proof of Theorem   \ref{thm:resLoc}. Here is its plan. 
\begin{itemize} 
\item [(i) ] First we consider the case where $\mum$ is finite, $\beta_i>0, i\in\kad$ and $\sui\beta_i=1$; parameters $\alpha\ge 0$ and $\gamma \ge 0$ are chosen arbitrarily.  Here, we use the fact that $K$ can be approximated by the  continuous additive functionals $K_\eps, \eps \in (0,1] $ defined by
\begin{equation}\label{jutro:1} K_\eps (t) \coloneqq \tfrac 1{2\eps}  \int_0^t \1_{\{0< |W_{\alpha,\beta,\gamma,\mum}(s)|\le \eps\}}  \ud s=  \int_0^t g_\eps (W_{\alpha,\beta,\gamma,\mum}(s))  \ud s,  \end{equation}
for $t\ge 0$, employ the previously obtained formula \eqref{joj:1prim} for their $\lam$-potentials, and justify passage to the limit as $\eps \to 0$. 
\item [(ii) ] Next, we drop the assumption $\sui \beta_i =1$ by means of \eqref{marek:1} and \eqref{ele}. 
\item [(iii) ] To treat the general case, we approximate $K$ by local times of the type considered in (i) and (ii). 
\end{itemize}

\begin{proof}[Proof of Theorem \ref{thm:resLoc}]  \ \ \ 

{\bf (i)}  Assume that  $\mum$ is finite   (perhaps even zero), $\beta_i>0, i\in\kad$ and $\sui\beta_i=1$; parameters $\alpha\ge 0$ and $\gamma \ge 0$ are chosen arbitrarily. 
In particular,  $\mum = \delta \mu$ where $\delta \ge 0$ and $\mu$ is a probability measure. Then, by the definition of Section \ref{sec:dow}, $W_{\alpha,\beta,\gamma,\mum}= W_{\alpha,\beta,\gamma,\delta\mu}$ is $\mc C_{\gamma,\delta,\mu}W_{\alpha,\beta}$.  Therefore, as expounded in Section \ref{sec:cim}, the local time of $W_{\alpha,\beta,\gamma,\mum}$ is constructed by means of a sequence of copies of the symmetric local time of $W_{\alpha,\beta}$, that is, of the local time given by \eqref{wika:1}. It follows that for each $s\ge 0$ the value ${K(s)}$  of the local time is the almost sure limit, and  hence also the limit in probability (for  the all underlying probability measures $\mathsf P_x$) of  $K_\eps (s) $ defined in \eqref{jutro:1}. 

This suggests that we could perhaps use Lemma \ref{lem:conv-potent} (with $K_0 \coloneqq K$) to conclude that $\lam$-potentials of $K_\eps$s converge to that of $K$ --- we have just discovered that condition (c) of the lemma holds. To check the other conditions, we recall that,  by the second part of Lemma \ref{onezymus}, the characteristics of the continuous additive functionals $K_\eps$ have a common finite bound: for each $t \ge 0$, 
\[ \sup_{x\in \scan} \sup_{\eps \in (0,1]} \E_x K_\eps (t) < \infty. \]
Lemma \ref{lem:dzik1} then implies that, for a certain (finite) constant $\kappa>0$, 
\[ \sup_{x\in \scan} \sup_{\eps \in (0,1]} \E_x [K_\eps (s)]^2 \le  2\kappa^2 (1+s)^2, \qquad s \ge 0. \]
Hence, also condition (b) of Lemma \ref{lem:conv-potent}  is satisfied. Since condition (a) is obvious, we infer that 
 the $\lam$-potential $u_\lam$ of ${K}$ is given by (see \eqref{joj:1prim})
\[ u_\lam (x) = \grae \mathsf E_x \czn \e^{-\lam t} g_\eps (W_{\alpha,\beta,\gamma,\mum}(s))\ud s = \grae (\rla^0 g_\eps (x) + \overline C (g_\eps) \elam^0(x)),\] $x\in S$. This limit further reduces to  $\grae \overline C (g_\eps) \elam^0(x)$, because, as can be seen from the third line in \eqref{atka:1}, $\grae \rla^0g_\eps (x) =0$ for each $x\in S$. (Indeed, we have $\grae C_i(g_\eps)=\frac{1}{2\sqrt{2\lambda}}$, and so the first term in  $\rla^0g_\eps (x)$ converges to $\frac{1}{\sqrt{2\lambda}}\sinh\slam x$; since the limit of the second term is also equal to $\frac{1}{\sqrt{2\lambda}}\sinh\slam x$, the claim follows.) 

Now, the denominator in $\overline C_\lam (g_\eps)$ does not depend on $\eps$ and  equals $\lam \alpha + \slam  +\gamma + \lam \int_S \rla^0 \mathsf \mathsf 1_S \ud \mum  $ because $\sui \beta_i =1$ by assumption. Also, the numerator involves two summands, of which the first, by $\grae C_i(g_\eps)=\frac{1}{2\sqrt{2\lambda}}$, converges to $\sui \beta_i =1$. We have also proved that $\grae \rla^0g_\eps (x) =0$ for each $x\in S$. This, however, when combined with   \eqref{eq:bound_resol}, yields by the  dominated convergence theorem that $\grae \int \rla^0 g_\eps \ud \mum =0$.  
Hence,  
 \begin{align}\label{jutro:2}
    u_\lambda 
    =\frac{1}{\lambda \alpha +\sqrt{2\lambda} +\gamma+\lambda \int_SR_\lambda^01_S\ud \mum} {\mathpzc l}_\lambda^0, \qquad  \lam >0,
\end{align}
establishing \eqref{eq:res_formula_localtime} in the first case.

{\bf (ii)}  Assume now that the assumption $\sui \beta_i =1$ is violated, but the remaining ones are satisfied. Formulae \eqref{marek:1} and \eqref{ele} tell us that in this case the $\lam$-potential of the symmetric local time of $W_{\alpha,\beta,\gamma,\mum}$  is the $\lambda$-potential of the symmetric local time  of $W_{\alpha/\overline \beta,\beta/\overline \beta,\gamma/\overline \beta,\mum/\overline \beta}$ divided by $\overline \beta$. This shows that \eqref{eq:res_formula_localtime} holds also in this case.

{\bf (iii)} We finally allow some of $\beta_i$s to vanish, but compensate this by assumption \eqref{zez:4}. In order to extend our formula to this case, we approximate $W_{\alpha,\beta, \gamma, \mum} $ by processes  $W_{\alpha,\beta(\eps), \gamma, \delta (\eps) \mu_\eps} $ as in Sections \ref{sec:tcoim} and
\ref{sec:stocha}. More precisely,  we take positive $\beta_i(\eps), {i\in\kad}$ that converge to $\beta_i$ as $\eps \to 0$, $i \in \kad$,   and define finite measures $\mum_\eps\coloneqq \delta (\eps) \mu_\eps$ as the restrictions of $\mum $ to  the sets $\Gamma_\eps$; as before, $\Gamma_\eps \subset S$ is composed of points in $S$ that lie at a distance $\ge \eps$ from the graph's center. Then, for each $\eps\in (0,1]$, 
\begin{itemize}
\item $L_\eps$ defined in \eqref{czasy} is the symmetric local time of $W_{0,\beta(\eps),0, \mum_\eps}$,
\item $L_\eps \circ A_{\eps,\alpha}^{-1}$ (see \eqref{czasy} again) is the symmetric local time of $W_{\alpha,\beta(\eps),0, \mum_\eps}$, 
\item and the symmetric local time of  $W_{\alpha,\beta(\eps),\gamma, \mum_\eps}$, say, $K^\sharp_\eps$, (the superscript $\sharp$ is to distinguish this local time from $K_\eps$ of \eqref{jutro:1}) defined by an analogue of \eqref{mzoe:3}, can in view of definition \eqref{maletepsy} be written as 
\begin{equation}\label{beatka:1} K_\eps^\sharp (t) = L_\eps \circ A_{\eps,\alpha}^{-1} (t \wedge \mathpzc t_\eps), \qquad t \ge 0.\end{equation}
\end{itemize}
Since the coordinates of $\beta(\eps)$ are positive and $\mum_\eps$ is finite and formula \eqref{eq:res_formula_localtime} has been establish for such parameters, the $\lam$-potential of $K_\eps^\sharp$ is given by 
\begin{align}\label{beatka:3} 
\E_x\int_0^\infty \e^{-\lambda s} \ud K_\eps^\sharp (s)    
     =
     \frac{1}{\lambda \alpha +\sqrt{2\lambda} \; \overline {\beta(\eps)}+\gamma+\lambda \int_SR_\lambda^0\mathsf 1_S\ud \mum_\eps} {\mathpzc l}_\lambda(x). 
\end{align}
The right-hand side above converges to $ \frac{1}{\lambda \alpha +\sqrt{2\lambda} \overline \beta +\gamma+\lambda \int_SR_\lambda^0\mathsf 1_S\ud \mum} {\mathpzc l}_\lambda (x)$ as $\eps\to0$, because  $\grae \int_SR_\lambda^01_S\ud \mum_\eps=\grae \int_S\1_{\Gamma_\eps}R_\lambda^0\mathsf 1_S\ud \mum= \int_S R_\lambda^0\mathsf 1_S\ud \mum$ (by the Lebesgue dominated convergence theorem) and $\grae \beta (\eps) =\beta$. Hence, to establish \eqref{eq:res_formula_localtime} it suffices to show that the left-hand side converges to $\E_x\int_0^\infty \e^{-\lambda s} \ud K (s)$, the $\lambda$-potential of $K$ evaluated at $x$. 

To this end we will use Lemma \ref{lem:conv-potent} for $K_\eps\coloneqq K_\eps^\sharp$ and $K_0\coloneqq K$. Let us fix $x\in \scan$. First of all, it is clear that condition (a) of the lemma is satisfied.  Secondly, 
we recall formula \eqref{supl} speaking on convergence of $ L_\eps \circ  A^{-1}_{\eps,\alpha} (s)$ to  $L \circ A_\alpha (s)$, convergence that is uniform with respect to $s$ in finite intervals, and the fact established in Section \ref{sec:gamma} (part (B)) that $\mathpzc t_\eps$ converges to $\mathpzc t$ of \eqref{malet} almost surely with respect to $\mathsf P_x$ (we recall that forumla \eqref{supl} is established under standing assumption \eqref{zez:4}).  We thus infer that
\begin{equation}
    \label{eq:convergence_CAF_K}
    \grae  K_\eps^\sharp (s) = \grae L_\eps \circ  A^{-1}_{\eps,\alpha}(s\wedge \mat_\eps )  = L \circ  A^{-1}_{\alpha}(s\wedge \mat), \qquad s \ge 0
\end{equation}
in the mode that is stronger than convergence in probability (in $\mathsf P_x$); moreover,  in view of \eqref{malet},  we recognize  $L \circ  A^{-1}_{\alpha}(s\wedge \mat)$ as $K(s)$. This shows that condition (c) of the lemma is satisfied also. Finally, we estimate as in \eqref{beatka:2} to see that 
\begin{align*}
\sup_{\eps \in (0,1]} \sup_{x\in \scan} \E_x K_\eps^\sharp &(t) \leq \e^{\lambda t}\sup_{\eps \in (0,1]} \sup_{x\in \scan}\E_x\int_0^\infty \e^{-\lambda s} \ud K_\eps^\sharp (s)\\
&= \sup_{\eps \in (0,1]} \sup_{x\in \scan}\ \frac{\e^{\lam t}}{\lambda \alpha +\sqrt{2\lambda} \; \overline {\beta(\eps)}+\gamma+\lambda \int_SR_\lambda^0\mathsf 1_S\ud \mum_\eps} {\mathpzc l}_\lambda(x)\\ &\leq \sup_{\eps \in (0,1]}  \frac{\e^{\lam t}}{\lambda \alpha +\sqrt{2\lambda} \; \overline {\beta(\eps)}+\gamma+\lambda \int_SR_\lambda^0\mathsf 1_S\ud \mum_\eps}
\end{align*}
Since the last fraction converges, as $\eps \to 0+$, to a finite limit and $\int_SR_\lambda^0\mathsf 1_S\ud \mum_\eps$ increases when $\eps$ decreases, by redefining $\beta (\eps)$ if necessary for $\eps$s that are not so close to $0$, we can make sure that the supremum in the last line is finite. This, however, as in the proof of part (i), allows inferring that condition (b) of the lemma is also satisfied. 

Hence, the left-hand side of \eqref{beatka:3} converges to $\E_x\int_0^\infty \e^{-\lambda s} \ud K (s)$, as desired. This completes the proof.  \end{proof}

\subsection{Proof of Theorem \ref{thm:Loc}} 

 In terms of functions $g_\eps$ defined in \eqref{beatka:4}, our task is to show that 
\[
\overline \beta K(t)=\lim_{\eps\to0}\tfrac{1}{2\eps}\int_0^t g_\eps (W_{\alpha,\beta, \gamma, \mum} (s)) \ud s, \qquad t \ge 0
\]
in probability with respect to all underlying measures $\mathsf P_x$. The idea is to use Theorem \ref{thm:lim_of_potentials} which allows deducing such convergence from the convergence of the $\lam$-potentials, say, $u_{\eps, \lam}$, of $K_\eps$, where $K_\eps (t) \coloneqq \tfrac{1}{2\eps}\int_0^t g_\eps (W_{\alpha,\beta, \gamma, \mum} (s)) \ud s
$, $t \ge 0$.  We have 
\[ u_{\eps,\lam} (x) = \rla^0 g_\eps (x) + \overline C(g_\eps) \elam^0 (x), \qquad \eps \in (0,1], \lam >0,\]
see \eqref{joj:1prim} and the definition of $\overline C(g_\eps) $ following this formula. We need to find a common bound for  $ u_{\eps,\lam} (x) $ and the limit $\grae  u_{\eps,\lam} (x)$.

{\bf (i) }Starting with the latter, we argue as in the last lines of part (i) of the proof of Theorem  \ref{thm:resLoc} that $\grae \overline C(g_\eps)= \frac{\overline \beta}{\lambda \alpha +\sqrt{2\lambda}\overline \beta +\gamma+\lambda \int_SR_\lambda^01_S\ud \mum}$ (note that now $\overline \beta$ need not be $1$). In the same part of the proof just mentioned, but a bit earlier, we have also shown that $\grae \rla^0 g_\eps (x)=0, x\in \scan$; this allows us to conclude, by \eqref{eq:res_formula_localtime},
 that $\grae u_{\eps,\lam} (x) $ exists and coincides with the $\lam$-potential of $\overline \beta K$. However, assumption (a) of Theorem \ref{thm:lim_of_potentials} speaks of convergence that is uniform with respect to $x$, and thus we still need to show that $\grae \rla^0 g_\eps (x)=0$ uniformly in  $x\in \scan$. 
 
 To this end, we take a look at the second line in \eqref{atka:1}; this line reveals that $\slam \rla^0 g_\eps (x)$ can be written as 
 \(\e^{-\slam x} \tfrac 1\eps \int_0^\eps \sinh \slam y \ud y   \)
for $x \ge \eps$, and as 
\( \e^{-\slam x} \tfrac 1\eps \int_0^x \sinh \slam y \ud y + \tfrac 1\eps \sinh \slam x \int_x^\eps \e^{-\slam y}\ud y \) for $x\in [0,\eps]$. The first expression above does not exceed $\tfrac 1\eps \int_0^\eps \sinh \slam y \ud y$ and the second does not exceed $\tfrac 1\eps \int_0^\eps \sinh \slam y \ud y + \max_{0\le x\le \eps}\sinh\slam x$. Since $\grae \sinh \slam \eps =0$ implies $\grae \tfrac 1\eps \int_0^\eps \sinh \slam y \ud y =0$ and $\grae \max_{0\le x\le \eps}\sinh\slam x=0$, this completes the proof of uniform convergence of $\rla^0 g_\eps (x)$ to $0$, as $\eps \to 0$.

{\bf (ii) } To find a bound for $u_{\eps,\lam}(x)$ we estimate as in \eqref{beatka:2}:
\begin{align}
u_{\eps,\lam}(x) & \le \tfrac {\ka  \kappa_\lam}2 + \overline C(g_\eps),  \label{beatka:5}\end{align}
where again \eqref{eq:bound_resol} was used together with $\|g_\eps\|_{L^1(S)} = \frac \ka 2$.  
Since we know from the end of the proof of Lemma \ref{onezymus} that $\sup_{\eps \in (0,1]} \overline C(g_\eps)$ is finite, this shows that assumption (a) of  Theorem \ref{thm:lim_of_potentials} is satisfied. 

The theorem is thus applicable and says that $K_\eps$s converge to a CAF whose $\lambda$-potential coincides with the $\lam$-potential of $K$. But $\lam$-potentials characterize CAFs uniquely; $K_\eps$s must converge therefore to $K$. This completes the proof.

\newcommand{\mka}{\mathcal {k}}
\newcommand{\mel}{\mathcal {l}}
\section{Formulae for $T_i$s}\label{sec:calculation}
Our goal in this section is to express explicitly the time-change processes $T_i, 1\le i\le\ka$ featured in the representation of the process $W_{\alpha,\beta,\gamma, \mum}$ given in Theorem~\ref{thm:wbam}. Our argument hinges on Proposition \ref{prop:rept} (presented below) that belongs to the realm of real analysis, and complements Lemma \ref{lem:balance_uniq}. To recall, the problem at hand is whether, having given  continuous nondecreasing functions $\ell_i\colon [0,\infty)\to [0,\infty)$ with $\ell_i (0)=0$, $i \in \kad$, one can find a collection of functions $t_i\colon [0,\infty)\to [0,\infty)$   such that  \begin{equation}\label{calc:1} \sui t_i(s) =s, s \ge 0 \mqquad{and} \ell_i \circ t_i \text{ does not depend on } i\in \kad.\end{equation}
 Lemma \ref{lem:balance_uniq} establishes that the solution to the problem is unique provided that $\ell_i$s do not have common levels of constancy. Proposition \ref{prop:rept}, in turn, says that a solution always exists, and under assumption of no levels of constancy accompanied by $\gras \ell_i (s) =\infty, i \in \kad$, provides a representation of $t_i$s in terms of functions \[ \mka_i, i \in \kad, \mka   \mqquad{and} \mel \] defined as follows: 
\begin{itemize} \item   $\mka_i\coloneqq \ell_i^{-1}$ is the generalized inverse of $\ell_i$, that is,
\[
 \mka_i(s)\coloneqq \inf\{u\geq 0 \colon \ell_i(u)>s\}, \qquad s\ge 0, i\in \kad,
 \]
\item   $\mka \coloneqq \sui \mka_i$, and 
\item  $\mel \coloneqq \mka^{-1}$ is the generalized inverse of $\mka$. \end{itemize}
 We note that $\mka_i$s are well-defined provided that $\gras \ell_i (s) =\infty, i \in \kad$. Moreover, $\mka_i$s are strictly increasing  since $\ell_i$s are continuous. It follows that $\mka$ is strictly increasing as well, and thus $\mel$ is continuous.

\begin{proposition}\label{prop:rept}
   Assume that $\ell_i\colon [0,\infty) \to [0,\infty),i\in \kad$ are continuous non-decreasing functions such that $\ell_i(0)=0$ and $\gras \ell_i (s) =\infty, i\in \kad$. 
 \begin{itemize} 
 \item [(a) ]  There is a collection of nondecreasing continuous  functions  $t_i\colon[0,\infty)\to [0,\infty), i\in \kad$ such that \eqref{calc:1} is satisfied.  
\item [(b) ]   If, additionally, $\ell_i$s  do not have common levels of constancy, $t_i$s are determined uniquely.
Furthermore, the function $\mel$ introduced above coincides with $\ell_i\circ t_i, i \in \kad$, and $t_i$s  are represented as follows: 
\begin{equation}
    \label{eq:T_i_formula}
    t_i(s)=\mka_i \circ \mel (s)+\1_{(\mdelta \mka_i)\circ \mel (s)>0} (s-\mka \circ \mel (s)), \quad s \ge 0, i\in \kad,
\end{equation} 
where $\mdu \mka_i$ is the jump of $\mka_i,$ that is, $\mdu \mka_i (u) = \mka_i(u) - \mka_i(u-), u \ge 0, i \in \kad$. 
    \end{itemize}
    \end{proposition}

Before presenting the proof, we remark that
\eqref{eq:T_i_formula} is a counterpart of a formula obtained originally in 
        \cite{bayraktar}. However, differences between these two are worth spelling out here.  First of all, on the technical side,  in \cite{bayraktar} the role of $\mka_i$s is played by the 
        \cadlag \, inverses of $\ell_i$s defined as $\ell^{\leftarrow}(t)\coloneqq\inf\{s\geq 0 \colon l_i(s)\geq t\}, t \ge 0$. More importantly, the counterpart of \eqref{eq:T_i_formula}, as obtained in  \cite{bayraktar} applies merely to the case of Brownian motion and has been proved by stochastic means. Formula \eqref{eq:T_i_formula}, on the contrary, is of purely deterministic nature, and applies to paths of all processes that satisfy the assumptions of the proposition.

\begin{proof}[Proof of Proposition \ref{prop:rept}] Throughout the proof, we say that a sequence of functions converges locally uniformly if the convergence is uniform in compact subintervals of $[0,\infty)$.

\bf (a) \rm  Suppose first that $\ell_i$s are strictly increasing. Then the generalized  inverses $\mka_i, i\in \kad$ become the usual inverses, and as such are continuous; also, we obviously have $\mka_i \circ \ell_i (s) = \ell_i\circ \mka_i (s) =s, s\ge 0, i \in \kad$.   It follows that $\mka$ and $\mel$ are continuous and strictly increasing and that  $\mel\circ \mka (s) = \mka \circ \mel (s)=s, s \ge 0$. Hence, introducing \[ t_i \coloneqq \mka_i \circ \mel, \qquad i \in \kad \] (this is a very particular case of \eqref{eq:T_i_formula}), we see that $\ell_i \circ t_i $ does not depend on $i$, being equal to $\mel$. Since,  at the same same time $\sui t_i (s) = \sui \mka_i \circ \mel (s) = \mka \circ \mel (s) = s, s \ge 0$, we see that these $t_i$s satisfy conditions \eqref{calc:1}.

In the general case we approximate each $\ell_i, i\in \kad$ by a sequence of strictly increasing functions, say, $\jcg{\ell_{n,i}}$. For example, we may take $\ell_{n,i}(t)$ $\coloneqq \frac 1n t + \ell_i (t), t \ge 0, n\ge 1$, so that $\gra \ell_{n,i} = \ell_i, i \in \kad$ locally uniformly.  Then, for each $n\ge 1$, there is a collection $t_{n,i},i \in \kad$ of nondecreasing continuous functions satisfying \eqref{calc:1} with $\ell_i$ replaced by $\ell_{n,i}$. Arguing as in Lemma \ref{lemacikladny} we check that, for each integer $m\ge 1$ and $i\in \kad$, elements of the sequence $\jcg{t_{n,i}}$ form a pre-compact set in $C[0,m]$. Hence, we can choose a subsequence of this sequence that converges in $C[0,1]$, and from that one, a subsequence that converges in $C[0,2]$, etc. Therefore, the diagonal argument shows that, for each $i \in \kad$,  there is a subsequence of $\jcg{t_{n,i}}$ 
that converges to a function $t_i$ locally uniformly.  Moreover, it is rather easy to see that this subsequence can be chosen in such a way that  $\jcg{t_{n,i}}$ converges locally uniformly for all $i \in \kad$, that is, that the subsequence is the same for all $i\in \kad$.   Then, since $\ell_{n,i}\circ t_{n,i}$ does not depend on $i$, we can see immediately that neither does $\ell_i \circ t_i$, and similarly that $\sui t_i (s) = s, s \ge 0$. This, however, means that we have found $t_i$s that satisfy \eqref{calc:1}, completing the proof of (a).

\bf (b) \rm By assumption, $\ell_i$s do not have common levels of constancy; thus, Lemma \ref{lem:balance_uniq} guarantees that the collection of functions $t_i, i \in \kad $ satisfying \eqref{calc:1} obtained in (a) is unique, and thus we can restrict ourselves to proving the rest of (b).

Before continuing, let us recapitulate. We know from (a) that  there are strictly increasing functions $\ell_{n,i}, n\ge 1, i\in \kad$ (for each $i\in \kad$, $\jcg{\ell_{n,i}}$ is a subsequence of the sequence denoted $\jcg{\ell_{n,i}}$ in point (a)),
such that $\gra \ell_{n,i}=\ell_i$ locally uniformly, and simultaneously  $t_i =\gra t_{n,i}, i \in \kad$ in the same way, where $t_{n,i}, i \in \kad$ satisfy \eqref{calc:1} with $\ell_i$ replaced by $\ell_{n,i}$. In particular, introducing $\mel_n $ as a common value of $\ell_{n,i} \circ t_{n,i}, i \in \kad$, that is, 
\[ \mel_n \coloneqq \ell_{n,i} \circ t_{n,i}, \qquad i \in \kad,\]
 we see that $\mel_n$s converge locally uniformly to a function, say, $\mel_\infty$. This function is a common value of $\ell_i \circ t_i$,
\[ \mel_\infty \coloneqq \ell_{i} \circ t_{i}, \qquad i \in \kad,\] 
 but we would like to identify it with $\mel $ --- the function introduced above as the generalized inverse of $\mka \coloneqq \sui \mka_i$. This is the first of our two goals (the other one being formula \eqref{eq:T_i_formula}).

\bf Goal no. 1. \rm It is the biggest obstacle on our way that the operation  of taking the generalized inverse is not continuous in the metric of the Skorokhod space used in this paper so far, and neither is the operation  of summation. 
Fortunately, the metric used so far (the most popular one, see e.g. \cite{bill2}, known also as $J_1$) is not the only one at our disposal: in fact, there are at least four natural, related but not identical, topologies in this space, introduced by Skorokhod himself. In particular, Theorem 13.6.3 in \cite{Whitt} says that a sequence of monotone functions in $D((0,\infty), \R)$ converges in the so-called $M_1$ metric iff so does the sequence of their generalized inverses. 
This convergence also is implied by a pointwise convergence of the functions on a dense subset of $(0,\infty)$, and implies such pointwise convergence. As applied to our scenario, the theorem declares thus that, for each $i \in \kad$,  the generalized inverses $\mka_{n,i}$ of restrictions of $\ell_{n,i}$ to $(0,\infty)$ converge to the generalized inverse of $\ell_i$ in the sense of $M_1$ in  $D((0,\infty), \R)$; for, $\ell_{n,i}$ converge locally uniformly to $\ell_i$ and thus also in the sense of $M_1$. For future use we record this information  as 
\begin{equation}\label{fft:1}    \lim_{n\to\infty}\mka_{n,i}\overset{M_1} = \mka_i, \qquad i \in \kad.
\end{equation}
   
To continue, we recall that whereas, as we have just seen, $M_1$ topology has nice properties related to the operation of taking inverse, the so-called $SM_1$ topology (i.e., strong $M_1$ topology) has nice properties related to Cartesian products. Elaborating on this succinct statement, given functions $f\colon \ap{(}0,\infty)\to \R^i$ and $g\colon\ap{(}0,\infty) \to \R^j$ for some integers $i$ and $j$,  we will write $(f,g)$ to denote the function $t \mapsto (f(t),g(t))\in \R^{i+j}$. In this notation, Theorem 12.6.1 in \cite{Whitt} says that if a sequence, say, $\jcg{f_n}$ of elements of $D(\ap{(}0,\infty),\R^i)$ converges to an $f$ in the sense of $SM_1$, and a~sequence $\jcg{g_n}$ of elements of $D(\ap{(}0,\infty),\R^j)$ converges to a  $g$  in the sense of $SM_1$, and, additionally, $f$ and $g$ do not have common points of discontinuity (i.e., there is no argument $t$ such that both $f$ and $g$ are discontinuous at $t$), then $(f_n,g_n)$s converge to $(f,g)$ in the sense of $SM_1$ in $D(\ap{(}0,\infty), \R^{i+j})$. 
(We note on the side that the Whitt book speaks of convergence of members of Skorokhod spaces of functions defined on a finite interval, not on $[0,\infty)$ or $(0,\infty)$. However, the argument in the case of our interest is analogous, see \S 12.9 in \cite{Whitt}.)

Lucky for us, $SM_1$ has further nice properties. First of all, as explained on p. 394 in \cite{Whitt}, even though in general topologies $M_1$ and $SM_1$ differ, in the case of $D(\ap{(}0,\infty),\R)$ they coincide. Secondly, in $SM_1$ the operation of summation is, in a limited sense, continuous. To be more explicit: Corollary 12.7.1 in \cite{Whitt} says that if, for some integer $i$, two sequences, say, $\jcg{f_n}$ and  $\jcg{g_n}$, of elements of $D(\ap{(}0,\infty),\R^i)$ converge to $f$ and $g$, respectively, in the metric of $SM_1$, and, additionally, $f$ and $g$ do not have common points of discontinuity, then the sequence $\jcg{f_n+g_n}$ converges to $f+g$ in the same metric. 

These three results allow us to argue as follows. First of all, we note that convergence in \eqref{fft:1} can be replaced by that in the sense of $SM_1$. Then, using Theorem 12.6.1 in \cite{Whitt} a couple of times, we conclude that the functions $ t \mapsto \seq{\mka_{n,i} (t)} \in \R^\ka$ converge, as $n\to \infty$, to the function  $ t \mapsto (\mka_{i}  (t)) \in \R^\ka$ in the sense of $SM_1$ metric; the theorem is applicable because, by assumption, $\ell_i$s do not have common points of constancy, and so $\mka_i$s do not have common points of discontinuity.  By the same token,  Corollary 12.7.1 of \cite{Whitt} can be used to see that 
$\sui \mka_{n,i} $ converge to $\mka = \sui \mka_i$ in the $SM_1$ metric of $D(\ap{(}0,\infty),\R)$, or, which is the same, in the $M_1$ metric of this space.  But then, by  Theorem 13.6.3 in \cite{Whitt}, the generalized inverses of $\sui \mka_{n,i} $ converge to the generalized inverse of $\mka $, that is, to $\mel$. But $\mel_n$s are defined as the usual inverses of  $\sui \mka_{n,i} $ (see part (a)), and we know that $\mel_\infty = \gra \mel_n$ locally uniformly. Since all the functions involved here are  
c\`adl\`ag, this completes the proof of the fact that (on the entire half-line)
\[ \mel = \mel_\infty = \ell_i \circ t_i, \qquad i \in \kad. \]

\bf Goal no. 2. \rm We are left with showing \eqref{eq:T_i_formula}. All we know so far about $t_i$s is that they are locally uniform limits of $t_{n,i}$s, that the latter satisfy 
\[ t_{n,i} = \mka_{n,i} \circ \mel_n, \qquad i\in \kad, n\ge 1, \]
and that $\mel_n$s converge locally uniformly to $\mel_\infty=\mel$. 

To proceed, we need the following real analysis result concerning generalized inverses. Suppose that nondecreasing functions $f_n$ converge locally uniformly to a function $f$ and that a $u\ge 0$ is not a level of constancy of $f$, that is, that $u$ is a point of continuity of $f^{-1}$; suppose furthermore that nonnegative numbers $u_n $ converge to $u$. Then $\gra f_n^{-1}(u_n) = f^{-1}(u)$. 
Indeed, in such circumstances, since $f^{-1}$ is continuous at $u$, for any $\eps>0$ there exists a $\delta>0$ such that 
\[
f(f^{-1}(u)-\eps)\leq u-\delta < u < u+\delta\leq  f(f^{-1}(u)+\eps).
\]
Therefore
\[
\limsup_{n\to\infty}f_n(f^{-1}(u)-\eps)\leq u-\delta < u < u+\delta\leq  \liminf_{n\to\infty} f_n(f^{-1}(u)+\eps), 
\]
and consequently
\[
f^{-1}(u)-\eps\leq \liminf_{n\to\infty} f_n^{-1}(u-\delta)  \leq \limsup_{n\to\infty} f_n^{-1}(u+\delta)\leq  f^{-1}(u)+\eps. 
\]
Since $\eps>0$ is arbitrary, monotonicity of $f_n^{-1}$ implies  thus that we have
$\gra f_n^{-1}(u_n) = f^{-1}(u)$ as long as $\lim_{n\to\infty }u_n=u.$.

Armed with this knowledge, fix an $i\in \kad$ and take an $s\ge 0$ such that $u\coloneqq \mel (s)$ is a point of continuity of $\mka_i$, that is, not a level of constancy of $\ell_i$. Since $\gra \mel_n (s) = \mel (s)$, the result proven above implies that $t_i(s) = \gra  \mka_{n,i} ( \mel_n (s)) = \mka_i (\mel (s))$ and thus establishes \eqref{eq:T_i_formula} in this case. 

It remains therefore to treat the case where $u\coloneqq \mel(s)$ is not a point of continuity of $\mka_i$, that is, we have $\mdu \mka_i(u)>0 $ ($i$ is still fixed).

In such a scenario, $u$ is a level of constancy of $\ell_i$ and thus not a level of constancy of $\ell_j, j \not =i$. It follows that $u$ is a point of continuity of $\mka_j, j \not =i$ and we conclude that    $ \mdu \mka(u)= \mdu \mka_i(u)$. This means that at $u$ both $\mka$ and $\mka_i$ have a jump of value, and the size of the jump is in both cases the same. 
For simplicity of notation, let us write simply $\mdelta $ instead of $ \mdu \mka(u)$ and  $\mdu \mka_i(u)$.

Since $\mka$ has a jump at $u$ of size $\mdelta$ and $\mel$ is the generalized inverse of $\mka$, $\mel$ remains constant  at the level $u$ in a time interval of length $\mdelta$; moreover to the left of this interval $\mel (v) $ is smaller than $s$, and larger than $s$ to the right of this interval. We record this fact as follows
\begin{equation}\label{fft:2} \mel (v)\begin{cases} <u, & v< v^\flat,\\
= u, & v\in [v^\flat,v^\flat +\mudelta], \\ 
>u, & v> v^\flat,\end{cases}
\end{equation} 
where $v^\flat \coloneqq \inf\{v\geq 0 \colon \mel(v)=u\}$.   
Similarly, since $\mka_i$ is the inverse of $\ell_i$,
\begin{equation}\label{fft:3} \ell_i (v) \begin{cases} <u, & v< v^\sharp,\\
= u, & v\in [v^\sharp,v^\sharp +\mudelta], \\ 
>u, & v> v^\sharp,\end{cases} \end{equation} 
where $v^\sharp \coloneqq \inf\{v\geq 0 \colon \ell_i (v)=u\}$. 


So prepared, let us consider a sequence $\jcg{v_n}$ of reals smaller than but converging to $v^\flat$ and chosen in such a way that each $v_n$ is a point of continuity of $\mka_i$. Then, by the already established part, 
\[ t_i(v_n) = \mka_i (\mel (v_n)), \qquad  n \ge 1. \] 
Since $t_i$ is continuous, the left-hand side above converges to $t_i(v^\flat)$. At the same time, 
since $\mel$ is continuous, $\gra \mel (v_n) = \mel (v^\flat) =u$ and so, since $\mka_i$ has left limits, the right-hand side converges to $\mka_i (u-)$,  which, by \eqref{fft:3} equals $v^\sharp$.  This shows that  $t_i(v^\flat)= v^\sharp$, that is, that $v^\flat$ is the time when $t_i$ reaches the level of constancy of $\ell_i$ visible in \eqref{fft:3}.  

Lemma \ref{det} (a) says thus that, in the interval $[v^\flat,v^\flat+\mudelta]$, $t_i$ grows linearly with slope $1$:
\[ t_i(v)=v^\sharp +v-v^\flat  \] 
(whereas the other $t_j$s stay constant). 
Next, since $\mel (s) =u$, \eqref{fft:2} shows that $s \in [v^\flat,v^\flat +\mudelta]$, and so the above formula applies to $v=s$. Formula \eqref{fft:2} shows also that $\mka (u) = v^\flat +\mudelta$. Thus, recalling that $v^\sharp = \mka_i(u-) = \mka_i(u) - \mudelta$, we finally see that
\begin{align*}
t_i(s) &= \mka_i(u) - \mudelta +s -v^\flat = \mka_i (u) + s - \mka (u) \\&= \mka_i \circ \ell_i (s) + s - \mka \circ \ell_i (s). \end{align*} 
This completes the proof. \end{proof}

As a corollary to this proposition, we see that the time-change processes $T_i, i \in \kad$ spoken of in Theorem \ref{thm:wbam} are given by a counterpart of our formula \eqref{eq:T_i_formula} (almost surely). This is simply because, as discussed in the proof of the theorem,  paths of processes 
$U_i^{-1}\mc \circ \mc L_i $ do not have common levels of constancy (almost surely).

\section{Synopsis} \label{sec:synopsis}

\subsection{The analytical part}
The analytical part of the paper can be summarized as follows. Let $\coss$ be the space of continuous functions on the natural $\ka$-point compactification $S$ of the infinite star graph with $\ka$ edges and center at $\zero$ (see Section \ref{sec:looa} for details). Also, let $\mathfrak C^2(S)$ be the space of twice continuously differentiable functions on $S$, and  $\Delta$  be the Laplace operator given by $\Delta f = \frac 12 f''$ on $\mathfrak C^2(S)$. Our point of departure is the fact that for any Feller generator, say, $\gen$,  in $\coss$ that is a restriction of $\Delta $, there are non-negative constants $\alpha, \beta_i,i\in \kad $, and $\gamma$, and a (possibly infinite) positive Borel measure $\mum$ on $S_{\zero} \coloneqq S\setminus \{\zero\}$ such that 
$\int_{S_{\zero}} (1\wedge x)\, \mum (\mud x)<\infty$, at least one of the conditions $\alpha + \sui \beta_i >0$ and $\mum (S) =\infty$ holds, and
\begin{equation}\label{syn:1} {\tfrac \alpha 2} f''(\zero) - \sum_{i\in \kad} \beta_i f'_i(0) + \gamma f(\zero) = \int_{S_{\zero}} [f(x) - f(\zero)]\, \mum (\mud x), \qquad f \in \dom{\gen }. \end{equation}	

In our search for a converse, we exclude the possibility for $\mum$ to have a mass at any of the points at infinity, because it allows the related process to jump from $\zero$ to one of those points never to come back to the regular state-space, and such a behavior is uninteresting. Under this additional assumption we prove that for any set of parameters $\alpha,\beta=\seq{\beta_i}$, $\gamma$ and $\mum$ described above, the operator $\Delta$ as restricted to the space of $f$ such that \eqref{syn:1} holds, is a Feller generator. 

As in the seminal paper of L.C.G. Rogers \cite{rogersik}, the resolvent of this generator can be expressed by means of the resolvent of the minimal Brownian motion and the parameters described above --- see \eqref{joj:1}. Formula \eqref{joj:1} reduces to \eqref{loo:4} when $\mum =0$ and has an alternative form \eqref{dow:1} when $\mum$ is nonzero but finite. Moreover, it informs of continuous dependence of the underlying processes on the quadruple of parameters. In the degenerate case where $\alpha + \sui \beta_i =0$ and yet $\mum (S) < \infty$, \eqref{joj:1} is not a Feller resolvent, but when restricted to a subspace of $\coss$ gives rise to a Feller-like semigroup; in the related process, however, the graph center does not belong to the state-space  --- see Section \ref{sec:degenerate}.  

\subsection{The stochastic part}
Stochastic analysis part commences with the description of the all-important Walsh's spider process, corresponding to the case where, in the boundary condition above,  $\alpha=\gamma=0$, $\mum$ is zero measure, and all $\beta_i$s are positive and add up to $1$. In Section \ref{sec:wbm} we recall two of this process' characterizations, due to Barlow, Pitman and Yor \cite{barlow}. Also, following the recent work by Bayraktar et al. \cite{bayraktar}, we describe the process by means of a multi-dimensional time-change.  

In Section \ref{sec:tcom=0} we show that by slowing down the Walsh's process at the point $\zero$, see \eqref{stickywalsh}, we obtain a new Feller process, the process generated by the operator with boundary condition \eqref{syn:1} featuring nonzero $\alpha $ (and still with $\gamma =0$ and $\mum=0$). Moreover, 
the so-obtained process can be further modified by requiring that it is killed when its local time at $\zero$ reaches the level of an independent exponential random variable, see \eqref{wkilled}. Such a modified process still enjoys the Feller property; in fact, it is generated by the Laplace operator with boundary condition \eqref{joj:1} where both $\alpha$ and $\gamma$ positive ($\mum$ is still zero). Interestingly, if proper care is exercised, the order in which we slow down and kill can be reversed without changing the end result.

In Section \ref{sec:finite} we turn to the case where $\mum$ is nonzero but finite, so that it is a scalar multiple of a probability measure $\mu$: $\mum =\delta \mu$. We show that the related process can be constructed as a concatenation of the processes of the type considered in Section \ref{sec:tcom=0}: when the process of the latter type is killed, the new process with certain probability restarts at a random point distributed according to the measure $\mu$. Theorem \ref{thm:wba} provides a description of the process corresponding to $\alpha =\gamma =0$ and $\mum = \delta \mu$ (with the standing assumption that $\beta_i$s are positive and add up to $1$) by means of (a) a single $\ka$-dimensional Brownian motion, (b) a family of subordinators build of $\beta_i$s and a~compound Poisson process, and (c) a multi-parameter time-change.  Notably, the compound Poisson process   is a function of $\mum$, for $\delta$ is its intensity and $\mu$ is the distribution of its jumps. In the limit case  where $\mum$ vanishes, Theorem \ref{thm:wba} reduces to the result of Bayraktar et al. cited above. The full case of nonzero but finite  $\mum$ is obtained by scaling parameters if necessary (see \eqref{marek:1}) and then slowing down and/or killing the process of Theorem \ref{thm:wba}. Degenerate scenarios are described in Section \ref{sec:degene}.

The main result of Section \ref{sec:stocha}, Theorem \ref{thm:wbam}, extends Theorem \ref{thm:wba} to the situation where $\mum $ is infinite. The crucial difference is that  in these new circumstances jumps of the underlying process are not modeled by means of a compound Poisson process but rather by a Poisson  point process   with mean measure $\ud t \times \mum$. Nevertheless, the theorem remains in nature the same as Theorem \ref{thm:wba}: it builds the searched-for process of a standard Brownian motion,  a family of subordinators and a multi-dimensional time-change. Theorem \ref{thm:wbam}, with formulae \eqref{rep1} and \eqref{mainrep}, can be thought of as the climax of the paper and of its stochastic part in particular. Interestingly, this multi-dimensional time-change can be given rather explicitly --- see formula \eqref{eq:T_i_formula}.

In this context, Theorem \ref{thm:resLoc} that provides an explicit form of the $\lam$-potential of the local time of the process with boundary condition \eqref{syn:1}, is rightfully viewed as a corollary to the main results of the analytic and stochastic parts of the paper. 
 \section{Appendix} 

In this appendix we prove two results taken for granted in the main body of the paper.  

\subsection{Existence of constants $\alpha,\beta_i,\gamma$ and measure $\mum$ such that \eqref{loo:0} holds}\label{bckon}

This section is devoted to the proof of the proposition stated in Section \ref{sec:bcft}, saying that for each Feller generator $\gen$ that is a restriction of $\Delta$ there are constants $\alpha,\beta_i, i\in \kad$ and $\gamma $ and a measure $\mum$ such that the domain of $\gen$ is characterized by relation \eqref{loo:0}.

To begin the proof we note that, by assumption, 
\begin{equation}\label{caod:1}  \pol f''(\zero)= \grato t^{-1} [T(t)f(\zero) - f(\zero)], \qquad f \in \dom{\gen}. \end{equation}
Moreover, $f \mapsto T(t)f(\zero)$ being a positive functional on $\coss$ of norm not exceeding one, there is a Borel sub-probability measure $\mathsf n_t$ such that $T(t)f(\zero)= \int_{S}f  \ud \mathsf n_t$ and thus we have
\[ \pol f''(\zero) =\grato \left[  t^{-1} \int_{S}[f (x) - f(\zero)] \mathsf n_t (\mud x) - a_t \, f(\zero)\right ], \] 
where 
\[ a_t \coloneqq t^{-1}(1-\mathsf n_t(S)) \ge 0, \qquad t >0. \]  
(We note that the measures $\mathsf n_t,t >0$ cannot be all equal to zero, because by assumption $\grato T(t)f(\zero)=f(\zero).$) Introducing 
\( g =\seq{g_i}$ with $g_i(x) \coloneqq \frac {f_i(x) - f(\zero)}{1-\e^{-x}}, x >0 \) and recalling that $\gen$ is a restriction of $\Delta$,
we see that this function  extends to a member of $\mathfrak C(\widetilde S)$ (see \eqref{tildees}) with $g_i(0)=f_i'(0), i \in \kad$. In terms of $g$, our relation becomes
\[ \pol f''(\zero) =\grato \left[  t^{-1} \sui \int_{[0,\infty]}g_i(x) (1-\e^{-x}) \mathsf n_{t,i} (\mud x) - \alpha_t \, f(\zero)\right ], \]
where $\mathsf n_{t,i}$ is the restriction of $\mathsf n_t$ to the $i$th edge, as identified with $[0,\infty]$; note that all integrands in the display above vanish at $x=0$.

Next, we turn our attention to the (non-zero) positive measures on $[0,\infty]$, defined by
\[ \nu_{t,i} (\mud x) \coloneqq t^{-1} (1-\e^{-x}) \mathsf n_{t,i}(\mud x), \qquad t >0, i\in \kad,\]
and write 
\begin{equation}\label{caod:2} \pol f''(\zero) =\grato \left[ \sui \|\nu_{t,i}\| \int_{[0,\infty]}g(x) \|\nu_{t,i}\|^{-1} \nu_{t,i} (\mud x) - a_t f(\zero)\right ], \end{equation}
where $\|\nu_{t,i}\|= \nu_{t,i}([0,\infty]), t >0, i\in \kad$.
By Helly's principle, there is a sequence $\jcg{t_n}$ of positive numbers converging to $0$ such that, for each $i\in \kad$,  the associated probability measures
\[ \|\nu_{t_n,i}\|^{-1} \nu_{t_n,i} \]
 converge weakly to a probability measure on $[0,\infty]$, say, $\nu_i$. 
 
By \eqref{caod:2}, it is amply clear that further analysis depends on convergence of sequences $\jcg{a_{t_n}}$ and $\jcg{\|\nu_{t_n,i}\|}, i \in \kad$. However, they are composed of non-negative numbers, and thus passing to subsequences if necessary, we obtain the existence of the possibly infinite limits 
 \[ a \coloneqq \gra a_{t_n}, \quad   b_i \coloneqq  \gra \|\nu_{t_n,i}\|,i\in \kad  \mqquad {and} c_i   \coloneqq  \gra \frac{a_{t_n}}{\|\nu_{t_n,i}\|}, i \in 
\kad,\]
and are naturally led to the following two cases. 

\textbf {Case 1. $a$ and $b_i$s are finite.} In this scenario, \eqref{caod:2} yields
\[ \pol f''(\zero) = \sui b_i \int_{[0,\infty]}g_i \ud \nu_i -   a f(\zero), \] 
that is
\[ \pol f''(\zero) = \sui b_i \nu_i (\{0\}) f'_i(0) +  \sui \int_{(0,\infty]} [f_i(x) - f_i(0)] \mum_i (\mud x) -    a f(\zero),\]
where $\mum_i(\mud x) \coloneqq  \frac{b_i \nu_i (\mud x)}{1-\e^{-x}}, i \in \kad$, so that $ \int_{(0,\infty]} [1 - \e^{-x}]\mum_i(\mud x)\le   b_i< \infty$. This is \eqref{loo:0} with 
\[ \alpha \coloneqq 1, \quad  \beta_i \coloneqq b_i \nu_i (\{0\}), i \in \kad, \quad  \gamma \coloneqq  a, \mquad{and} \mum \coloneqq \textstyle \sui \mum_i,\]
where $\mum_i$s are now seen as measures on $S$, that is, each $\mum_i$, by definition a measure on $[0,\infty]$, is identified with its image via the map $x \mapsto (i,x)\in S$; the latter is a measure on $S$ supported on the $i$th edge.

\textbf {Case 2. At least one of $a$ and $b_i$s  is infinite.} In this case, by passing to further subsequences if necessary, we can find an index $i_0\in \kad$ such that $d_i \coloneqq \gra \frac{\|\nu_{t_n,i}\|} {\|\nu_{t_n,i_0}\|}$ exists and is finite for all $i \in \kad$. Dividing \eqref{caod:2} (with $t$ replaced by $t_n$) by $a_{t_n} + \|\nu_{t_n,i_0}\|$, and letting $n\to \infty$, yields
\begin{equation}\label{caod:3} \sui \frac {d_i}{1+c_{i_0}} \int_{[0,\infty]}g_i \ud \nu_i = \frac {c_{i_0}}{1+c_{i_0}} f(\zero), \qquad f    \in \dom{\gen}, \end{equation}
provided that $c_{i_0} \not = \infty$. This again can be rewritten as 
\begin{equation*}
\sui  \frac {d_i  \nu_i (\{0\}) }{1+c_{i_0}} f'_i(0)+ \sui \int_{(0,\infty]} [f_i(x) - f_i(0)] \mum_i (\mud x) =   \frac {c_{i_0}}{1+c_{i_0}}  f(\zero),\end{equation*}
where this time  $\mum_i(\mud x) \coloneqq  \frac {d_i}{1+c_{i_0}}  \frac{ \nu_i (\mud x)}{1-\e^{-x}}, i \in \kad$. 
In other words, we obtain \eqref{loo:0} with 
\[\alpha \coloneqq 0, \quad \beta_i \coloneqq   \frac {d_i  \nu_i (\{0\}) }{1+c_{i_0}}, i \in \kad, \quad \gamma \coloneqq    \frac {c_{i_0}}{1+c_{i_0}},   \mquad{and}  \mum \coloneqq \textstyle \sui \mum_i. \]
Finally, we note that $c_{i_0}$ cannot be infinite. Indeed, otherwise \eqref{caod:3} would turn into 
$f(\zero)=0, f\in \dom{\gen}$.
This is clearly impossible, because $\gen$, being a generator, is densely defined.

We are hence left with explaining why  at least one of the conditions $\alpha + \sui \beta_i >0$ and $\mum (S) =\infty$ must hold.  To this end, assume that $\alpha + \sui \beta_i =0$ and yet $\mum (S) <\infty$. Then \eqref{loo:0} turns into 
\[ (\gamma + \delta) f(\zero) = \int_{S_0} f\ud \mum, \qquad f \in \dom{\gen}\]
where $\delta \coloneqq \mum (S_0)$. Now, if $\gamma +\delta >0$, this relation shows that $\dom{\gen}$ is contained in the kernel of the non-zero bounded linear functional $f \mapsto (\gamma +\delta) f(\zero) - \int_{S_0} f\ud \mum$, and thus $\gen$ cannot be a generator, since it is not densely defined; a~contradiction. Moreover, if $\gamma +\delta =0$, \eqref{loo:0} becomes $0=0$. This means that $\gen$ coincides with $\Delta$, and here also we have the desired contradiction because $\Delta$ cannot be a generator, its resolvent equation having too many solutions.

\subsection{The positive-maximum principle for $\Delta$ with boundary condition \eqref{loo:0}}\label{maks}

In this section, we prove that for any constants $\alpha,\beta_i,\gamma$ and a measure $\mum$ that satisfy the properties listed in Section \ref{sec:bcft}, the operator $\Delta$, as restricted to the domain where \eqref{loo:0} holds, satisfies the positive-maximum principle. 

First of all, if a positive maximum of an $f\in \dom{\gen}$ is attained at an $x\not = \zero$, the claim is obvious by the classical theorem of analysis. To treat the case where the maximum is attained at $\zero$,  in turn, we consider three sub-cases. 

{\bf (a)} Assume first that $\alpha \not =0$. Then \[ \gen f(\zero) = \pol f''(\zero) = \tfrac 1\alpha  (\sui \beta_i f'_i(0) - \gamma f(\zero) + \int_{S_\zero} [f(x) - f(\zero)] \ud \mum )\le 0 \] 
because  $f'_i(0) \le 0 $ for all $i\in \kad$, the coefficients are nonnegative, and $f(\zero)$ is both nonnegative and no smaller than any other value of $f$. 

{\bf (b)} Next, assume that $\alpha =0$, but at least one of $\beta_i$s is positive. Then, on the one hand, 
\[  \sum_{i\in \kad} \beta_i f'_i(0)  = \gamma f(\zero) + \int_{S_{\zero}} [f(\zero)-f(x)]\, \mum (\mud x)\ge 0, \]	
and, on the other, as before, $f'_i(0) \le 0 $ for all $i\in \kad$. It follows that for all $i$ such that $\beta_i>0$ we have $f_i'(0)=0$, and by assumption at least one such $i$ exists. 
But then the even extension of $f_i$ to the whole of $\R$ is twice continuously differentiable and, since its maximum is attained at $x=0$, this yields $f''_i(0)\le 0$. Moreover, by the definition of $\Delta$, 
$ f''_i(0)  $ does not depend on $i$ and coincides with $2 \gen f(\zero)$. This completes the proof in this case. 

{\bf (c)} Finally, assume that $\alpha + \sui\beta_i =0$ so that $\mum$ is necessarily infinite. In such a scenario, \eqref{loo:0} reduces to 
\[\gamma f(\zero) = \int_{S_{\zero}} [f(x) - f(\zero)]\, \mum (\mud x).\]
Since by assumption the left-hand side is nonnegative whereas the right-hand side is nonpositive, this relation shows that both expressions are zero, and in particular $f(x)$ equals $f(\zero)$ for all $x$ except on a set of  $\mum$ measure $0$.

Next, let $\mathcal C$ be the support of $\mum$, that is, let \[ \mathcal C \coloneqq \{x\in S_{\zero} \colon \text{for any ball $B$ containing $x$, $\mum (B)>0$}\}.\]  
We claim that 
\[ \updelta \coloneqq \inf_{x \in \mathcal C} d(x,\zero), \]
where $d$ is the metric in $S$ discussed in Section \ref{lumping2}, is zero.  Indeed, otherwise the open ball, say, $B_\updelta$, with center at $\zero$ and radius $\updelta$, is of $\mum$ measure $0$, and, moreover, since $\1_{B_\updelta^\complement} \coloneqq \1_{\{d(x,\zero)\ge \updelta\}}\le \frac 1{\min (\updelta,1)} (1\wedge d(x,\zero))$, 
\[ \mum (B_\updelta^\complement) \le \tfrac 1{\min (\updelta,1)}  \int_{S_\zero} 1\wedge d(x,\zero) \, \mum (\mud x) < \infty. \]
Since this contradicts our assumption that $\mum$ is infinite, the claim is established. 

As a result, there is a sequence $\jcg{x_n}$ of elements of $\mc C$ such that $\gra d(x_n,\zero) =0$. Choosing a subsequence if necessary, we can also assume that all $x_n$s belong to one edge, say the $i$th. But, for each $x\in \mathcal C$, $f(x) =f(\zero)$ as established above, and this implies that $f_i'(0) =0$. As in the case (b) we can thus argue that $f''_i(0) \le 0$ and so $\gen f(\zero) \le 0$, completing the proof. 

\newcommand{\oldzik}{
 we are thus in the scenario of Case 2 described above, for in Case 1, $\alpha =1$.  Let $\kal$ be the set of indices $i\in \kad$ such that $d_i >0$. By definition of $m_i$, $\mum =\sum_{i\in \kal} \mum_i$. Moreover, since all $\beta_i$s vanish, 
\eqref{caod:4}  takes the form 
\begin{equation}\label{caod:5} \sum_{i\in \kal} \int_{(0,\infty]} f_i \ud \mum_i =   \big (\sum_{i\in \kal}  \mum_i ((0,\infty]) + \frac {c_{i_0}}{1+c_{i_0}}\big ) f(\zero).\end{equation}}
   
\vspace{1cm} 

\bf Acknowledgment. \rm The authors would like to thank the Isaac Newton Institute for Mathematical Sciences, Cambridge, for support and hospitality during the programme \emph{Stochastic systems for anomalous diffusion}, where work on this paper was undertaken. This work was supported by EPSRC grant EP/Z000580/1. 

Furthermore, A. Pilipenko  thanks  the Swiss National Science
Foundation for partial support  of the paper (grants IZRIZ0\_226875, 200020\_214819, 200020\_200400, and 200020\_192129).
Finally, the authors are grateful to T.~Ichiba and  M. Portenko for fruitful discussions.

\def\cprime{$'$}\def\cprime{$'$}\def\cprime{$'$}\def\polhk#1{\setbox0=\hbox{#1}{\ooalign{\hidewidth \lower1.5ex\hbox{`}\hidewidth\crcr\unhbox0}}}\ifx \cedla \undefined \let \cedla = \c \fi\ifx \cyr \undefined \let \cyr = \relax \fi\ifx \cprime \undefined \def \cprime {$\mathsurround=0pt '$}\fi\ifx \prime \undefined \def \prime {'} \fi\def\polhk#1{\setbox0=\hbox{#1}{\ooalign{\hidewidth \lower1.5ex\hbox{`}\hidewidth\crcr\unhbox0}}}

 \end{document}